\documentclass[11pt]{article}

\usepackage{amsfonts,amssymb,amsopn,amsmath,mathrsfs,theorem,bbm}
\usepackage[round]{natbib}
\usepackage{hyperref}
\usepackage{graphicx}
\usepackage{verbatim}
\usepackage{authblk}
\usepackage{color}
\usepackage{enumitem}
\usepackage[font=small]{caption}
\usepackage{graphicx}
\usepackage{xcolor}

\newcounter{thm_counter}
\numberwithin{thm_counter}{subsection}
\numberwithin{figure}{subsection}

\newtheorem{lemma}[thm_counter]{Lemma}
\newtheorem{prop}[thm_counter]{Proposition}
\newtheorem{theorem}[thm_counter]{Theorem}

\theorembodyfont{\upshape}
\newtheorem{remark}[thm_counter]{Remark}
\newtheorem{defn}[thm_counter]{Definition}

\newenvironment{proof}{\vspace{1ex}\noindent{\textsc{Proof:}}\hspace{0.5em}}{\hfill\qed\vspace{1ex}}

\numberwithin{equation}{section} 

\DeclareMathOperator{\dist}{dist}

\DeclareMathOperator{\flow}{flow}
\DeclareMathOperator{\web}{web}
\DeclareMathOperator{\ram}{ram}

\newcommand{\noi}{\noindent}

\newcommand{\quand}{\quad\mbox{and}\quad}

\newcommand{\de}{\delta}

\newcommand{\eps}{\epsilon}

\newcommand{\sig}{\sigma}

\newcommand{\Ai}{{\cal A}}
\newcommand{\Bi}{{\cal B}}

\newcommand{\Ki}{{\cal K}}

\newcommand{\up}{\uparrow}
\newcommand{\down}{\downarrow}
\newcommand{\updo}{\updownarrow}
\newcommand{\sub}{\subset}

\newcommand{\ov}{\overline}

\def\ra{\Rightarrow} 
\def\to{\rightarrow} 
\def\iff{\Leftrightarrow}
\def\sw{\subseteq} 
\def\mc{\mathcal} 
\def\mb{\mathbb} 
\def\mf{\mathfrak}
\def\wh{\widehat}
\def\sc{\setminus} 
\def\sr{\stackrel}
\def\vec{\boldsymbol} 
\def\E{\mb{E}} 
\def\P{\mb{P}}
\def\R{\mb{R}} 
\def\N{\mb{N}}
\def\Q{\mb{Q}}
\def\~{\sim}
\def\-{\,;\,} 

\def\qed{$\blacksquare$}
\def\1{\mathbbm{1}}
\def\cadlag{c\`{a}dl\`{a}g}

\def\l{\left}
\def\r{\right}

\def\toas{\sr{\rm a.s}{\to}}

\def\tod{\sr{\rm d}{\to}}
\def\eqas{\sr{\rm a.s.}{=}}
\def\eqd{\sr{\rm d}{=}}

\def\s{\sigma}
\def\t{\tau}
\def\lhd{\vartriangleleft}
\def\nlhd{\ntriangleleft}

\def\Rc{\mathbb{R}^2_{\rm c}}
\def\Wb{\mc{W}_{\rm b}}
\def\Fb{\mc{F}_{\rm b}}
\def\mfs{\mf{s}}

\def\preceqd{\preceq_{\rm d}}
\def\precd{\prec_{\rm d}}

\fontsize{11pt}{16.0pt}
\selectfont

\newcommand{\rot}{\circlearrowright}
\newcommand{\Gr}{G}
\newcommand{\Hr}{H}

\newcommand{\nmr}[2]{\stackrel{\scriptscriptstyle\rm{#1}}{#2}}

\begin{document}

\makeatletter\@addtoreset{equation}{section}
\makeatother\def\theequation{\thesection.\arabic{equation}} 


\allowdisplaybreaks

\author[1]{Nic Freeman\thanks{n.p.freeman@sheffield.ac.uk}}
\author[2]{Jan Swart\thanks{swart@utia.cas.cz}}
\affil[1]{School of Mathematics and Statistics, University of Sheffield}
\affil[2]{The Czech Academy of Sciences, Institute of Information Theory and Automation.}

\title{Weaves, webs and flows}

\date{\today}
\maketitle

\begin{abstract}
We introduce \textit{weaves}, which are random sets of non-crossing {\cadlag} paths that cover 
space-time $\ov{\R}\times\ov{\R}$. 
The Brownian web is one example of a weave,
but a key feature of our work is that we do 
not assume that particle motions have any particular distribution.
Rather, we present a general theory of the structure, characterization and weak convergence of weaves.

We show that the space of weaves has a particularly appealing geometry, 
involving a partition into equivalence classes under which 
each equivalence class contains a pair of distinguished objects known as a \textit{web} and a \textit{flow}. 
Webs are natural generalizations of the Brownian web 
and the 
flows provide a pathwise representations of stochastic flows. 
Moreover, there is a natural partial order on the space of weaves,
characterizing the efficiency with which paths cover space-time, 
under which webs are precisely minimal weaves
and flows are precisely maximal weaves. 
This structure 
is key to establishing weak convergence criteria for general weaves,
based on weak convergence of finite collections of particle motions.
\end{abstract}

\noi
{\it Acknowledgement.}  Work sponsored by GA\v{C}R grant 22-12790S. 
This work was initiated during the program Genealogies of
Interacting Particle Systems at the IMS, NUS, Singapore.

\thispagestyle{empty}
\newpage

\setcounter{tocdepth}{2}
\tableofcontents

\thispagestyle{empty}
\newpage
\setcounter{page}{1}

\section{Introduction}\label{sec:intro}

In this article we introduce a rich and natural class of objects that generalize the Brownian web.
We call these objects \textit{weaves}.
Informally, a weave is a random set of non-crossing {\cadlag} paths,
such that each point of space-time is almost surely touched by at least one path.
The paths take values in $\ov{\R}$ and each path runs until time $+\infty$,
but paths may begin at any point of space-time.
It is important to note what is missing:
we do \textit{not} require that the paths follow any particular distribution
(in the example of the Brownian web, they follow coalescing Brownian motions).

We will establish a framework for weak convergence (i.e.~in law) of general weaves,
akin to the modern theory of weak convergence for real valued stochastic processes.
As the example of the Brownian web shows,
individual weaves may display a rich internal geometry.
We will see that
\textit{space of} weaves
also has an interesting structure in its own right.
This structure has major implications for the characterization of weaves,
thus also for weak convergence.

We use the term \textit{half-infinite}
for paths that, after beginning anywhere within space-time, continue until time $+\infty$.
If such a path begins at time $-\infty$ then it is said to be \textit{bi-infinite}.
Weaves consisting exclusively of bi-infinite paths provide natural pathwise representations of 
(sufficiently regular) stochastic flows,
but can also represent more complicated structures of branching-coalescing paths.
We refer to weaves of bi-infinite paths as \textit{flows},
although formally we will first give a different definition 
and later show equivalence to this.

Stochastic flows have been studied for many decades,
as detailed in the book of \cite{Kunita1997}.
It is remarkable that, despite their long history, 
stochastic flows have struggled to give rise to a viable theory of their own weak convergence.
A key problem is that
stochastic flows have traditionally been given a `pointwise' representation,
where for each pair of times $-\infty<s<t<\infty$ a random function $X_{s,t}:\R\to\R$ represents 
the movement of particles during $[s,t]$.
More precisely, $X_{s,t}(x)$ denotes the position at time $t$ of the particle that, at time $s$, was at location $x$.
The concept of a flow is therefore encapsulated by the 
consistency condition $X_{s,t}\eqas X_{s,u}\circ X_{u,t}$, required at deterministic times.
This representation is analogous to the old-fashioned representation of a real valued stochastic process 
as an infinite family of random variables $(X_t)_{t\geq 0}$,
where $X_t\in\R$ denotes the position of the particle at time $t\geq 0$.

The modern perspective is to view a stochastic process as a single random variable,
whose value is a random path.
Such a representation is known as a `pathwise' representation.
\cite{Skorohod1956} introduced a suitable state space $\mc{D}$,
whose elements are {\cadlag} paths,
and the resulting theory is detailed within the now ubiquitous texts
of \cite{EthierKurtz1986} and \cite{Billingsley1995}.
From an analytical point of view
it is far more convenient to work with convergence of one random {\cadlag} path,
than with convergence of infinitely many $\R$ valued random variables.
To abstract this principle a little further,
it is better to define a single random variable within a highly structured state space,
than to work with infinitely many `smaller' random variables in a more straightforward state space.



The same principle will apply to random sets of {\cadlag} paths, however
such objects have not yet made an analogous transition --
with the exception of the Brownian web.
The present article seeks to remedy this situation.
The Brownian web is a pathwise representation of the stochastic flow of \cite{Arratia1979},
in which particles perform independent Brownian motions until they meet,
after which particles remain coalesced for all remaining time.
Loosely, one such particle begins at each point of space-time.

The modern study of the Brownian web began with \cite{TothWerner1998},
who were first to understand its rich internal structure.
Based on this work,
\cite{FontesIsopiEtAl2004a}
represented the Brownian web as a 
(single) random variable whose value is a random set of continuous paths,
and introduced the term \emph{Brownian web}.
In this representation they gave the first 
conditions for weak convergence to the Brownian web,
based on the forwards-in-time motions of finite sets of particles.
A large body of literature has since emerged,
leading to the refined criteria available in the survey of \cite{SchertzerSunEtAl2017}.
Close relatives of the Brownian web have been investigated in similar style
and the Brownian web is understood to be the scaling limit
of a large and diverse universality class.

The key to this success has been the 
availability of good criteria for characterization and weak convergence.
Such criteria must strike a careful balance:
a type of convergence that preserves less information
is often easier to prove, and is more often true, but is also less meaningful.
One possible approach, 
used by \cite{BerestyckiEtAl2015} and \cite{CannizzaroHairer2021}
for the case of continuous coalescing paths,
is to map sets of paths to other objects
(respectively, to sets of `tubes' and real trees)
in order to induce a topology that may be used 
as a basis for weak convergence.

In the present work we 
handle sets of {\cadlag} paths directly,
in the style that 
has become popular within the literature of the Brownian web.
We give criteria for characterization and convergence of general weaves,
with no requirement that the particle motions follow any particular distribution.
We must also introduce a suitable state space;
a version of Skorohod's space $\mc{D}$ suitable for random sets of {\cadlag} paths
begun at arbitrary points of space-time.
The state space constructed 
by \cite{FontesIsopiEtAl2004a}
is a subset of our own, with matching induced subspace topology.

Let us now briefly comment on the significance of webs.
Our exploration of the space of weaves will uncover a natural partition into equivalence classes.
Each equivalence class features two distinguished elements,
one of which is a flow (as discussed above)
and the other of which we will refer to as a \textit{web}.
We will see that the property of being a web is equivalent to what remains 
if one takes the usual definition of the Brownian web 
and \textit{removes} the requirement that the particle motions have a particular distribution.
Webs and flows are in bijective correspondence;
moreover they are the extremal points, respectively minima and maxima, 
within a structure that we will shortly describe.

Within much of the literature on the Brownian web,
the proofs 
rely heavily on the distribution of coalescing Brownian motions.
Consequently our own arguments have little in common.
Despite this, we remark that what is known about the Brownian web has been invaluable in writing the present article,
and the Brownian web is a canonical example of a weave.
In fact the majority of our results are new even in the special case of the Brownian web.

\subsection{Outline of results}
\label{sec:results_outline}

In Section \ref{sec:results}
we will introduce our state space and, following that, 
give rigorous statements of our main results.
Setting up the state space
requires some significant work,
so we will give here a non-rigorous presentation of our main results
and the ideas that led to them.

We require that {\cadlag} paths are allowed to jump \textit{at their initial times}.
Naturally, this requires some supporting structure,
which we delay for now and appeal instead to the readers intuition.
Our concept of a {\cadlag} path is precisely equivalent to the classical {\cadlag} path 
$f:[t,\infty]\to\ov{\R}$ that is right-continuous with left limits,
plus a possible jump at the initial time $t$.
See Figure \ref{fig:example_weaves} for an example showing why this augmentation is necessary.

\begin{figure}[t]
\centering
\includegraphics[scale=1.25]{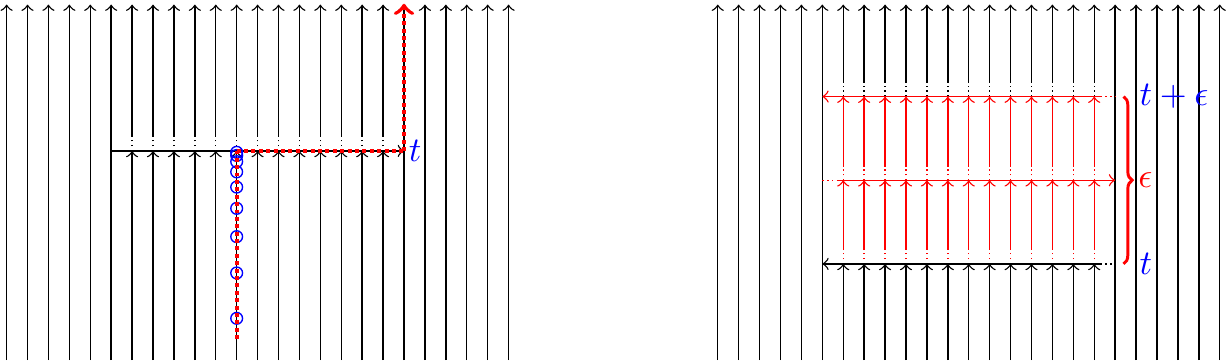}   
\caption{
In both figures, time runs upwards and the spatial axis is horizontal.
In each figure a weave is depicted via solid lines,
with the corresponding flow depicted via including dotted lines.\\
\textit{On the left:}
A weave $\mc{A}$, featuring a {\cadlag} path jumping at its initial time.
Order the blue circles from bottom to top.
Consider the particle motion $f_n$ starting within the $n^{\rm th}$ blue circle,
which then follows the red dotted line.
The limiting path $f$ is a trajectory that jumps rightwards at its initial time $t$.
We require that weaves are closed sets 
and we require that weaves have particle motions;
consequently we require that the path $f$ exists.
\\
\textit{On the right:} 
A warning example related to Theorems \ref{t:flow_conv} and \ref{t:weave_conv}. 
A weave $\mc{A}_\eps$ is depicted, along with the corresponding flow $\mc{F}_\eps$
of bi-infinite paths that do not cross $\mc{A}_{\eps}$.
Space-time points within the horizontal arrows, and forwards in time continuations thereof, are ramified.
In the limit as $\eps\to 0$ the red area vanishes;
the weaves $\mc{A}_\eps$ converge to a pervasive system of paths that contains crossing (jumping in both directions at $t$);
the sequence of flows $(\mc{F}_\eps)_{\eps>0}$ are not relatively compact (due to paths that jump left-right-left between $t$ and $t+\eps$);
whilst the $m$-particle motions $\mc{A}_\eps|_{\vec{z}}=\mc{F}_\eps|_{\vec{z}}$ from finite sets $\vec{z}$ of non-ramified points converge to 
those of a weave (which contains only the leftwards jump at $t$).
}
\label{fig:example_weaves}
\end{figure}

The theory of weak convergence of real valued stochastic processes is normally presented in Skorohod's J1 topology.
We require the (slightly coarser) Skorohod M1 topology.
The reasons for this are rather technical, 
but roughly speaking our use of non-crossing paths makes it natural to consider
jumps as part of the path, rather than as an empty region of space that the path jumps over.
The former perspective corresponds to Skorohod's M1 topology, the latter to J1.
An example of {\cadlag} paths 
such that $f_n\to f$ in M1 but not in J1 is
$f_n(t)= 0 \vee nt \wedge 1$
with limit
$f(t)=\1_{\{t\geq 0\}}$, both defined for all $t\in\ov{\R}$.

A key insight from the Brownian web is that we should consider random \emph{compact} sets of paths;
we do so within a suitable version of the M1 topology.
We say that a set of {\cadlag} paths is \textit{pervasive}
if each space-time point $z=(x,t)$ is contained within at least one path,
including jumps.

A central concept is the \textit{$m$-particle motion} of a weave.
Loosely, if we choose a point $z=(x,t)$ in space-time,
we may place a particle at the point within the weave and then watch how it moves, 
forwards in time.
For most deterministic points of space-time (in fact, Lebesgue almost all)
this operation is well defined and an almost surely unique forwards in time motion exists.
This motion is a single random {\cadlag} path
with initial time $t$.
If we do the same for $m\in\N$ space-time points at once,
then we obtain the $m$-particle motion of the weave.

It is clear a priori that a pair of {\cadlag} paths might cross each other.
The meaning is clear for continuous paths and, for now, we appeal to the readers intuition.
When we come to define crossing rigorously some clarification will be required,
to handle cases where {\cadlag} paths jump over each other at their initial times.
We say that a set of paths is \textit{non-crossing} if none of its elements cross each other.

We are now in a position to describe our main results concerning weaves.
Formally, 
a \textit{weave} is a probability measure on M1-compact sets of half-infinite {\cadlag} paths,
that is almost surely non-crossing and pervasive.
We also use the term weave for a random variable with such a law.
We adopt the convention of using caligraphic letters, such as $\mc{A}$ and $\mc{B}$, for weaves.
We also remind the reader that within a general partial order,
a typical element might 
sit below anything from none to infinitely many maxima;
similarly for minina.

\begin{enumerate}
\item 
There exists a natural partial order $\preceqd$ on the space of weaves.
Informally, the statement $\mc{A}\preceqd\mc{B}$ means:
there exists a coupling under which
$\mc{B}$ covers space-time more efficiently than $\mc{A}$
i.e.~with fewer and/or longer paths.

By definition, we say that a weave is a \textit{web} if it is minimal (within the space of all weaves) with respect to $\preceqd$.
We say that a weave is a \textit{flow} if it is maximal.

\item
The space of weaves is partitioned into equivalence classes,
each of which has a flow as its unique maximal element and a web as its unique minimal element.
We write this equivalence relation as $\mc{A}\sim\mc{B}$.
Elements within the same equivalence class need not be $\preceqd$-comparable.

\item
There exists a pair of deterministic functions $\web(\cdot)$ and $\flow(\cdot)$
with the following properties.
\begin{enumerate}
\item A weave $\mc{A}$ is a web if and only if $\web(\mc{A}) \eqas \mc{A}$.
\item A weave $\mc{A}$ is a flow if and only if $\flow(\mc{A}) \eqas \mc{A}$.
\end{enumerate}

Moreover, a weave is a flow if and only if it comprises exclusively of bi-infinite paths. 
Therefore, flows are natural pathwise representations of stochastic flows.

The web operation is a slight generalization of the operator $\mc{W}\mapsto \ov{\mc{W}(D)}$
that is familiar within the standard characterization of the Brownian web.
By definition, $\flow(\mc{A})$
is the set of bi-infinite {\cadlag} paths that do not cross $\mc{A}$.
The map $\flow(\cdot)$ is continuous, but $\web(\cdot)$ is not.

\item Two weaves $\mc{A}$ and $\mc{B}$ satisfy $\mc{A}\sim\mc{B}$
if and only if the $m$-particle motions of $\mc{A}$ and $\mc{B}$ have the same distribution.

\item A weak limit of flows is necessarily a flow.
Moreover, 
for flows,
weak convergence is equivalent to tightness plus weak convergence of the $m$-particle motions.

An analogous result holds for general weaves, at the level of equivalence classes.
Here we must include the assumption that weak limit points are non-crossing.

\item
Each web $\mc{W}$ has an associated dual web $\wh{\mc{W}}$, 
of {\cadlag} paths running backwards in time,
such that $\mc{W}$ and $\wh{\mc{W}}$ are almost surely non-crossing.

The triplet $(\mc{W},\wh{\mc{W}},\mc{F})$ may be reconstructed from any single one of $\mc{W},\wh{\mc{W}}$ and $\mc{F}$.
If any one of these three consists exclusively of continuous paths, then they all do.
\end{enumerate}

Underpinning all of these results is a delicate operation that takes a half-infinite path within a weave
and extends it, backwards in time, into a bi-infinite path, without inducing crossing
and preserving {\cadlag}ness.
Moreover such extension may be done to all paths within a weave, 
without breaking the compactness, 
to obtain its corresponding flow.
Note that weaves are by definition closed sets, 
so this operation does \textit{not} involve taking a limit of suitable paths within the weave.
Let us briefly describe what it does involve.

There is a partial order $\sw$ on (individual) {\cadlag} paths,
corresponding to the idea that $f\sw g$ if and only if 
the path $f$ may be extended, forwards and/or backwards in time, to give $g$.
Paths within weaves run until time $+\infty$, 
so for weaves only extension backwards in time is relevant.
For a given weave $\mc{A}$, let $\mc{A}_{\max}$ denote the set of maximal elements of $(\mc{A},\sw)$.

It turns out that there is a natural bijection between
Dedekind cuts of $\mc{A}_{\max}$ 
and 
bi-infinite paths that do not cross $\mc{A}$.
This relationship is reminiscent of Dedekind's famous construction of $\R$ from $\Q$,
but in our case the operation that connects $\mc{A}_{\max}$ to $\flow(\mc{A})$
is not a topological closure.
Loosely, we may take a bi-infinite path $h$ that does not cross $\mc{A}$,
and the corresponding Dedekind cut is all paths $f\in\mc{A}_{\max}$ that lie strictly to the left of $h$.
The inverse function of this correspondence is more complicated to define 
and we do not attempt a description at this point.
To extend half-infinite paths backwards in time,
we note that a Dedekind cut may be constructed in the same way from 
any {\cadlag} path (not necessarily bi-infinite) that does not cross $\mc{A}$, 
and then use that inverse function to produce a corresponding bi-infinite path.

The proof of this relationship between half-infinite and bi-infinite paths
relies on delicate analysis.
It requires a formulation of {\cadlag} paths
where potential jumps at the initial time are an integral part of the path,
rather than an afterthought to the otherwise classical definition.
We introduce such a formulation in Section \ref{sec:Rpm},
followed by a description of our state space in Section \ref{sec:pi}.
We then introduce key notation concerning crossing and ordering of paths in Section \ref{sec:terminology},
at which point we are able to give a rigorous presentation of our main results in Section \ref{sec:results_weaves}.
We discuss connections to the Brownian web in Section \ref{sec:bw},
and connections to related state spaces in Sections \ref{sec:related_topologies}.

The proofs appear in Sections \ref{sec:technical_1}-\ref{sec:weaves_random}.
In Section \ref{sec:technical_1} we set up machinery to work with general {\cadlag} paths,
and with jumps at their initial times.
Section \ref{sec:weaves_det} treats the internal structure of weaves
and the implications thereof for the space of weaves,
which is best studied (initially) in a deterministic context.
This includes the key result on path extension.
In Section \ref{sec:weaves_random} we give the proofs of our main results.

\section{Results}
\label{sec:results}

\subsection{The split real line}
\label{sec:Rpm}

A function is said to be \textit{{\cadlag}}, 
from the French `continue \`a droit, limite \`a gauche', 
if it is right-continuous with left limits.
\cite{Kolmogorov1956} observed that a real \cadlag\ function together with its left-continuous modification 
can be viewed as a continuous function on a peculiar topological space, introduced by \cite{AU29}. 
This will provide an elegant formulation of our results, 
as well as being a necessary component of more technical proofs.
We give here a brief introduction to this space.

Let $\ov\R:=[-\infty,\infty]$ denote the extended real line. 
By definition, for any subset $I\sw\ov\R$, 
we let 
$$I_\mf{s}
=\big\{(t,\star)\-t\in I\text{ and }\star\in\{-,+\}\big\}.$$
We will almost always write $t\star$ in place of the formal notation $(t,\star)$.
We call $\R_\mfs$ the \emph{split real line} and $\ov\R_\mfs$ the \emph{extended split real line}. 
Loosely,
to construct $\R_\mfs$ from $\R$,
each $t\in\R$ has been split into two parts, 
a left part $t-$ and a right part $t+$.
We equip $\ov\R_\mfs$ with the lexicographic order, 
from left to right,
that is $t_1\star_1<t_2\star_2$ if and only if either $t_1<t_2$ or both $t_1=t_2$ and $\star_1=-$, $\star_2=+$.
We use notation for intervals in $\ov\R_\mfs$ 
similar to the usual notation for the extended real line:
\begin{align}
(t_1\star_1,t_2\star_2)
&=\{t\star\in\ov\R_\mfs:t_1\star_1<t\star<t_2\star_2\}\\
[t_1\star_1,t_2\star_2]
&=\{t\star\in\ov\R_\mfs:t_1\star_1\leq t\star\leq t_2\star_2\},
\end{align}
and analogously for half-open intervals such as $(t_1\star_,t_2\star_2]$
or $[t_1\star_1,t_2\star_2)$. Note that there is some redundancy in this
notation since, for example, $(s-,t+)=[s+,t-]$.
We say that a set $A\sw\R_\mfs$ is \emph{bounded} if
$A\sw[-T,T]_{\mfs}$ for some $T<\infty$.

We equip $\R_\mfs$ and $\ov\R_\mfs$ with the \emph{order topology}. 
Recall that, in a 
totally ordered space $(S,<)$, 
the order topology is generated by the open intervals $(a,b)=\{x\in S \-a<x<b\}$ where $a<b$.
The order topology on $\R$ thus coincides with the usual Euclidean topology.
The following lemma records all that we need to know about the order topology on $\R_\mfs$.
Parts 1 and 4 appear respectively as Lemma 2.1 and Proposition 2.3 in \cite{FreemanSwart2023}.
Parts 2 and 3 are straightforward consequences of part 1.

\begin{lemma}
\label{l:Rpm}
The following hold.
\begin{enumerate}
\item 
A sequence $t_n\star_n$ converges to the limit $t+$
(resp.\ $t-$) if and only if $t_n\to t$ in $\R$ and $t_n\star_n\geq t+$
(resp.\ $t_n\star_n\leq t-$) for all $n$ sufficiently large.

\item 
Intervals of the form $(t_1\star_1,t_2\star_2)$ are open and intervals
of the form $[t_1\star_1,t_2\star_2]$ are closed. 

\item 
The quotient of $\R_\mfs$ by the relation $s\star \stackrel{_\cdot}{\sim} t\star\iff s=t$ is homeomorphic to $\R$.

\item
The space $\R_\mfs$ is a Hausdorff topological space. It is
separable 
but not metrisable.
For $C\sw\R_\mfs$, the following three statements are equivalent: 
(i) $C$ is compact; 
(ii) $C$ is sequentially compact;  
(iii) $C$ is closed and bounded.
\end{enumerate}
\end{lemma}

The topology on $\ov\R_\mfs$ permits an elegant description of {\cadlag} functions on $\ov\R$,
developed in Section 2.2 of \cite{FreemanSwart2023}.
We require this characterization only for the case of real {\cadlag} functions on closed intervals, as follows.
Let $I=[a,b]\sub\ov\R$ be a closed interval and
let $f:I_\mfs\to \ov{\R}$ be a function.
The following statements are equivalent:
\begin{enumerate}
\item $f$ is continuous with respect to the topology on $I_\mfs$, as a subset of $\R_\mfs$;
\item the function $t\mapsto f(t+)$ defined from $[a,b]\mapsto\ov{\R}$ is right continuous with left limits,
and $t\mapsto f(t-)$ defined from $(a,b]\to\R$ is its left continuous modification;
\item the function $t\mapsto f(t-)$ defined from $[a,b]\mapsto\ov{\R}$ is left continuous with right limits,
and $t\mapsto f(t-)$ defined from $[a,b)\to\R$ is its right continuous modification.
\end{enumerate}
\begin{defn}
\label{d:cadlag}
We refer to a function $f:[a,b]_\mfs\to\ov{\R}$ satisfying (any of) these criteria as a \textit{{\cadlag} path}.
\end{defn}
Definition \ref{d:cadlag}
is a minor extension of the classical
notion of a {\cadlag} function on $[a,b]\sw\ov\R$.
Specifically, we attach a formal meaning and value to the `left limit' at $a-$,
which is absent in the classical definition.
It may take any value, which is to say that the value of
$f(a-)$ is not restricted
by the values of $f(t\star)$ for $t\star>a-$.
This introduces the possibility that $f(a-)\neq f(a+)$,
corresponding to a jump at the initial time.

Given a {\cadlag} path $f$ with domain $[a,b]_\mfs$ we write
$I(f)=[a,b]\sw\ov\R$ and $I(f)_\mfs=[a,b]_\mfs\sw\ov\R_\mfs$.
We write $\s_f=a$ for the \emph{initial time} and $\t_f=b$ for the \emph{final time} of $f$.
Similarly, we call $(f(\s_f-),\s_f)$ and $(f(\t_f+),\t_f)$ respectively the \textit{intial} and \textit{final points} of $f$.
We say that $f$ \textit{begins} and its initial point and \textit{ends} at its final point.

We say that $f$ makes a \textit{jump} at $t\in [\sigma_f,\tau_f]$ if $f(t-)\neq f(t+)$.
The jump is said to be \textit{to the left} if $f(t+)<f(t-)$ and \textit{to the right} if $f(t+)>f(t-)$.
As mentioned, {\cadlag} paths may jump at their initial and final times or at any time in between.
The number of such jumps is at most countable.
If $f(t-)=f(t+)$ then we say that $f$ is \textit{continuous} at $t$, in which case (and only in this case) we write $f(t)=f(t-)=f(t+)$.

\subsection{The path space}
\label{sec:pi}

In this section we introduce the space $\Pi$, 
whose elements are {\cadlag} paths defined on closed intervals of $\ov\R_\mfs$,
and the space $\mc{K}(\Pi)$,
whose elements are compact subsets of $\Pi$.
We will refer to \cite{FreemanSwart2023}
which considers a more general setup 
but also acts as a companion paper providing the topological basis for the present article.
Let
\begin{equation}
\Pi=\l\{f:I_\mfs\to\ov\R\-\text{$f$ is a {\cadlag} path and $I\sw\ov\R$ is a closed interval}\r\}.
\end{equation}
We regard two elements $f,g\in\Pi$ as equivalent if
they have the same values outside of times $\pm\infty$.
Formally,
define the equivalence relation 
\begin{equation}
\label{eq:pi_compactification}
f\stackrel{\Pi}{\sim} g 
\quad\iff\quad 
I(f)=I(g)\text{ and }f(t\pm)=g(t\pm)\text{ for all }t\in I(f)\cap\R,
\end{equation}
and work implicitly with the resulting equivalence classes of $\Pi$.
We abuse notation slightly by continuing to write $f\in\Pi$ for a {\cadlag} path, 
but including the notational convention%
\footnote{
To be precise: \cite{FreemanSwart2023} uses a slightly different compactification procedure to \eqref{eq:pi_compactification},
which allows the domain of {\cadlag} paths to be non-interval sets
and defines $\ov\R_\mfs$ to be a two-point compactification of $\R_\mfs$.
See Section 2.1 
within \cite{FreemanSwart2023} for details.
The space $\Pi$ from the present article is denoted there by $\Pi^|$.
}
that $f(t\star)=*$ whenever $t=\pm\infty$.

Our main results require Skorohod's M1 topology on $\Pi$, 
which we now introduce.
We will introduce the J1 topology at the same time,
as it is more widely used and the reader may wish to make a comparison.
We define the \emph{closed graph} $\Gr(f)$ 
and
\emph{interpolated graph} $\Hr(f)$ of a {\cadlag} path $f\in\Pi$ as
\begin{align*}
\Gr(f)&=\big\{(x,t)\in\ov\R^2\-
t\in I(f),\ x\in\{f(t-),f(t+)\}\big\},\\
\Hr(f)&=\big\{(x,t)\in\ov\R^2\-
t\in I(f),\ x\in[f(t-),f(t+)]\big\},
\end{align*}
where we use the convention $[s,t]:=[s\wedge t,s\vee t]$  for $s,t\in\ov\R$. 
See Figure~\ref{fig:Rc} for a picture displaying 
the difference between $\Gr(f)$ and $\Hr(f)$:
at times $t\in\R$ when the path makes a jump,
the line segments between $(f(t-),t)$ and $(f(t+),t)$ are drawn in $H(f)$ but not in $G(f)$.

\begin{figure}[t] 
\centering
\includegraphics{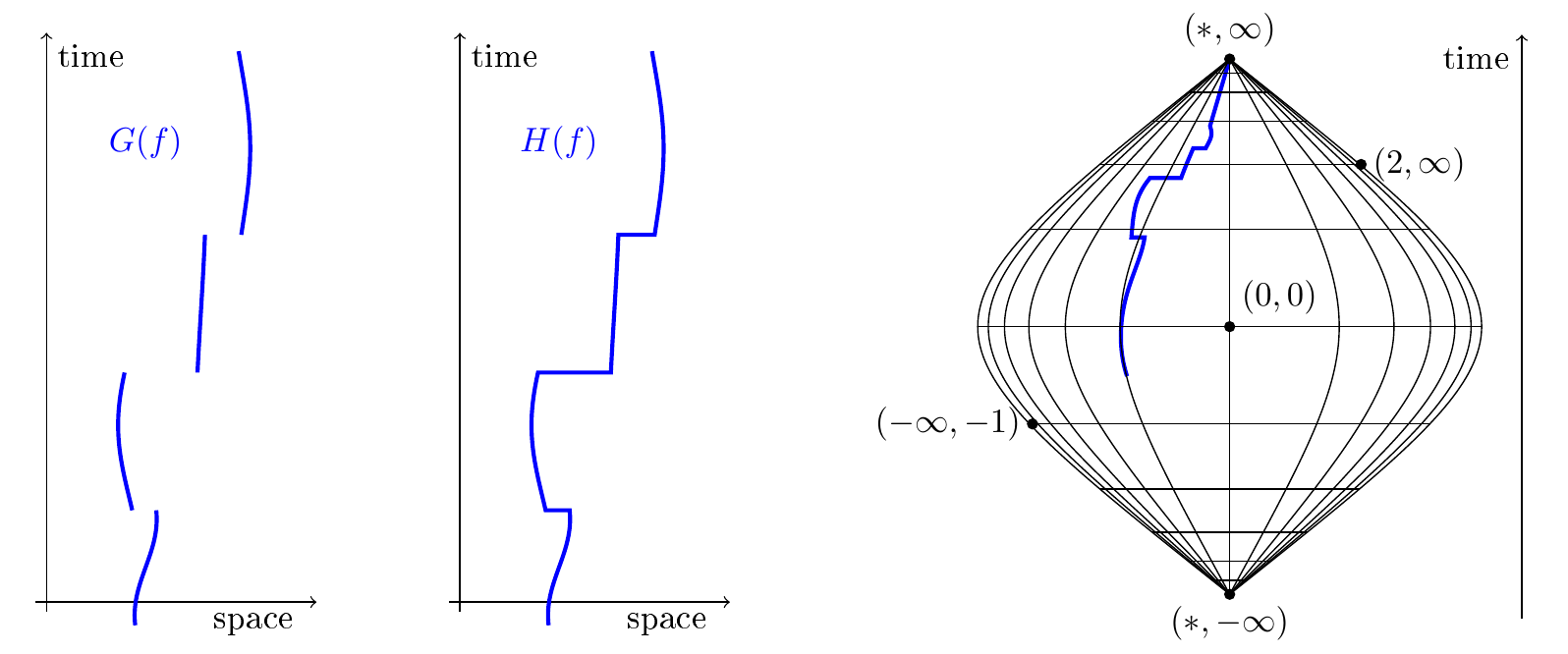}   
\caption{ 
On the left, the closed graph $\Gr(f)$ and interpolated graph $\Hr(f)$ of a path $f$. 
On the right, the compactification $\R^2_{\rm c}$ of
$\ov\R\times\R$ with the interpolated graph of $f$,
with some points marked for convenience.
In \cite{FreemanSwart2023} $\R^2_c$ is referred to as a \textit{squeezed space}.
}
\label{fig:Rc}
\end{figure}

The reason for \eqref{eq:pi_compactification} is that 
we intend to treat $G(f)$ and $H(f)$ as compact subsets of a suitable space,
which in turn will allow us to describe the J1 and M1 topologies. 
With this in mind, we define
\begin{equation*}
\Rc:=\l(\ov\R\times\R\r)\cup\big\{(\ast,-\infty),(\ast,\infty)\big\}
\end{equation*}
and equip $\Rc$ with a metrisable topology such that the induced subspace topology on $\ov\R\times\R$
is the product topology and, as $n\to\infty$,
\begin{equation*}
(x_n,t_n)\to(\ast,\pm\infty) 
\quad\mbox{if and only if}\quad 
t_n\to\pm\infty.
\end{equation*}
Equation (2.26) of \cite{FreemanSwart2023} gives an explicit metric $d_{\R^2_c}$ with these properties,
under which $\R^2_{\rm c}$ is compact
See Figure \ref{fig:Rc} for an illustration of $G(f)$, $H(f)$ and the compactification.
We endow $\Rc$ with the two-dimensional Lebesgue measure on $\ov{\R}\times\R\sw\Rc$, 
placing zero mass at $(*,\pm\infty)$.

Let $f\in\Pi$. 
There is a natural total order on both $G(f)$ and $H(f)$, which we denote by $\sqsubseteq$.
For $G(f)$ this order is given by $(f(t_1\star_1),t_1)\sqsubseteq (f(t_2\star_2),t_2)$ whenever $t_1\star_1\leq t_2\star_2$.
For $H(f)$ it requires a little more care:
we say that $(x_1,t_1)\sqsubseteq (x_2,t_2)$ whenever
$t_1<t_2$, 
or if
$t_1=t_2$ and $x_1$ is non-strictly closer to $f(t_1-)$ than $x_2$.

Informally, the J1 topology on $\Pi$ corresponds to convergence of closed graphs,
and the M1 topology to convergence of interpolated graphs, 
with the caveat that (in both cases) the total order $\sqsubseteq$ is preserved by the convergence.
Our next step is to formalize this intuition and give a definition of the two topologies.

For a metric space $(M,d_M)$, 
let $\Ki(M)$ denote the space of all nonempty compact subsets of $M$, equipped with the Hausdorff metric induced by $d_M$.
See Appendix \ref{a:hausdroff_metric} for a brief introduction to the Hausdorff metric.
We define the
\emph{second order closed graph} $G^{(2)}(f)$ and the 
\emph{second order interpolated graph} $H^{(2)}(f)$ 
of a path $f\in\Pi$ to be
\begin{align*}
G^{(2)}(f)&=
\big\{(z_1,z_2)\-
z_i\in G(f)\text{ and }z_1\sqsubseteq z_{2}\big\},\\
H^{(2)}(f)&=
\big\{(z_1,z_2)\-
z_i\in H(f)\text{ and }z_1\sqsubseteq z_{2}\big\},
\end{align*}
where $\sqsubseteq$ is as defined above. 
Lemma  3.1 of \cite{FreemanSwart2023} gives that the sets 
$G(f)$ and $H(f)$ are compact subsets of $\Rc$,
moreover the sets $G^{(2)}(f)$ and $H^{(2)}(f)$ are compact subsets of $(\Rc)^2$.
Note that $G^{(2)}(f)$ and $H^{(2)}(f)$ preserve information about the total order $\sqsubseteq$,
whereas $G(f)$ and $H(f)$ do not.
It can be seen that $G^{(2)}(f_n)\to G^{(2)}(f)$ implies 
$G(f_n)\to G(f)$,
but the converse is not true.
Similarly for $H^{(2)}$ and $H$.
If $G(f_n)\to G(f)$ then $H(f_n)\to H(f)$,
but the converse is not true.
Similarly for $G^{(2)}$ and $H^{(2)}$.

\begin{prop}
\label{p:J1M1}
In each case, under the metric listed $\Pi$ is a Polish space.
\begin{enumerate}
\item The \emph{J1 topology}: $d_{\rm J1}(f,g)=d_{\mc{K}((\Rc)^2)}\l(G^{(2)}(f),G^{(2)}(g)\r)$.
\item The \emph{M1 topology}: $d_{\rm M1}(f,g)=d_{\mc{K}((\Rc)^2)}\l(H^{(2)}(f),H^{(2)}(g)\r)$.
\end{enumerate}
\end{prop}
We mention also that $d_{\rm J2}(f,g)=d_{\mc{K}(\Rc)}\l(G(f),G(g)\r)$
and $d_{\rm M2}(f,g)=d_{\mc{K}(\Rc)}\l(H(f),H(g)\r)$
respectively correspond to Skorohod's J2 and M2 topologies.
In \cite{FreemanSwart2023} it is shown that each such metric
is equivalent to the corresponding classical Skorohod metric on 
$\mc{D}_{[a,b]}=\{f\in\Pi\-I(f)=[a,b],\; f(a-)=f(a+)\}$, for $a,b\in\R$.
See Sections 2.4, 3.2 and 3.4 of that article for details and proofs of these facts.

We now specialize to the case that is relevant to the present article:
\begin{quote}
\emph{from now on the space $\Pi$ is (implicitly) equipped with the M1 topology.}
\end{quote}
The same applies to subsets of $\Pi$.
We write $d_\Pi=d_{\rm M1}$, generating the M1 topology on $\Pi$,
but we remark that the metric space $(\Pi,d_\Pi)$ is not complete.
In Appendix \ref{a:M1} we collate together some results from \cite{FreemanSwart2023} concerning the M1 topology,
including criteria for relative compactness and tightness.


Let
$\Pi^\up=\{f\in\Pi\-\tau_f=\infty\}$,
$\Pi^\down=\{f\in\Pi\-\sig_f=-\infty\}$,
and $\Pi^\updo=\Pi^\up\cap\Pi^\down$
be the subspaces of (respectively) forwards and backwards \emph{half-infinite} paths, and \emph{bi-infinite}
paths. 
Note that $\Pi$ and $\Pi^\updownarrow$ are both symmetric under time reversal.
We write $\Pi_{\rm c}=\{f\in\Pi\-f(t-)=f(t+)\text{ for all }t\in I(f)\}$
for the subspace of \emph{continuous} paths.
We write $\Pi^\up_{\rm c}:=\Pi^\up\cap\Pi_{\rm c}$ and so on.
Informally, the induced topology on $\Pi_{\rm c}$
may be described as convergence of starting and final times plus
locally uniform convergence of continuous paths. 
In fact $\Pi^\uparrow_{\rm c}$ is the state space introduced by \cite{FontesIsopiEtAl2004a} for the Brownian web.
See Proposition 3.4 of \cite{FreemanSwart2023} for details.

Our main results will concern systems of half-infinite {\cadlag} paths 
and we will tend to state results forwards in time
i.e.~we will mostly work in $\Pi^\up$ or its subspace $\Pi^\updownarrow$. 
We require $\Pi^\downarrow$ only for results concerning duality.
We write $\mc{K}(\Pi)$ for the metric space of compact subsets of $\Pi$,
where the underlying metric on $\Pi$ comes from Proposition \ref{p:J1M1}.
It is easily seen from Proposition \ref{p:J1M1} that 
$\Pi^\uparrow$, $\Pi^\downarrow$ and $\Pi^\updownarrow$ are closed subsets of $\Pi$,
from which it follows that $\mc{K}(\Pi^\uparrow)$, $\mc{K}(\Pi^\updownarrow)$ and $\mc{K}(\Pi^\downarrow)$
are closed subsets of $\mc{K}(\Pi)$.


\subsection{Crossing and ordering}
\label{sec:terminology}

We now introduce notation and terminology associated to $\Pi$ and $\Pi^\uparrow$.
We say that a {\cadlag} path $g$ \emph{extends} a {\cadlag} path $f$ if 
$H^{(2)}(f)\sw H^{(2)}(g)$.
We note that, in all but the trivial case for which $\s_f=\t_f$ or $\s_g=\t_g$
this reduces to the more intuitive condition $H(f)\sw H(g)$.
The point is that the total ordering of $H(f)$ (see Section \ref{sec:pi}) 
should coincide with its induced order as a subset of $H(g)$.

We write $f\sw g$ to denote that $g$ extends $f$.
It is easily seen that $\sw$ is a partial order on $\Pi$
and we write the corresponding strict order relation as $\subset$.
For $A\sw\Pi$ we write 
\begin{align}
\label{eq:Amax_def}
A_{\max} &= \{g\in A\-\text{for all }f\in A,\text{ if }g\sw f\text{ then }g=f\} \\
A_\uparrow &= \{g\in\Pi^\uparrow\- g\sw f \text{ for some }f\in A\} \label{eq:A_uparrow_def}. 
\end{align}
In words, 
$A_{\max}$ denotes 
the set of maximal elements of $(A,\sw)$,
whilst $A_\uparrow$ denotes
the set of half-infinite paths
that may be extended to some $f\in A$.

For sets of paths $A,B\sw\Pi^\uparrow$ we define the relation 
\begin{equation}
\label{eq:preceq}
A\preceq B \quad\iff\quad A_\uparrow\cap B\sw A\sw B_\uparrow,
\end{equation}
which, as a consequence of Lemma \ref{l:preceq_deterministic}, is a partial order on (the set of) subsets of $\Pi^\uparrow$.
The corresponding strict order relation is written $\prec$.
The intuition behind \eqref{eq:preceq} is one of efficient covering of space-time: 
loosely
$A\prec B$ means that $B$ covers more of space-time using longer and/or fewer paths than $A$.

Note that if $A$ and $B$ are random variables on the same probability space then
we can make sense of the event $\{A\preceq B\}$ via \eqref{eq:preceq}.
If $A$ and $B$ are $\mc{K}(\Pi^\uparrow)$ valued random variables without an implicit coupling then,
with mild abuse of notation, we further extend $\preceq$ by writing
$A\preceqd B$ 
if and only if there exists a coupling of $A$ and $B$ such that 
$\P[A\preceq B]=1$.
We write $A\precd B$ if $A\preceqd B$ where $A$ and $B$ do not have the same marginal distributions.
In Lemma \ref{l:preceq_random} we show that $\preceqd$ is a partial order on the space of (laws of) $\mc{K}(\Pi^\uparrow)$ valued random variables.

\begin{defn}
\label{d:crossing}
We say that paths $f,g\in\Pi$ are \textit{non-crossing}
if there exists paths $f',g'\in\Pi^\updownarrow$
such that $f\sw f'$, $g\sw g'$ and $f(t\star)\leq g(t\star)$ for all $t\star\in\ov\R_\mfs$.
\end{defn}

The precise format of Definition \ref{d:crossing} is motivated by the complication that a pair of {\cadlag} paths may share the same initial (or final) time and might both jump at this time, with perhaps overlapping jumps. 
The reader may wish to glance forward at Figure \ref{fig:cross}
which depicts some of these complications.
We say that a set of paths $A\sw\Pi$ is non-crossing if all pairs of elements of $A$ are non-crossing.
We say that $A$ and $B$ are non-crossing if $A\cup B$ is non-crossing.

If $f\in\Pi$ and $z\in H(f)$,
then we say that $f$
\emph{passes through} the space-time point $z$. 
For $f\in\Pi^\uparrow$, if $f$ passes through $z\in\Rc$ then we define the \textit{restriction}
\begin{equation}
\label{eq:f|z}
\text{
$f|_z\in\Pi^\uparrow$ to be the unique $g\sw f$ such that $(g(\sigma_g-),\sigma_g)=z$.
}
\end{equation}
More generally, we say that $f$ passes through a set $D\sw\Rc$ if 
$f$ passes through some point $(x,t)\in D$. 
For an (unordered) set
$D\sw\Rc$ and $A\sw\Pi^\uparrow$ we write
\begin{equation}
\label{eq:A(z)}
A(D) = \big\{f\in A:f\text{ passes through }D\big\}
\end{equation}
for the set of paths in $\mc{A}$ that pass through $D$.
For convenience, for $z\in\Rc$ we write 
$A(z)=A(\{z\})$.
In the same vein we define
\begin{equation}
\label{eq:A|z}
A|_D = \{f|_z\in\Pi^\uparrow\-f\in\mc{A}(z)\text{ and }z\in D\}
\end{equation}
for the set of paths in $A$ that pass through some $z\in B$,
with the part prior to $z$ removed.
For $z\in\Rc$ we write $A|_z=A|_{\{z\}}$.

\subsection{Weaves, webs and flows}
\label{sec:results_weaves}

We are interested in systems of non-crossing paths that
touch every point of space-time. More rigorously,
we say that $A\sw\Pi^\uparrow$ is \emph{pervasive} if
$A(z)\neq\emptyset$ for all $z\in\Rc$.
We remark that, as a consequence of Lemma \ref{l:appdx_1_sw_limits},
if $A\in\mc{K}(\Pi)$ then
it suffices to check that $A(z)\neq\emptyset$ on a dense subset of $z\in\Rc$.
The key objects studied within the present article are as follows.

\begin{defn}
\label{d:weave}
A \emph{weave} is the law of a $\mc{K}(\Pi^\uparrow)$ valued random variable that is almost surely
pervasive and non-crossing.
Let $\mathscr{W}$ denote the set of weaves.
A weave that is a minimal element of $\mathscr{W}$ with respect to $\preceqd$ is known as a \textit{web}.
A weave that is a maximal element of $\mathscr{W}$ with respect to $\preceqd$ is known as a \textit{flow}.
\end{defn}

A weave is a probability measure on $\mc{K}(\Pi)$,
however we mildly abuse terminology in the usual way (c.f.~`a' Brownian motion) 
by saying that a $\mc{K}(\Pi^\uparrow)$ valued random variable is a weave if its law satisfies Definition \ref{d:weave}.
Similarly for webs and flows.
We write 
\begin{equation}
\label{eq:Wdet}
\mathscr{W}_{\det}=\{\mc{A}\in\mc{K}(\Pi^\uparrow)\-\mc{A}\text{ is non-crossing and pervasive}\}.
\end{equation}
If $\mc{A}$ is a weave then $\P[\mc{A}\in\mathscr{W}_{\det}]=1$.
Elements of $\mathscr{W}_{\det}$ are said to be \textit{deterministic weaves}.
Although $\mathscr{W}_{\det}$ is not formally a subset of $\mathscr{W}$,
it may be viewed as such by identifying 
$\mc{A}\in\mathscr{W}_{\det}$ with the probability measure that is a point-mass at $\mc{A}$.
The following concept plays a central role for both deterministic and random weaves.

\begin{defn}
\label{d:ramified}
Let $\mc{A}$ be a weave and let $z\in\Rc$ be a possibly random point of space-time.
We say that $z$ is a \textit{ramification point} of $\mc{A}$ 
if there exists $f,g\in\mc{A}(z)$ such that neither $f\sw g$ nor $g\sw f$.
Otherwise, $z\in\Rc$ is said to be \textit{non-ramified} in $\mc{A}$.
\end{defn}

If $z$ is non-ramified in $\mc{A}$ with $f,g\in\mc{A}(z)$ then $f\sw g$ or $g\sw f$.
Loosely, ramification points capture where weaves display atypical path behaviour,
for example branching or coalescing of paths, or perhaps both.
We stress that `$z$ is non-ramified in $\mc{A}$' is an event, with some associated probability,
and not a deterministic statement.
If it is clear from the context 
which weave is meant then we may simply say 
that $z\in\R^2_c$ is non-ramified.
We say that $D\sw\Rc$ is non-ramified 
if all $z\in D$ are non-ramified.


A recurring theme in our results is that behaviour at non-ramified points determines the full behaviour of the weave.
This suggests that non-ramified points should be plentiful.
In Lemma \ref{l:ramification_meas_zero_det} we show that 
for any weave $\mc{A}$
the deterministic set 
$$\{z\in\Rc\-\P[z\text{ is ramified in }\mc{A}]>0\}$$ 
has zero Lebesgue measure.
Hence if a random point $z\in\Rc$ is sampled independently of $\mc{A}$, according to some non-atomic law,
then $z$ is almost surely non-ramified in $\mc{A}$.

Let $A\in\mc{K}(\Pi)$ and let $D\sw\Rc$.
We define a key pair of deterministic operations as follows:
\begin{align}
\web_D(A)
&=\ov{(A|_D)_\uparrow}\label{eq:web_op}
\\
\flow(A) 
&=\{f\in \Pi^\updownarrow\- f\text{ does not cross }A\}. \label{eq:flow_op}
\end{align}
Let us briefly comment on the $\web_D(\cdot)$ operation.
The use of $\ov{(\cdot)}$ denotes closure in $\mc{K}(\Pi^\uparrow)$.
In Lemma \ref{l:web_op_D} we will see that, for $\mc{A}\in\mathscr{W}_{\det}$,
the value of $\web_D(\mc{A})$ does not depend upon the choice of dense and non-ramified $D\sw\R^2$.
Thus \eqref{eq:web_op} defines a deterministic function $\web(\cdot)$ with domain $\mathscr{W}_{\det}$,
which we write without explicit specification of $D$.

We are now ready to state our first main result.
It shows that extremal points of $(\mathscr{W},\preceqd)$ may be characterized 
as fixed points of the $\web$ and $\flow$ operators.
This leads to a particularly nice description of the structure of $\mathscr{W}$.

\begin{theorem}
\label{t:weave_structure}
Let $\mc{A}$ be a weave. 
\begin{enumerate}
\item
The following are equivalent:
(a) $\mc{A}$ is a web; (b) $\mc{A} \eqas \web(\mc{A})$.
\item 
The following are equivalent:
(a) $\mc{A}$ is a flow; (b) $\mc{A} \eqas \flow(\mc{A})$; (c) $\mc{A}\sw\Pi^\updownarrow$ almost surely.
\item
Almost surely, $\web(\mc{A})\preceq\mc{A}\preceq\flow(\mc{A})$.
\item
There exists a unique (in distribution) web $\mc{W}$ and a unique flow $\mc{F}$ such that
$\mc{W}\preceqd \mc{A}\preceqd\mc{F}$, given by
$\mc{W}\eqd\web(\mc{A})$ and $\mc{F}\eqd\flow(\mc{A})$.
\end{enumerate}
\end{theorem}
Thus, the space of weaves is partitioned by the equivalence relation
\begin{equation}
\label{eq:weave_sim}
\mc{A}\sim\mc{B}\quad\iff\quad \web(\mc{A})\eqd \web(\mc{B}) \quad\iff\quad \flow(\mc{A})\eqd \flow(\mc{B}),
\end{equation}
under which each equivalence class has a web as its unique minimal element, and a flow as its unique maximal element.

The relation $\sim$ can also be characterized using finite collections of particle motions, for which we now introduce formal notation.
Consider a weave $\mc{A}$, a flow $\mc{F}$ and a non-ramified point $z\in\Rc$. 
The set $\mc{A}|_z$
contains a single path, which begins at $z$.
Similarly, $\mc{F}(z)$ contains a single path, which passes through $z$.
This makes it natural to define versions of \eqref{eq:A(z)} and \eqref{eq:A|z} 
specialized to ordered sets of non-ramified points.

Let $m\in\N$. 
Given a weave $\mc{A}$ and an almost surely non-ramified $\vec{z}=(z_i)_{i=1}^m\in(\Rc)^m$, 
we write $\mc{A}|_{\vec{z}}=(f_1,\ldots,f_m)$ where 
$\{f_i\}\stackrel{a.s}{=}\mc{A}|_{z_i}$.
Similarly, given a flow $\mc{F}$ and an almost surely non-ramified $\vec{z}=(z_i)_{i=1}^m\in(\Rc)^m$, 
we write $\mc{F}(\vec{z})=(f'_1,\ldots,f'_n)$
where $f'_i\in\Pi^\updownarrow$ is the almost surely unique element of $\mc{F}(z_i)$.
We say that $\mc{A}|_{\vec{z}}$ is the \textit{(forwards in time) $m$-particle motion} of $\mc{A}$ from $\vec{z}$.
They are defined up to almost sure equivalence.

Loosely, the relation $\sim$ also characterizes when two weaves have 
the same forwards in time $m$-particle motions, in distribution.
With this in mind, 
we will need to make statements featuring multiple weaves that concern non-ramified points.
We adopt the implicit convention that non-ramification 
is with respect to all weaves featured in the corresponding statement.
We are now ready to state our second main result.

\begin{theorem}
\label{t:weaves_characterization}
Let $\mc{A},\mc{B}$ be weaves. 
\begin{enumerate}
\item
Suppose that $\mc{A}\sim\mc{B}$.
If $\vec{z}\sw\R^2_c$ is finite and almost surely non-ramified
then $\mc{A}|_{\vec{z}}\eqd \mc{B}|_{\vec{z}}$.
\item
The following are equivalent: 
\begin{enumerate}[label={(\alph*)}]
\item $\mc{A}\sim\mc{B}$;
\item there exists a coupling of $\mc{A}$ and $\mc{B}$ such that $\mc{A}\cup\mc{B}$ is almost surely non-crossing;
\item there exists a (deterministic) countable dense $D\sw\R^2$,
which is almost surely non-ramified, such that $\mc{A}|_{\vec{z}}\eqd \mc{B}|_{\vec{z}}$ for all finite $\vec{z}\sw D$.
\end{enumerate}
\end{enumerate}
\end{theorem}	

Theorems \ref{t:weave_structure} and \ref{t:weaves_characterization} 
lead towards an appealing limit theory for weaves,
which we now develop.
We denote weak convergence (i.e.~convergence in law) by $\tod$.
We saw in Section \ref{sec:pi} that $\mc{K}(\Pi^\uparrow)$ is a Polish space,
thus weak convergence of $\mc{K}(\Pi^\uparrow)$ valued random variables,
or equivalently of probability measures on $\mc{K}(\Pi)$,
is defined in the standard way e.g.~as in Section 3.3 of \cite{EthierKurtz1986}.
This provides a natural sense in which to consider weak convergence of weaves.

Our next theorem shows that weak convergence of flows is equivalent to 
tightness plus weak convergence of $m$-particle motions,
and explores the same statement in the context of equivalence classes of weaves.
Note that convergence of $m$-particle motions is a statement about weak convergence of
$\ov{\R}^m$ valued stochastic processes,
which is within the realms of the classical theory in e.g.~\cite{EthierKurtz1986}.

\begin{theorem}
\label{t:flow_conv}
Let $\mc{F}_n,\mc{F}$ be flows.
\begin{enumerate}
\item
If $\mc{F}_n\tod \mc{F}$ then
for any $m\in\N$ and non-ramified $\vec{z}\in(\Rc)^m$ we have
$\mc{F}_n(\vec{z})\tod \mc{F}(\vec{z})$.
\item
Any weak limit point of $(\mc{F}_n)$ is a flow.
If $(\mc{F}_n)$ is tight and
for any $m\in\N$ and almost surely non-ramified $\vec{z}\in(\R^2)^m$ we have
$\mc{F}_n|_{\vec{z}}\tod \mc{F}|_{\vec{z}}$,
then $\mc{F}_n\tod \mc{F}$.
\item
Let $\mc{A}_n,\mc{A}$ be weaves with $\mc{A}_n\sim\mc{F}_n$ and $\mc{A}\sim\mc{F}$.
If $\mc{A}_n\tod \mc{A}$ then $\mc{F}_n\tod \mc{F}$.
Conversely, 
if $\mc{F}_n\tod \mc{F}$ then any weak limit point $\mc{B}$ of $(\mc{A}_n)$ is a weave and satisfies $\mc{B}\sim\mc{F}$.
\end{enumerate}
\end{theorem}

Lemma \ref{l:flow_map_cts},
shows that the function $\flow(\cdot)$ is continuous on $\mathscr{W}_{\det}$,
which gives the forwards implication of part 3 of Theorem \ref{t:flow_conv}.
The function $\web(\cdot)$ is not continuous, 
as shown by example in Figure \ref{fig:web_dc},
which depicts a sequence of webs $(\mc{W}_n)$ converging to a weave $\mc{A}$
that is neither a web nor a flow.
This suggests that, for purposes of convergence, 
flows are a more natural representative element of their equivalence class than webs.

Let us now give analogues for general weaves of parts 1 and 2 of Theorem \ref{t:flow_conv}.
In part 3 of Theorem \ref{t:flow_conv} flows provide an overarching structure for weaves
in which the non-crossing property is preserved by taking limits of paths.
The non-crossing property is 
preserved when taking limits of bi-infinite paths, but
is not necessarily preserved in limits of half-infinite paths.
Consequently, if we wish to establish convergence of weaves
but also wish to avoid handling their associated flows,
then it becomes necessary to check that limit points are non-crossing.
See Figure \ref{fig:example_weaves} for a related warning example.

\begin{theorem}
\label{t:weave_conv}
Let $\mc{A}_n,\mc{A}$ be weaves.
\begin{enumerate}
\item
If $\mc{A}_n\tod \mc{A}$ then
for any $m\in\N$ and non-ramified $\vec{z}\in(\Rc)^m$ we have
$\mc{A}_n|_{\vec{z}}\tod \mc{A}|_{\vec{z}}$.
\item
If a weak limit point $\mc{B}$ of $(\mc{A}_n)$ is non-crossing then $\mc{B}$ is a weave.
If, additionally,
for any $m\in\N$ and almost surely non-ramified $\vec{z}\in(\R^2)^m$ we have
$\mc{A}_n|_{\vec{z}}\tod \mc{A}|_{\vec{z}}$,
then $\mc{A}\sim\mc{B}$.
\end{enumerate}
\end{theorem}

Our next result concerns time-reversed duality,
for which we must introduce some more notation.
Given $f\in\Pi^\uparrow$, define $f^\rot\in\Pi^\downarrow$ by $f^\rot(t\pm)=-f(-t\mp)$.
This operation, which corresponds to a rotation of space by 180 degrees, is applied pointwise to sets of paths as $\mc{A}^\rot=\{f^\rot\-f\in\mc{A}\}$.
Clearly $(A^\rot)^\rot=A$.
Note that $\cdot^\rot$ is an automorphism of $\Pi$ and $\Pi^\updownarrow$,
and that $(\Pi^\uparrow)^\rot=\Pi^\downarrow$.
Proposition \ref{p:relcom_tightness} implies that $A\sw\Pi$ is relatively compact if and only if $A^\rot$ is,
and it is trivial to see that the same holds for pervasiveness and the non-crossing property.
For $B\sw\Pi^\downarrow$ we write
\begin{equation}
B_\downarrow = \{g\in\Pi^\downarrow\- g\sw f \text{ for some }f\in B\} \label{eq:A_downarrow_def}
\end{equation}
in analogy to \eqref{eq:A_uparrow_def}.

A random subset of $\Pi^\downarrow$ that is compact, pervasive and non-crossing is said to be a \textit{dual weave}.
Thus $\mc{A}$ is a dual weave if and only if $\mc{A}^\rot$ is a weave.
We say that $\mc{A}\sw\Pi^\downarrow$ is a dual web if and only if $\mc{A}^\rot$ is a web.
Equivalently, we could define a relation on $\mc{K}(\Pi^\downarrow)$ akin to \eqref{eq:preceq} but with 
time reversed 
(i.e.~with $\cdot_\downarrow$ in place of $\cdot_\uparrow$)
and then a dual web would be a minimal dual weave with respect to this relation.
Note that $\mc{F}$ is a flow if and only if $\mc{F}^\rot$ is a flow.

\begin{defn}
\label{d:double_web}
A pair $(\mc{W},\hat{\mc{W}})$ is said to be a \textit{double web} if it consists of a web $\mc{W}$ and dual web $\hat{\mc{W}}$ coupled such that $\mc{W}\cup\hat{\mc{W}}$ is non-crossing.
\end{defn}
Our next result states that each web gives rise to a corresponding double web, 
in which $\hat{\mc{W}}$ essentially contains the extra segments of paths that are required to construct $\mc{F}$ directly from $\mc{W}$. 
There is a subtlety, however: when we come to connect $\mc{W}$ and its dual $\hat{\mc{W}}$ together to create $\mc{F}$, we must be careful not to introduce crossing.
In particular we should be wary of ramification points, at which the multiple in-going and out-going trajectories must be reconnected in such a way that they enter and exit $z$ without crossing each other.

Recall our terminology that 
a path $f\in\Pi$ begins at the point $(f(\sigma_f-),\sigma_f)\in\Rc$
and 
ends at the point $(f(\tau_f+),\tau_f)$.
Given $D\sw\Rc$ we say that $f$ begins in $D$ if $f$ begins at some point of $D$,
and $f$ ends in $D$ if $f$ ends at some point of $D$.
For $f\in\Pi^\uparrow,g\in\Pi^\downarrow$ 
are non-crossing,
with $\t_g=\s_f$ and $g(\t_g+)=f(\s_f-)$,
then we define $h=g_{\hookrightarrow} f\in\Pi^\updownarrow$ by 
$$
h(t\star)=
\begin{cases}
g(t\star) & \text{ for }t\star\leq \s_f -\\
f(t\star) & \text{ for }t\star\geq \s_f +.
\end{cases}
$$
Thus $g_{\hookrightarrow} f\in\Pi^\updownarrow$ is the concatenation of a path $f\in\Pi^\uparrow$ and a path $g\in\Pi^\downarrow$
that (respectively) begin and end at the same point of space-time.

\begin{theorem}
\label{t:dual_webs}
Let $\mc{A}$ be a weave and let $\mc{W}=\web(\mc{A})$, $\mc{F}=\flow(\mc{A})$.
There exists a dual web $\mc{\widehat{W}}$ on the same probability space such that $(\mc{W},\widehat{\mc{W}})$ is a double web,
and $\mc{\widehat{W}}$ is unique up to almost sure equivalence.
For any $D\sw\R^2$
that is dense and almost surely non-ramified,
\begin{align}
\widehat{\mc{W}}
&=\ov{\{g\in\Pi^\downarrow\-g\text{ does not cross }\mc{A}\text{ and }g\text{ begins in }D\}_\downarrow}
\label{eq:What_def} \\
\mc{F}
&=\ov{\{g_{\hookrightarrow}f\in\Pi^\updownarrow\-g\in\wh{\mc{W}}\text{ ends and }f\in\mc{W}\text{ begins at the same point of }D\}}. \label{eq:F_WWHat}
\end{align}
\end{theorem}

From Theorems \ref{t:weave_structure} and \ref{t:dual_webs}, if $(\mc{W}, \wh{\mc{W}}, \mc{F})$ is a triplet containing a double web and flowlines, all coupled to be non-crossing of each other, then given any one element of the triplet we may reconstruct the other two. 

Recall that $\Pi_c=\{f\in\Pi\-f\text{ is continuous}\}$.
Let us end this section by recording that continuity of paths is preserved through all of 
the various relationships established above.
We say that a weave $\mc{A}$ is continuous if $\mc{A}\sw\Pi^\uparrow_c$, and similarly for dual weaves.

\begin{theorem}
\label{t:weaves_continuity}
Let $\mc{A}$ and $\mc{B}$ be weaves such that $\mc{A}\sim\mc{B}$.
Then $\mc{A}$ is continuous if and only if $\mc{B}$ is continuous.
If $\mc{W}$ is a web then $\mc{W}$ is continuous if and only if $\wh{\mc{W}}$ is continuous.
\end{theorem}



\subsection{Discussion}
\label{sec:discussion}

In this section we first discuss how various objects related to the Brownian web fit into the framework of weaves.
We then discuss other topologies that have been introduced for sets of paths and closely related objects.
Some open problems are mentioned along the way.

\subsubsection{The equivalence class of the Brownian web}
\label{sec:bw}

We write $\Wb$ for the Brownian web,
as defined in (for example) Theorem 2.3 of the survey article of \cite{SchertzerSunEtAl2017}.
Let us first resolve an apparent conflict in notation.
In common with the literature of $\Wb$,
\cite{SchertzerSunEtAl2017}
defined $\Wb(D)$ to be the set of paths in $\Wb$ that \emph{begin} at some $z\in D$, where $D\sw\Rc$.
According to \eqref{eq:A(z)} we reserve $\Wb(D)$ for the set of paths that \textit{pass through} some $z\in D$.
This may seem to conflict at first glance, 
but in fact there is no conflict here, as we now explain.

For the Brownian web, the notation $\Wb(D)$ is widely used when
$D\sw\Rc$ is deterministic and countable,
for example in the well known identity $\Wb\eqas\ov{\Wb(D)}$.
For the Brownian web, almost surely, for each $z\in D$ the set
$\Wb(z)=\{f\in\Wb\-f\text{ passes through }z\}$ consists of a single path that begins at $z$.
Thus, for the Brownian web, $\Wb(z)\eqas\Wb|_z$ and $\Wb(D)\eqas\Wb|_D$.
In general the distinction between $\mc{A}(D)$ and $\mc{A}|_D$ does matter
and \eqref{eq:web_op}, which features the latter, is required to construct $\web(\mc{A})$.
See the left part of Figure \ref{fig:web_dc} for a related example.

\begin{figure}[t]
\centering
\includegraphics[scale=1.07]{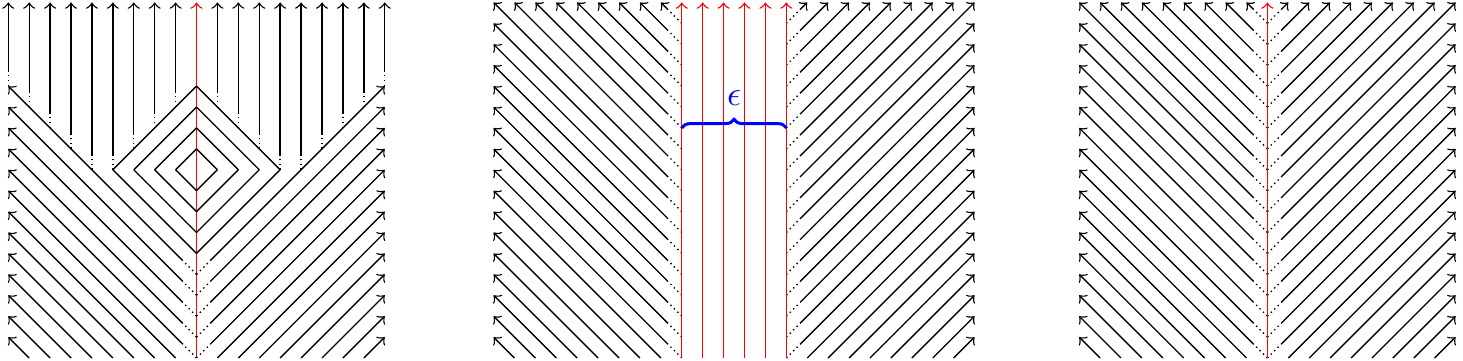}   
\caption{
In all three figures, time runs upwards and the spatial axis is horizontal.
In each figure a weave is depicted via solid lines,
with the corresponding flow depicted via including dotted lines.\\
\textit{On the left:}
A weave $\mc{A}$ such that
$\mc{A}(D)$ contains the bi-infinite red path,
whereas $\mc{A}|_D$ does not.
Here $D\sw\R^2$ must be dense and non-ramified with respect to the weave $\mc{A}$.
In this case we have that
$\web(\mc{A})=\ov{(\mc{A}|_D)_\uparrow}\prec\ov{\mc{A}(D)_\uparrow}$.\\
\textit{On the center and right:}
An example related to continuity of the $\flow(\cdot)$ map
and lack of continuity of the $\web(\cdot)$ map,
as well as to the existence of isolated points within flows and general weaves.
In the center, the weave $\mc{A}_\eps$ 
and corresponding flow $\mc{F}_\eps=\flow(\mc{A}_\eps)$
are depicted.
The limiting weave $\mc{A}=\lim_{\eps\to 0}\mc{A}_\eps$ is depicted on the right,
with $\mc{F}=\flow(\mc{A})$ again depicted via including dotted lines.
The red paths collapse to a single bi-infinite path in the limit.
Note that $\mc{F}_\eps\to\mc{F}$ in accordance with part 3 of Theorem \ref{t:flow_conv}.
In this case $\mc{W}_\eps=\web(\mc{A}_\eps)$ is equal to $(\mc{A}_\eps)_\uparrow$.
Therefore $\mc{W}_\eps\to\mc{A}_\uparrow$.
On the right, note that $\mc{W}=\web(\mc{A})$ does not include the bi-infinite red path,
and that every point of this path ramified.
Consequently $\lim_{\eps\to 0}\mc{W}_\eps\neq \mc{W}$,
showing that the map  $\web(\cdot)$ is discontinuous at $\mc{W}$.
Note that the bi-infinite red path on the right is an isolated point of $\mc{A}$
but is not an isolated point of $\mc{F}$.
}
\label{fig:web_dc}
\end{figure}

\begin{lemma}
\label{l:bw_web}
It holds that $\Wb$ is a web.
\end{lemma}

A short proof of Lemma \ref{l:bw_web}
is given in Appendix \ref{a:bw_web}.
It rests on combining Theorem \ref{t:weave_structure} with the key property $\Wb\eqas\ov{\Wb(D)}$,
from which we may deduce that $\Wb\eqas\web(\Wb)$.
We refer to the equivalence class of the Brownian web as the class of \textit{Brownian weaves},
which are introduced for the first time in the present article.
The \emph{double Brownian web} $(\Wb,\wh{\Wb})$
as defined in Theorem 2.4 of \cite{SchertzerSunEtAl2017}
consists of a pair of coupled random variables,
where $(\wh{\Wb})^\rot$ and $\Wb$ have the same marginal distributions, 
such that
$\Wb$ and $\wh{\Wb}$ do not cross.
It follows from Theorem \ref{t:dual_webs} that $(\Wb,\wh{\Wb})$ is a double web in our framework.

Let us now comment on
the role of special points of the Brownian web 
versus ramification points of weaves.
Within the Brownian web each space-time point $z=(x,t)$ is assigned a `type' 
denoted $(z_{\rm in},z_{\rm out})$.
Here, 
$z_{\rm in}$ is the number of equivalence classes of incoming paths of $z=(x,t)$
that are distinct under the relation that two paths are equivalent if they are equal on a time interval $[t-\eps,t]$ for some $\eps>0$.
Similarly for $z_{\rm out}$, using outgoing paths and $[t,t+\eps]$.
See Section 2.5 of \cite{SchertzerSunEtAl2017} for further detail.
Note that $z_{\rm in}$ and $z_{\rm out}$ are local properties (in space-time) of $z$, 
whereas
ramification of $z$ is not a local property,
because ramification depends on the behaviour of paths within $\mc{A}(z)$ for \textit{all} time.
The two concepts are related but have different purposes.

Theorem 2.11 of \cite{SchertzerSunEtAl2017} describes the various types of special point within $\Wb$,
plus their associated local geometry and Hausdorff dimension.
This provides a highly detailed understanding of the microscopic structure of $\Wb$.
Within the Brownian web points of type $(0,1)$ have full measure in $\R^2$ and are non-ramified.
Points of all other types are regarded as `special' points of $\Wb$
and are ramified in $\Wb$.
This is something of a coincidence:
in general weaves points of type $(0,1)$ can be ramified
(for example, if they are upstream of a branch point)
and points of type $(1,1)$ can be non-ramified
(for example, the constant paths in Figure \ref{fig:example_weaves} that do not interact with the jumps).
Moreover,
for general weaves
the most abundant type is not necessarily $(0,1)$.
Within Figures \ref{fig:example_weaves}
and \ref{fig:web_dc}
all weaves depicted have points of type $(1,1)$ with full measure.

\cite{FontesNewman2006} explored two examples of Brownian weaves.
They considered the \emph{full Brownian web},
which in our terminology is precisely the flow $\Fb$ associated to the Brownian weave,
and the \emph{full forwards Brownian web},
which in our terminology is $(\Fb)_\uparrow$,
and is an example of a Brownian weave that is neither a flow nor a web.
Extension of paths, backwards in time, also features within \cite{FontesNewman2006}.
Their treatment relies fully on the structure of the Brownian web,
using the forwards-backwards reflection of Brownian paths established by 
\cite{SoucaliucTothEtAl2000a}.
The particular case of our own Theorem \ref{t:dual_webs} corresponding to $(\Wb,\hat{\Wb})$ and $\Fb$
is essentially Proposition 2.5 of \cite{FontesNewman2006}.
They also adapated results of \cite{FontesIsopiEtAl2004a}
and the earlier work of \cite{Piterbarg1998}
to give weak convergence criteria for $\Fb$.

The Brownian net of \cite{SunSwart2008} is not a weave, 
because it contains paths that cross. 
It is interesting to ask if a generalized form of nets exist,
as a family of $\mc{K}(\Pi)$ valued random variables
corresponding to general weaves,
but we do not attempt to answer this question within the present article.
It is also interesting to ask if there is a generalization of our results 
to pervasive systems that permit crossing,
such as the $\alpha$-stable web of \cite{MountfordRavishankarEtAl2019},
however at present very few non-trivial examples of such systems are available.

\subsubsection{Related topologies}
\label{sec:related_topologies}

We have already mentioned that the state space $\mc{K}(\Pi_{\rm c}^\uparrow)$ constructed by \cite{FontesIsopiEtAl2004a},
upon which most of the recent work on the Brownian web is based,
is a topological subspace of our own state space $\mc{K}(\Pi^\uparrow)$.
In this section we discuss some other recent works
concerning topologies induced upon sets of paths.

\cite{BerestyckiEtAl2015}
mapped sets of paths to sets of `tubes'.
Loosely, a tube is a subset of space-time that possesses a bottom face, sides and a top face.
The so-called \textit{tube topology} is then induced based on which tubes are traversed
(i.e.~from bottom to top, as time passes, whilst remaining within the sides) 
by the paths.
It is restricted to sets of continuous coalescing paths, 
but permits paths to cross.
  
The state space defined by \cite{BerestyckiEtAl2015} is compact,
which has substantial technical advantages (in particular, tightness becomes automatic)
but this comes at the cost of some loss of detail:
the tube topology is coarser than that of $\mc{K}(\Pi^\uparrow_{\rm c})$
and the map from sets of paths to sets of tubes is not injective.
In fact, 
the tube topology regards the sets
$A$ and $A_\uparrow$ as identical, where $f\in\Pi^\uparrow$,
and similarly for all $A'$ such that $A_\uparrow\preceq A'\preceq A$,
meaning that much of the structure displayed in Theorem \ref{t:weave_structure} is lost.
However, the weaker representation makes characterization and convergence easier.

Another piece of detail that is kept visible in $\mc{K}(\Pi^\uparrow)$,
but is dropped by the tube topology,
is behaviour near the start times of paths.
This can have implications for universality.
For example, \cite{BerestyckiEtAl2015}
showed that systems of coalescing random walkers with heavy tailed jumps (linearly interpolated)
will converge in law to the Brownian web, under the tube topology;
\cite{NewmanEtAl2005} had previously shown that such convergence failed within $\mc{K}(\Pi_{\rm c}^\uparrow)$
because jumps attempt to form in the limit at the initial times of some paths.
We conjecture that such systems \textit{will converge but not to the Brownian web}
if considered as elements of $\mc{K}(\Pi^\uparrow)$.
Loosely, we expect that the limit of such systems will be a Brownian web 
that is suitably augmented with jumps at initial times of paths.

Aside from tubes, another possibility is to
view the Brownian web as a real tree,
where the natural root is a point at time $+\infty$ at which all paths coalesce.
In this representation, loosely, 
each space-time point on a path within the Brownian web
becomes a point within the corresponding real tree,
and a metric is induced that captures 
both the natural tree structure and the distances travelled, forwards in time, 
along individual paths within the Brownian web until coalescence points.
Of course, not all sets of paths are suited to such a representation;
systems that contain branching are not.

\cite{CannizzaroHairer2021} identify a subset of $\mc{K}(\Pi^\uparrow_{\rm c})$ 
that can be naturally represented as real trees.
A similar theme underlies the framework of \textit{marked metric measure spaces} 
introduced by \cite{DepperschmidtEtAl2011}.
Like the tube topology, the setup of \cite{CannizzaroHairer2021} 
does not distinguish between $A_\uparrow$ and $A$,
in this case by associating real trees with sets of paths of the form $A_\uparrow$ 
(i.e.~decreasing sets under $\sw$).
They construct a Polish topology that is shown to be finer than that of $\mc{K}(\Pi)$,
in particular it enforces that coalescence times of paths are preserved when taking limits.




\section{Preliminaries}
\label{sec:technical_1}

We now begin the proofs.
In this section we develop some underlying concepts that we require for our theory of weaves,
related to the structure of $\mc{K}(\Pi^\uparrow)$.
In Section \ref{sec:crossing} we relate our notion of crossing 
to a relation $\lhd$ that describes when one path lies to the left of another.
In Section \ref{sec:preceq} we show that the relation $\preceq$ is a partial order.
Finally, in Section \ref{sec:compatibility} we examine the interaction between order relations and topology,
including to what extent the relations $\sw$, $\lhd$ and $\preceq$ are preserved by taking limits.
These results are technical in nature.
Readers wishing to gloss over technical issues may prefer to 
note the key results and definitions (at minimum, Definition \ref{d:lhd} and Lemma \ref{l:Amax_order}),
then proceed to Section \ref{sec:weaves_det}.

\subsection{On crossing}
\label{sec:crossing}

In this section we study the interaction between 
crossing and the idea of one path staying to the left (or right) of another.
The results in this section have straightforward extensions to $\Pi$
but for brevity we will state results covering $\Pi^\uparrow$.
Consequently we must handle jumps at the initial time $\s_f$ of $f\in\Pi^\uparrow$
but our compactification of space-time, in Figure \ref{fig:Rc}, 
means that jumps do not occur at time $\tau_f=+\infty$.

\begin{defn}
\label{d:lhd}
For $f,g\in\Pi^\uparrow$, we write $f\lhd g$ if there exist $f',g'\in\Pi^\updownarrow$ with $f\sw f'$ and $g\sw g'$
such that $f'(t\star)\leq g'(t\star)$ for all $t\star\in\R_\mfs$.
We write $f\nlhd g$ if this property fails to hold.
\end{defn}

The statement $f\lhd g$ should be interpreted as `$f$ lies to the left of $g$'.
There is some subtlety involved here.
If $f\lhd g$ then $f(t\star)\leq g(t\star)$ for all $t\star\geq\s+$,
but no such guarantee exists concerning $g(\s-)$ and $f(\s-)$.
An example of $f\lhd g$ with $g(\s-)<f(\s-)$ appears as (i) in Figure \ref{fig:cross}.

The relation $\lhd$ does not define a partial order on $\Pi^\uparrow$.
Antisymmetry fails because if $g\sw f$ with $g\neq f$ then we have both $g\lhd f$ and $f\lhd g$.
Transitivity fails too, for example if $f(t\star)=-\frac12+\1\{t\star=\s_f-\}$, $g(t\star)=0$, $h(t\star)=\frac12-\1\{t\star=\s_f-\}$,
with $\s_f=\s_h=0$ and $\s_g=1$, then $f\lhd g$ and $g\lhd h$ but $f\nlhd h$
(we leave it as an exercise for the reader 
to construct an example where transitivity fails and $\{f,g,h\}$ is also non-crossing!).
For this reason we will not use the symbols $\vartriangleright$ and $\ntriangleright$ in this article:
it would be possible to define them as analogous concepts to $\lhd$ and $\nlhd$ with the roles of right and left swapped,
but the notation would be unintuitive since $\nlhd$ and $\vartriangleright$ would not be equivalent.
In Lemma \ref{l:Amax_order} we will establish a more restricted setting in which $\lhd$ is better behaved.

Definition \ref{d:lhd} is intuitive and interacts well with Definition \ref{d:crossing}, but 
it is helpful to have a more explicit characterization of when one path lies to the left of another.
We define subsets 
$L(f)$ and $R(f)$ of $\ov\R\times\R_\mfs$, and
$L_{t\star}(f)$ and $R_{t\star}(g)$ of $\ov{\R}$ by:
\begin{align*}
L_{t\star}(f)&=
\begin{cases}
\emptyset & \text{ if } t\star<\s_f-,\;\text{ or if } t\star=\s_f- \text{ and } f(\s_f-)\leq f(\s_f+), \\
[-\infty,f(t\star)) & \text{ if } t\star\geq\s_f+, \text{ or if } t\star=\s_f- \text{ and } f(\s_f+)<f(\s_f-), \\
\end{cases} \\
R_{t\star}(f)&=
\begin{cases}
\emptyset & \text{ if } t\star<\s_f-,\;\text{ or if } t\star=\s_f- \text{ and } f(\s_f+)\leq f(\s_f-), \\
(f(t\star),\infty] & \text{ if } t\star\geq\s_f+, \text{ or if } t\star=\s_f- \text{ and } f(\s_f-)<f(\s_f+), \\
\end{cases} 
\end{align*}
\begin{equation}
\label{eq:LR_sets}
L(f)=\bigcup_{t\star\in \R_{\mfs}} L_{t\star}(f)\times\{t\star\},
\qquad\quad
R(f)=\bigcup_{t\star\in \R_{\mfs}} R_{t\star}(f)\times\{t\star\}.
\end{equation}
Note that $L(f)$ and $R(f)$ are subsets of $\ov{\R}\times\R_\mfs$.
The significance of $L(f)$ is that if $g(t\star)\in L_{t\star}(f)$ then 
$g$ must stay `to the left' of $f$ in order to avoid crossing it.
Similarly if $g({t\star})\in R_{t\star}(g)$, to the right.
We will formalize this intuition in Lemma \ref{l:lhd_left_right}.
See Figure~\ref{fig:cross} for a picture. 

For $t\geq \s_f+$, the sets $L_{t\star}(f)$ and $R_{t\star}(f)$ are, respectively,
the set of points strictly to the left and right of $f(t\star)$.
However, for $t=\s_f$ we must take into account the presence and direction of a jump at time $\s_f$.
Note that $t\star$ with $t=\pm\infty$ are excluded from $L(f)$ and $R(f)$,
in accordance with the compactification of space-time in Figure \ref{fig:Rc}.

\begin{figure}[t]
\centering
\includegraphics{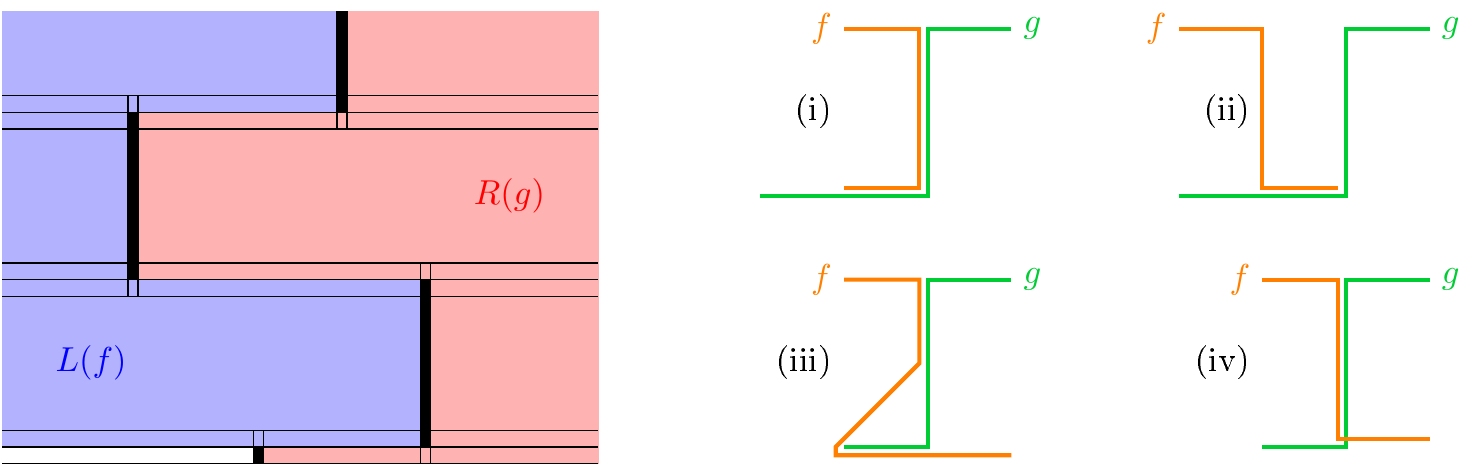}   
\small
\caption{
\textit{On the left:} 
A schematic depiction of the subsets $L(f)$ and $R(f)$ of $\ov\R\times\R_\mfs$. 
Time is running upwards, with values taken by $f(t\star)$ shown as a thick black line.
To help visualize, when $f$ jumps at time $t$ we depict time as split into $t-$ and $t+$
via thin horizontal lines.
The path $f$ makes a jump
at its starting time $\s_f$, at the very bottom of the figure.
Note that $L_{\s_f-}(f)=\emptyset$ and $R_{\s_f-}(f)=(f(\s_f-),\infty]$
because the jump at $\s_f$ is rightwards.\\
\textit{On the right:} 
Four examples of sets $\{f,g\}$ containing two paths. 
In each example, the initial point of $g$ lies to the left of the initial point of $f$. 
The paths in (i) satisfy $f\lhd g$ and do not cross,
but in examples (ii)--(iv) we have that $f$ crosses $g$ from right to left.
}
\label{fig:cross}
\end{figure}

\begin{lemma}
\label{l:lhd_LR}
Let $f,g\in\Pi^\uparrow$. Then
$f\lhd g$ if and only if $L(f)\cap R(g)=\emptyset$.
\end{lemma}
\begin{proof}
We prove the forwards and backwards implications in turn, beginning with the former.
Suppose that $f\lhd g$.
Then there exists $f',g'\in\Pi^\updownarrow$ with $f\sw f'$ and $g\sw g'$ such that $f'(t\star)\leq g'(t\star)$ for all $t\star\in \R_\mfs$.
It follows from \eqref{eq:LR_sets} that $L(f')\cap R(g')=\emptyset$, which since $L(f)\sw L(f')$ and $R(g)\sw R(g')$ implies that $L(f)\cap R(g)=\emptyset$.

For the reverse implication, let $f,g\in\Pi^\uparrow$ with $L(f)\cap R(g)=\emptyset$.
Let $\s=\s_f\vee\s_g$.
It follows immediately from \eqref{eq:LR_sets} that $f(t\star)\leq g(t\star)$ for all $t\star\geq \s+$.
We will construct explicit $f',g'\in\Pi^\updownarrow$ such that $f\sw f'$, $g\sw g'$ and $f'(t\star)\leq g'(t\star)$ for all $t\star\in\R_\mfs$.
Note that we have nothing to prove if $\s_f=\s_g=-\infty$,
and that we do not need to define $f'$ or $g'$ at $t\star$ for $t=\pm\infty$.
We consider three cases, at least one of which must occur.
\begin{enumerate}
\item
Consider if $L_{\s-}(f)=\emptyset$.
By \eqref{eq:LR_sets} we have $\s=\s_f\geq\s_g$ and $f(\s-)\leq f(\s+)$.
In this case set
$$
f'(t\star)=
\begin{cases}
f(t\star) & \text{ if }t\star\geq \s_f+ \\
-\infty & \text{ if }t\star\leq \s_f-
\end{cases}
\qquad\qquad
g'(t\star)=
\begin{cases}
g(t\star) & \text{ if }t\star\geq \s_g+ \\
g(\s_g-) & \text{ if }t\star\leq \s_g-.
\end{cases}
$$
\item
Consider if $R_{\s-}(g)=\emptyset$.
By \eqref{eq:LR_sets} we have $\s=\s_g\geq\s_f$ and $g(\s+)\leq g(\s-)$.
In this case set
$$
f'(t\star)=
\begin{cases}
f(t\star) & \text{ if }t\star\geq \s_f+ \\
f(\s_f-) & \text{ if }t\star\leq \s_f-
\end{cases}
\qquad\qquad
g'(t\star)=
\begin{cases}
g(t\star) & \text{ if }t\star\geq \s_g+ \\
\infty & \text{ if }t\star\leq \s_g-.
\end{cases}
$$
\item
Consider if $L_{\s-}(f)=[-\infty,f(\s-)$ and $R_{\s-}(g)=(g(\s-),\infty]$.
Using that $L(f)\cap R(g)=\emptyset$ we have $f(\s-)<g(\s-)$.
If $\s=\s_f\geq\s_g$ then we set
$$
f'(t\star)=
\begin{cases}
f(t\star) & \text{ if }t\star\geq \s_f+ \\
f(\s_f-)+g(t\star)-g(\s_f-) & \text{ if }t\star\in[\s_g+,\s_f-] \\
f(\s_f-)+g(\s_g-)-g(\s_f-) & \text{ if }t\star\leq \s_g- \\
\end{cases}
\qquad
g'(t\star)=
\begin{cases}
g(t\star) & \text{ if }t\star\geq \s_g+ \\
g(\s_g-) & \text{ if }t\star\leq \s_g-.
\end{cases}
$$
Note that $f'$ copies the increments of $g'$ during $t\star\in[\s_g+,\s_f-]$, 
backwards in time 
starting from the condition $f'(\s_f)=f(\s_f-)<g(\s_f-)=g'(\s_f)$,
and then backwards from time $\s_g-$ both paths remain constant.
This ensures that $f'(t\star)\leq g'(t\star)$ for all $t\star$.
If $\s=\s_g\geq\s_f$ then we may employ a similar strategy,
where $g$ copies the increments of $f$ during $[\s_f+,\s_g-]$, backwards in time starting from $g(\s_g-)$.
\end{enumerate}
In all cases it is clear that $f\sw f'$, $g\sw g'$ and $f'(t\star)\leq g'(t\star)$ for all $t\star\in\R_\mfs$.
\end{proof}

\begin{lemma}
\label{l:lhd_noncr}
Let $f,g\in\Pi^\uparrow$. 
The following statements are equivalent:
(i) $f$ and $g$ are non-crossing;
(ii) $L(f)\cap R(g)=\emptyset$ or $R(f)\cap L(g)=\emptyset$;
(iii) $f\lhd g$ or $g\lhd f$.
\end{lemma}
\begin{proof}
Equivalence of (ii) and (iii) follows immediately from Definition \ref{d:crossing} and Lemma \ref{l:lhd_LR}.
It follows trivially from Definitions \ref{d:crossing} and \ref{d:lhd} that (iii) implies (i).
Let us now show that (i) implies (ii).
If $f$ and $g$ are non-crossing then we have $f',g'\sw\Pi^\updownarrow$ such that $f\sw f'$, $g\sw g'$ and 
$f'(t\star)\leq g'(t\star)$ for all $t\star\in\R_\mfs$ or $g'(t\star)\leq f'(t\star)$ for all $t\star\in\R_\mfs$.
In the former case by \eqref{eq:LR_sets} we have $L(f')\cap R(g')=\emptyset$, which implies $L(f)\cap R(g)=\emptyset$.
In the latter case by \eqref{eq:LR_sets} we have $L(g')\cap R(f')=\emptyset$, which implies $L(g)\cap R(f)=\emptyset$.
Thus we have (ii).
\end{proof}

By Lemma \ref{l:lhd_noncr}, if $f$ and $g$ cross, then there must exist $t_1\star_1\neq t_2\star_2$ such that
$R_{t_1\star_1}(\pi)\cap L_{t_1\star_1}(\pi')\neq\emptyset$ and
$L_{t_2\star_2}(\pi)\cap R_{t_2\star_2}(\pi')\neq\emptyset$. If this happens
for $t_1\star_1<t_2\star_2$ (resp.\ $t_1\star_1>t_2\star_2$), then we say that
$f$ crosses $g$ \emph{from left to right} (resp.\ \emph{from right to
  left}). Of course, it can happen that $f$ crosses $g$ from left to
right and also from right to left. 
See Figure \ref{fig:cross} for a picture.

\begin{lemma}
\label{l:lhd_left_right}
Suppose that $f,g\in\Pi^\uparrow$ are non-crossing and
let $I_\mfs=I(f)_\mfs\cap I(g)_\pm$. 
If there exists $t\star\in I_\mfs$ such that $(f(t\star),t\star)\in L(g)$ then $f\lhd g$.
If there exists $t\star\in I_\mfs$ such that $(f(t\star),t\star)\in R(g)$ then $g\lhd f$.
Moreover, in either case $f\neq g$.
\end{lemma}
\begin{proof}
We will establish the first claim first:
let $f,g$ be as given and suppose $(f(t\star),t\star)\in L(g)$.
Hence $(f(t\star),t\star)\in L_{t\star}(g)$ so $L_{t\star}(g)=[-\infty, g(t\star))$
and $f(t\star)<g(t\star)$.
Let $\s=\s_f\vee \s_g$.
Consider if $s\bullet\geq\s+$ and $f(\s\bullet)<g(s\bullet)$
for some $s\bullet\geq \s-$.
Then, from Definition \ref{d:lhd}, we have $L_{s\bullet}(g)=[-\infty,g(s\bullet))$ and $R_{s\bullet}(f)=(f(s\bullet),\infty]$,
which implies $L(g)\cap R(f)\neq\emptyset$.
Lemma \ref{l:lhd_noncr} thus implies $L(f)\cap R(g)=\emptyset$,
from which Lemma \ref{l:lhd_LR} gives $f\lhd g$.
In particular, if $t\star\geq \s_f+$ then $f\lhd g$.

It remains only to consider the case of $t\star=\s_f-$
and, from what we have shown in the paragraph above, 
in this case we may assume without loss of generality that 
$f(s\bullet)\leq g(s\bullet)$ for all $s\bullet\geq\s+$.
We have $f(\s_f-)<g(\s_f-)$ and $L_{\s_f-}(g)=[-\infty,g(t\star))$.
If $f(\s_f-)<f(\s_f+)$ then $R_{\s_f-}(f)=[-\infty,f(\s_f-))$ and hence $L(g)\cap R(f)\neq\emptyset$, 
so here also Lemmas \ref{l:lhd_noncr} and \ref{l:lhd_LR} imply $f\lhd g$.
Otherwise, $f(\s_f+)\leq f(\s_f-)$, in which case it is immediate from Definition \ref{d:lhd} that $f\lhd g$.


The second claim, regarding the case $(f(t\star),t\star)\in R(g)$,
follows by symmetry (consider space $\ov{\R}$ reflected about the origin).
Lastly, the fact that $f\neq g$ follows from noting that points of the form $(g(t\star),t\star)\in\Rc$ are not elements of $L(g)$ or $R(g)$.
\end{proof}

\begin{lemma}
\label{l:lhd_possibilities}
The following hold:
\begin{enumerate}
\item
Suppose $f,g\in\Pi^\uparrow$ with $f\lhd g$ and write $\sigma=\sig_f\vee\sig_g$. Precisely one of following occurs: 
\begin{enumerate}[label=(\roman*)]
\item $f\sw g$ or $g\sw f$; 
\item $f(t\star)<g(t\star)$ for some $t\star\geq\s +$; 
\item $f(t\star)=g(t\star)$ for all $t\star\geq\s+$, $\;f(\s-)<f(\s+)=g(\s+)<g(\s-)$;
\item $f(t\star)=g(t\star)$ for all $t\star\geq\s+$, $\;f(\s-)<g(\s-)\leq f(\s+)=g(\s+)$ and $\s_g<\s_f=\s$;
\item $f(t\star)=g(t\star)$ for all $t\star\geq\s+$, $\;f(\s+)=g(\s+)\leq f(\s-)<g(\s-)$ and $\s_f<\s_g=\s$.
\end{enumerate}
In cases (ii)-(v)
we have $g\nlhd f$.

\item 
Let $f,g\in\Pi^\uparrow$. Then $f\lhd g$ and $g\lhd f$ if and only if $f\sw g$ or $g\sw f$.
\end{enumerate}
\end{lemma}
\begin{proof}
Let us begin with the first statement.
It is clear that all five cases are distinct.
Suppose neither of (ii), (iii) (iv) and (v) occurs,
and we will seek to prove that (i) holds.
Since (ii) fails we have $g(t\star)\leq f(t\star)$ for all $t\star\geq \s+$.
Since $f\lhd g$ we also have $f(t\star)\leq g(t\star)$ for all such $t\star$, hence in fact we have equality for $t\star\geq \s+$,
in particular at $t\star=\s+$.
Since (iii) fails, $f(t-)$ and $g(t-)$ lie (non-strictly) on the same side of $f(t+)=g(t+)$.

Consider first when they both lie to the left, that is $f(t-)\vee g(t-)\leq f(t+)=g(t+)$.
We divide into three cases
based upon whether $\s_f=\s_g$, $\s_f<\s_g$ or $\s_g<\s_f$. 
\begin{itemize}
\item If $\s_f=\s_g$ then $\s_f=\s_g=\s$, in which case (i) occurs.
\item If $\s_f<\s_g$ then $\s_g=\s$, so we have 
$L_{\s-}(f)=[-\infty,f(\s-))$ and $R_{\s-}(g)=(g(\s-),\infty]$.
By Lemma \ref{l:lhd_LR} we have $L_{\s-}(f)\cap R_{\s-}(g)=\emptyset$ so $f(\s-)\leq g(\s-)$. 
Hence $g\sw f$.
\item If $\s_g<\s_f$ then $\s_f=\s$ and as (iv) does not occur we must have $g(\s-)\leq f(\s-)$, which means $f\sw g$.
\end{itemize}
In all three cases we have that (i) occurs.
It remains to consider when both $f(\s-)$ and $g(\s-)$ lie to the right of $f(\s+)=g(\s+)$. 
A symmetric argument,
in which (v) takes the place of (vi), shows that (i) also occurs.
Note that in case (ii) we have $L_{t\star}(f)\cap R_{t\star}(g)\neq\emptyset$.
In cases (iii), (vi) and (v) we have $L_{\s-}(g)\cap R_{\s-}(f)\neq\emptyset$.
Thus, in all of cases (ii)-(v) Lemma \ref{l:lhd_LR} gives $g\nlhd f$.
For the second claim of the present lemma, 
the reverse implication is trivial from Definition \ref{d:lhd}
and the forwards implication follows from part 1 of the present lemma.
\end{proof}

Part 1 of Lemma \ref{l:lhd_possibilities} makes explicit 
the difficulties inherent to paths that may jump at their initial times.
We will often use arguments that
give some special attention to the initial times of paths.
The sets $L(f),R(g)$ and the relation $\lhd$
allow us to do so
without having to work through cases (i)--(v) in turn.

\begin{lemma}
\label{l:Amax_order} 
Let $A\sw\Pi^\uparrow$ be non-crossing. 
Then $(A_{\max},\lhd)$ is a totally ordered space.
\end{lemma}
\begin{proof}
Recall that $A_{\max}$ was defined in \eqref{eq:Amax_def}.
By Lemma \ref{l:lhd_noncr} all pairs $f,g\in A$ satisfy $f\lhd g$ or $g\lhd f$,
which holds in particular for $A_{\max}$.
It is clear from Definition \ref{d:lhd} that $f\lhd f$ for all $f\in\Pi^\uparrow$, 
thus also for all $f\in A_{\max}$.
If $f,g\in A_{\rm \max}$ satisfy $f\lhd g$ and $g\lhd f$, 
then Lemma~\ref{l:lhd_possibilities} tells us that $f\sw g$ or $g\sw f$. 
By maximality, this implies $f=g$. 
We have now shown that $\lhd$ is reflexive and antisymmetric,
and that all pairs of elements are comparable.
It remains to show that $\lhd$ is transitive.

Let $f,g,h\in A_{\max}$ with $f\lhd g$ and $g\lhd h$.
If $f=g$ or $g=h$ then it is trivial that $f\lhd h$.
If $f=h$ then we have $f\lhd g$ and $g\lhd h=f$, so Lemma \ref{l:lhd_possibilities}
implies that $f\sw g$ or $g\sw f$, which by maximality implies $f=g$, hence also $f\lhd g$.
Thus we may assume without loss of generality that $f,g,h$ are distinct elements of $\mc{A}_{\max}$. 
By Lemma \ref{l:lhd_LR} we have $L(f)\cap R(g)=\emptyset$ and $L(g)\cap R(h)=\emptyset$,
and we must show that $L(f)\cap R(h)$ is also empty.

We will argue by contradiction.
Suppose that $(x,t\star)\in L(f)\cap R(h)$,
which implies that $L_{t\star}(f)=[-\infty,f(t\star))$, $R_{t\star}(h)=(h(t\star),\infty]$ and $h(t\star)<x<f(t\star)$.
Consider first if $t\star\geq\s_g-$.
If $g(t\star)\leq x$ then $g(t\star)\in L_{t\star}(f)$, Lemma \ref{l:lhd_left_right} gives that $g\lhd f$,
in which case part 2 of Lemma \ref{l:lhd_possibilities} and maximality gives that $g=f$, which is a contradiction to our assumptions.
Similarly, if $g(t\star)\geq x$ then $g(t\star)\in R_{t\star}(h)$, Lemma \ref{l:lhd_left_right} gives that $g\lhd h$,
from which part 2 of Lemma \ref{l:lhd_possibilities} and maximality give $g=h$, which is again a contradiction.

It remains to consider when $t<\s_g$.
In this case $\s_f\vee\s_h<\s_g$.
Definition \ref{d:lhd} thus implies that for all $s\star\geq \s_g+$ we have $f(s\star)\leq g(s\star)\leq h(s\star)$.
Applying part 1 of Lemma \ref{l:lhd_possibilities} to $f\lhd g$ and $g\lhd h$,
and noting that these are distinct maximal paths (so case (i) of that lemma may not occur),
we obtain that for some $s,s'\geq \s_g$ it holds that $f(s-)<g(s-)$ and $g(s'-)<h(s'-)$.
If $s=s'=\s_g$ then $f(s-)<g(s-)<h(s-)$.
If $s>\s_g$ then $f(s-)<g(s-)\leq h(s-)$,
similarly if $s'>\s_g$ then $f(s'-)\leq g(s'-)<h(s'-)$.
In all three cases we have $u\geq\s$ such that $f(u-)<h(u-)$.
We have $\s_f\vee\s_h<\s_g\leq s$, so $L_{u-}(h)=[-\infty,h(u-))$ and $R_{u-}(f)=(f(u-),\infty]$,
meaning that $L(h)\cap R(f)\neq\emptyset$.
However now both $L(f)\cap R(h)$ and $L(h)\cap R(f)$ are non-empty, 
which by Lemma \ref{l:lhd_noncr} implies that $f$ and $h$ cross, 
which is contradiction.
\end{proof}

\begin{lemma}
\label{l:preceq_noncr}
Let $A$ be non-crossing and let $B$ be non-crossing,
both subsets of $\Pi^\uparrow$.
Suppose $A\preceq B$. 
Then $A\cup B$ is non-crossing.
\end{lemma}
\begin{proof}
Let $f\in A$, $g\in B$ and assume that $A\preceq B$.
Since $A\sw B_\uparrow$ there exists $h\in B_\uparrow$ such that $f\sw h$.
We have that $g,h\in B$ and $B$ is non-crossing,
so $\{g,h\}$ is non-crossing, thus also $\{f,h\}$ is non-crossing.
The result follows.
\end{proof}

\subsection{On partial orders of sets}
\label{sec:preceq}

For the duration of Section \ref{sec:preceq}
let $(E,\leq)$ denote a partially ordered set.
We now recall some standard notation associated to partial orders.
For $A\sw E$, we write 
$A_\leq:=\{e\in E:e\leq e'\mbox{ for some }e'\in A\}$
for the \emph{downset} of
$A$. Note that $(A_\leq)_\leq=A_\leq$. 
The \emph{upset} $A_\geq$ of $A$ is defined in the same way as the downset $A_\leq$, 
but for the reversed order. 
A \emph{maximal element} of a subset
$A\sw E$ is an element $e\in A$ such that there exists no $e'\in A$ with
$e<e'$. We write $A_{\rm max}=\{e\in A:e\mbox{ is a maximal
  element of }A\}$.
A \emph{minimal element} is defined in the same way 
but for the reversed order.
As usual, we write $e<e'$ if
$e\leq e'$ and $e\neq e'$.
 
\begin{remark}
For $A\in\mc{K}(\Pi)$, we have specified in \eqref{eq:Amax_def} that 
$A_{\max}$ refers to the maximal elements of $(A,\sw)$.
In
\eqref{eq:A_uparrow_def}
we defined the set $A_{\uparrow}$ to be the downset of $A\sw\Pi^\uparrow$ in $(\Pi^\uparrow,\sw)$.
In Remark \ref{r:A_max} we will note that
$A\sw (A_{\max})_\downarrow$ when $A$ is compact
i.e.~each path in $A$ extends to at least one element of $A_{\max}$.
Note that 
$B_\downarrow\sw\Pi^\downarrow$ from \eqref{eq:A_downarrow_def} is the downset of $B$ in $(\Pi^\downarrow,\sw)$.
\end{remark}

The following lemma puts the relation $\preceq$ introduced
in \eqref{eq:preceq} into a wider framework.
It is a natural concept that might have been studied elsewhere
but have not been able to locate a reference.
With slight abuse of notation we will briefly use the notation $\preceq$ in the more abstract setting.
We noted in Section \ref{sec:terminology} that $\preceq$ on $\Pi$ is related to how efficiently paths cover space.
In the abstract setting there is a somewhat clearer interpretation.
Specifically, under $\leq$ both of the following two operations will make a set $A\sw E$ strictly increase:
inserting a new element to $A$ that (once inserted) is a $\leq$-maximal element;
removing an existing element that (prior to removal) is not a $\leq$-maximal element.


\begin{lemma}
\label{l:preceq_deterministic}
Let $(E,\leq)$ be a partially ordered set. 
For $A,B\sw E$, write $A\preceq B$ to mean that $A_\leq\cap B\sw A\sw B_\leq$.
Then $\preceq$ is a partial order on the set of all subsets of $E$.
\end{lemma}

\begin{proof}
Clearly $A\preceq A$. Also, $ A\preceq B\preceq A$ implies
$ A\sw B_\leq\cap A\sw B$ and by a symmetric argument also
$ B\sw A$, so to complete the proof we must show that the relation $\preceq$
is transitive. The relations $A\preceq B\preceq C$ say that
\begin{equation*}
{\rm(i)}\ A_\leq\cap B\sw A,\quad
{\rm(ii)}\ A\sw B_\leq,\quad
{\rm(iii)}\ B_\leq\cap C\sw B,\quad
{\rm(iv)}\ B\sw C_\leq.
\end{equation*}
When we apply one of these facts we will indicate which
with a superscript above the corresponding $\sw$.
This implies (v) $ A_\leq\nmr{(ii)}{\sw} B_\leq$, (vi)
$ B_\leq\nmr{(iv)}{\sw} C_\leq$, and (vii)
$ B_\leq\cap C\nmr{(iii)}{\sw} B_\leq\cap B$, from which we get
\begin{equation*}
 A_\leq\cap C\nmr{(v),(vii)}{\sw} A_\leq\cap B\nmr{(i)}{\sw} A
\quand
 A\nmr{(ii)}{\sw} B_\leq\nmr{(vi)}{\sw} C_\leq,
\end{equation*}
proving that $ A\preceq C$.
\end{proof}


A set $A$ is said to be \textit{decreasing} if $A=A_\leq$
and \textit{increasing} if $A=A_\geq$.
If $A$ and $B$ are decreasing sets, then $A\preceq B$ if and only if $A\sw B$.
The proof is trivial and is left to the reader.

\subsection{On compatibility of order and topology}
\label{sec:compatibility}

We now turn our attention to the interaction between orders and limits.
We have now introduced several partial orders related to {\cadlag} paths:
the `path extension' order $\sw$ on $\Pi$,
the `coverage efficiency' order $\preceq$ on $\mc{K}(\Pi^\uparrow)$, 
and the `leftwards of' relation $\lhd$ that was shown in Lemma \ref{l:Amax_order} to be a total order on $A_{\max}$, when $A$ is closed.
In general $\lhd$ is not even a partial order.
With these situations in mind we make the following general definition.

\begin{defn}
\label{d:compatible}
Let $(E,\mathscr{T})$ be a topological space.
We say that a binary relation $\leq$ on $E$ is \emph{compatible} 
if $\big\{(e,f)\in E^2:e\leq f\big\}$ is a closed
subset of $E^2$ (in the product topology).
With mild abuse of notation we also apply this terminology to metric spaces $(E,d_E)$
where $d_E$ is a metric generating the topology on $E$.
\end{defn}

The point is that compatibility implies that if $e_n\to e$ and $f_n\to f$ with $e_n\leq f_n$,
all elements of $E$, then we may conclude that $e\leq f$.
We wish to study this notion in the three situations listed above, 
but as per our comments above we do not always wish to assume (in particular, for $\lhd$)
that $(E,\leq)$ is partially ordered.
Our first result in this direction is negative:

\begin{remark}
\label{r:preceq_not_compatible}
It is straightforward to see that $(\mc{K}(\Pi^\uparrow),d_{\mc{K}(\Pi)},\preceq)$ is not compatible.
For example,
let $f(t\pm)=0$ with $\sigma_f=0$ and let $g_n(t\pm)=1/n$ with $\sigma_{g_n}=1$.
Then $\{f\}\prec\{f,g_n\}$ and $\{f,g\}\prec\{f\}$.
Examples also exist where $A_n\prec B_n$ but the limits of $(A_n)$ and $(B_n)$ are incomparable;
we leave this as an exercise for the reader.
From a purely abstract point of view,
the non-trivial interaction of $\preceq$ with taking limits in $\mc{K}(\Pi^\uparrow)$
is the main source of interesting structure within the space of weaves.
\end{remark}

\begin{lemma}
\label{l:sw_compat} 
The partial order $\sw$ is compatible with $(\Pi, d_\Pi)$.
\end{lemma}

\begin{proof}
Proposition \ref{p:J1M1} gives that convergence in $d_\Pi$ is equivalent
to convergence under the metric $(f,g)\mapsto d_{\mc{K}((\Rc)^2)}(H^{(2)}(f),H^{(2)}(g))$.
By definition (see Section \ref{sec:terminology}) $f\sw g$ means that $H^{(2)}(f)\sw H^{(2)}(g)$.
With these facts in mind the stated result follows from Lemma \ref{l:sw_KM},
which asserts that the partial order of set inclusion is compatible with the Hausdorff metric.
\end{proof}

\begin{remark}
\label{r:A_max}
Using Lemma \ref{l:sw_compat} 
it is straightforward to check that if $A\in\mc{K}(\Pi)$ 
then for all $f\in A$ there exists $g\in A_{\max}$ such that $f\sw g$.
We will use this fact repeatedly, without referring back to this remark, from now on.
\end{remark}

We now consider the relation $\lhd$, which will require rather more work.
The following lemma shows that any lack of compatibility between $\lhd$ and $(\Pi^\uparrow,d_\Pi)$
must involve a jump at the initial time of a limiting path.
The reader should bear in mind examples like
$g_n(t\star)=\1\{t\star\geq\frac1n+\}$, $g(t\star)=\1\{t\star\geq0+\}$, $f(t)=\1\{t\star=0-\}$,
where $\s_{g_n}=\frac1n$ and $\s_f=\s_g=0$.
Note that $f\lhd g_n$ and $g_n\to g$, 
but $f$ and $g$ cross by jumping over each other in opposite directions at time $0$.
This example shows that $\lhd$ is not compatible with $(\Pi^\uparrow,d_\Pi)$.
\begin{lemma}
\label{l:lhd_psuedo_compat}
Let $f,g,f_n,g_n\in \Pi^\uparrow$ with $f_n\to f$ and $g_n\to g$.
Write $\s_n=\s_{f_n}\vee\s_{g_n}$ and $\s=\s_f\vee\s_g$.
Suppose that
$L_{s\bullet}(f_n)\cap R_{s\bullet}(g_n)=\emptyset$ 
for all $n\in\N$ and $s\bullet\geq \s_n+$.
Then also
$L_{s\bullet}(f)\cap R_{s\bullet}(g)=\emptyset$ 
for all $s\bullet\geq \s+$.
\end{lemma}
\begin{proof}
We will argue by contradiction.
Suppose that $L_{s\bullet}(f_n)\cap R_{s\bullet}(g_n)=\emptyset$ 
for all $n\in\N$ and $s\bullet\geq \s_n+$, 
and that $L_{t\star}(f)\cap R_{t\star}(g)\neq\emptyset$
where $t\star\geq\s+$.
By the {\cadlag} property of $f$ and $g$, without loss of generality we may take $t>\s$
and assume that both $f$ and $g$ are continuous at $t$.
Let $\eps\in(0,t-\sig)$
and $N\in\N$ be large enough that $|(\sigma_{f_n}\vee\sigma_{g_n})-\sig|\leq\eps/2$ for all $n\geq N$.
Our assumption that $L_{s\bullet}(f_n)\cap R_{s\bullet}(g_n)=\emptyset$  implies that
$f_n(s\bullet)\leq g_n(s\bullet)$ for all $s\bullet\geq\s_n+$,
which in particular for $N\geq n$ includes all $s\bullet\geq (t-\eps/2)+$.

Let $t_n\star_n\to t\star$.
It follows from Lemma \ref{l:Rpm}
that $t_n\star_n\geq(\s\vee\s_n)+$ for all sufficiently large $n$,
so let us pass to a subsequence and assume that this holds for all $n\in\N$.
By Lemma \ref{l:appdx_2_fntn} and continuity of $f,g$ at $t$ we have
$f_n(t_n\star_n)\to f(t\star)$ and $g_n(t_n\star_n)\to g(t\star)$.
As $g(t\star)<f(t\star)$ we obtain that there exists $\delta>0$ such that 
$g_n(t_n\star_n)\leq f_n(t_n\star_n)-\delta$.
For sufficiently large $n$ we have $|t_n-t|<\eps/2$, which implies $t_n\star_n\geq (t-\eps/2)+$.
This contradicts the result of the previous paragraph.
\end{proof}

\begin{lemma}
\label{l:lhd_compat_biinf}
\label{l:noncr_biinf_limits}
The relation $\lhd$ is compatible with $(\Pi^\updownarrow,d_\Pi)$.
Moreover: let $f_n,g_n,f,g\in\Pi^\updownarrow$ with $f_n\to f$ and $g_n\to g$.
If $\{f_n,g_n\}$ is non-crossing for each $n\in\N$, then $\{f,g\}$ is non-crossing.
\end{lemma}
\begin{proof}
The first claim follows from Lemmas \ref{l:lhd_psuedo_compat} and \ref{l:lhd_LR},
noting that for $f,g\in\Pi^\updownarrow$
the relation $f\lhd g$ is equivalence to requiring that
$L_{s+}(f)\cap R_{s+}(g)=\emptyset$ for all $s\in\R$.
For the second claim,
by part 2 of Lemma \ref{l:lhd_possibilities}
for each $n\in\N$  
we have that $f_n\lhd g_n$ or $f_n\lhd g_n$ .
At least one of these two possibilities must hold for infinitely many $n$.
From what we have already proved, it follows that $f\lhd g$ or $g\lhd f$,
from which Lemma \ref{l:lhd_noncr} gives that $f$ and $g$ do not cross.
\end{proof}

\begin{lemma}
\label{l:lhd_compat_cont}
\label{l:noncr_cont_limits} 
The relation $\lhd$ is compatible with $(\Pi^\uparrow_c,d_\Pi)$.
Moreover: let $f_n,g_n,f,g\in\Pi^\uparrow_c$ with $f_n\to f$ and $g_n\to g$.
If $\{f_n,g_n\}$ is non-crossing for each $n\in\N$, then $\{f,g\}$ is non-crossing.
\end{lemma}
\begin{proof}
The first claim follows from Lemmas \ref{l:lhd_psuedo_compat} and \ref{l:lhd_LR},
noting all $f\in\Pi^\uparrow_c$ satisfy $f(\s_f-)=f(\s_f+)$ so that $L_{\s_f-}(f)=R_{\s_f-}(f)=\emptyset$.
The proof of the second claim is essentially the same as that of Lemma \ref{l:lhd_compat_biinf}.
\end{proof}

The remainder of this section concerns conditions under which $\lhd$
is preserved in limits $f_n\lhd g_n$ with $f_n\to f$ and $g_n\to g$ for
half-infinite {\cadlag} paths.
We will see, in Lemma \ref{l:lhd_compat} that 
the key (extra) condition is that $\{f,g\}$ must be non-crossing.
From Lemma \ref{l:lhd_psuedo_compat} if any crossing is too occur in such a limit,
it must take place at time $\s=\s_f\vee\s_g$.
It is helpful to introduce another relation,
which quantifies the `amount that $f$ and $g$ cross by at $\s$',
whilst $f$ and $g$ are otherwise non-crossing.

\begin{defn}
\label{d:crossing_eps}
Let $\eps>0$.
Let $f,g\in\Pi^\uparrow$ and write $\s=\s_f\vee\s_g$.
We write $f \blacktriangleleft_\eps g$ if
$\s\in\R$ and
$L_{s\bullet}(f)\cap R_{s\bullet}(g)=\emptyset$ 
for all $s\bullet\geq \s+$, 
as well as 
$g(\s-)+\eps\leq f(\s-)$ and
$f(\s+)+\eps\leq g(\s+)$, with
$g(\s-)+\eps\leq g(\s+)$ and
$f(\s+)+\eps\leq f(\s-)$.
\end{defn}

\begin{lemma}
\label{l:crossing_eps}
Let $f,g\in\Pi^\uparrow$ and
let $\s=\s_f\vee\s_g$.
\begin{enumerate}
\item
If $f\blacktriangleleft_\eps g$ for some $\eps>0$
then $f$ and $g$ cross.
\item
Suppose that $f_n,g_n\in\Pi^\uparrow$ with $f_n\to f$ and $g_n\to g$.
If $f_n\lhd g_n$ for all $n$ and $f\nlhd g$ then 
there exists $\eps>0$ such that $f\blacktriangleleft_\eps g$.
\end{enumerate}
\end{lemma}
\begin{proof}
We prove the two claims in turn.
For the first, suppose that $f\blacktriangleleft_\eps g$.
Then $g(\s-)<g(\s+)$, $f(\s+)<f(\s-)$
and $g(\s-)<f(\s-)$.
Hence $L_{\s-}(f)\cap R_{\s-}(g)\neq\emptyset$.
Since $f(\s+)<g(\s+)$ we have $L_{\s+}(g)\cap R_{\s+}(f)\neq\emptyset$.
By Lemma \ref{l:lhd_noncr}, $f$ and $g$ cross.

For the second claim,
let $f_n\to f$, $g_n\to g$ with $f_n\lhd g_n$ and $f\nlhd g$.
From Lemma \ref{l:lhd_psuedo_compat} we have
$L_{t\star}(f)\cap R_{t\star}(g)=\emptyset$ and $f(t\star)\leq g(t\star)$ for all $t\star\geq\s+$.
Since $f\nlhd g$, by Lemma \ref{l:lhd_LR} we must have $L(f)\cap R(g)\neq\emptyset$,
which implies that $L_{\s-}(f)\cap R_{\s-}(g)\neq\emptyset$.
Therefore $L_{s-}(f)=[-\infty,f(\s-))$, $R_{\s-}(g)=(g(\s-),\infty]$
with $g(\s-)<f(\s-)$.

If $g(\s+)\leq g(\s-)$ then 
$f(\s+)\leq g(\s+)\leq g(\s-)<f(\s-)$
which implies $f\lhd g$, so this may not occur.
Similarly, if $f(\s-)\leq f(\s+)$ then $g(\s-)<f(\s-)\leq f(\s+)\leq g(\s+)$,
which implies $f\lhd g$, so this may not occur either.
Thus $g(\s-)<g(\s+)$ and $f(\s+)<f(\s-)$.

It remains only to show that $f(\s+)<g(\s+)$.
Recall that we have $f(\s+)\leq g(\s+)$,
so we need only eliminate the case $f(\s+)=g(\s+)$.
We will argue by contradiction.
We thus assume $g(\s-)<g(\s+)=f(\s+)<f(\s-)$.
Our strategy is to show that compactness must fail for $(f_n)$ or $(g_n)$,
because in avoiding crossing each other the paths $f_n$ and $g_n$ 
must become too erratic in a short time interval near $\s$.

Let 
\begin{equation}
\label{eq:crossing_eps_kappa}
\kappa=\min\{g(\s+)-g(\s-), f(\s-)-f(\s+)\}
\end{equation}
By right continuity of $f$ and $g$ at $\s+$,
there exists $\delta>0$ such that
\begin{equation}
\label{eq:crossing_eps_fg+}
|f(t\star)-f(\s+)|\vee|g(t\star)-g(\s+)|\leq \kappa/4
\qquad
\text{ for all }t\star\in[\s+,(\s+\delta)+].
\end{equation}

By Lemma \ref{l:appdx_3_time_approx} there exists
$s_n\bullet_n$ and $t_n\star_n$ such that
$s_n\to \s$, $t_n\to \s$ and
$f_n(s_n\bullet_n)\to f(\s-)$ and $g_n(t_n\star_n)\to g(\s-)$.
Without loss of generality
(or consider the following argument with the roles of $f_n$ and $g_n$ swapped)
we may assume that 
$s_n\bullet_n\leq t_n\star_n$ for infinite many $n\in\N$,
and let us pass to a subsequence upon which 
$s_n\bullet_n\leq t_n\star_n$ holds for all $n$.
In particular, $\s_n- \leq s_n\bullet_n$.
Without loss of generality we may pass to a further subsequence and assume that
$s_n\bullet_n \leq t_n\star_n \leq (\s+\delta/3)+$
for all $n\in\N$.
Let $u=\s+\delta.$
By Lemma \ref{l:appdx_3_time_approx} there exists
$u_n\diamond_n$ such that $u_n\to u$ and  $f_n(u_n\diamond_n)\to f(u+)$.
Without loss of generality we may pass to a further subsequence and assume that
$u_n\geq \s+2\delta/3$,
which implies that
\begin{equation}
\label{eq:crossing_eps_order_tsu}
s_n\bullet_n\leq t_n\star_n < u_n\diamond_n.
\end{equation}
Again, without loss of generality we may pass to a further subsequence and assume that
\begin{align*}
|f_n(s_n\bullet_n)-f(\s-)| &\leq \kappa/4 \\ 
|g_n(t_n\star_n)-g(\s-)| &\leq \kappa/4  \\
|f_n(u_n\diamond_n)-f(u+)| &\leq \kappa/4 
\end{align*}
for all $n$. It follows from the above equations,
\eqref{eq:crossing_eps_kappa} and \eqref{eq:crossing_eps_fg+} that
\begin{align}
f_n(s_n\bullet_n) &\geq f(\s+) +  3\kappa/4 \label{eq:crossing_eps_tn_control} \\
g_n(t_n\star_n) &\leq f(\s+) - 3\kappa/4 \label{eq:crossing_eps_sn_control_pre} \\
|f_n(u_n\diamond_n)-f(\s+)| &\leq \kappa/2 \label{eq:crossing_eps_un_control} 
\end{align}
for all $n$.

We must now briefly divide into two cases.
If $g_n(t_n-)<g_n(t_n+)$ then
$L_{t_n-}(g_n)=[-\infty,g_n(t_n-))$
and by Lemma \ref{l:lhd_left_right} we have 
$f_n(t_n-)\leq g_n(t_n-)<g_n(t_n+)$.
Alternatively, if $g_n(t_n+)\leq g_n(t_n-)$ then we have
$f_n(t_n+)\leq g_n(t_n+)<g_n(t_n-)$.
In either case we have $\circ\in\{-,+\}$ such that
$f_n(t_n\circ_n)\leq g_n(t_n-)\wedge g_n(t_n+) \leq g_n(t_n\star_n).$
From 
\eqref{eq:crossing_eps_sn_control_pre} we thus obtain
\begin{equation}
\label{eq:crossing_eps_sn_control}
f_n(t_n\circ_n) \leq f(\s+)-3\kappa/4.
\end{equation}
We must again briefly divide into two cases.
If $\star_n=-$ then
\eqref{eq:crossing_eps_order_tsu}
gives $s_n<t_n$, so trivially $s_n\bullet_n \leq t_n\circ_n \leq u_n\diamond_n$.
Alternatively, if $\star_n=+$ then
we have $f_n(t_n\star_n)\leq g_n(t_n\star_n)$,
which from \eqref{eq:crossing_eps_tn_control} and \eqref{eq:crossing_eps_sn_control_pre} we have
that $s_n\bullet_n\neq t_n\star_n$.
From \eqref{eq:crossing_eps_order_tsu} we thus have $s_n\bullet_n \leq t_n-$,
so in this case too we obtain that
\begin{equation}
\label{eq:crossing_times_control}
s_n\bullet_n \leq t_n\circ_n < u_n\diamond_n.
\end{equation}
From Proposition \ref{p:relcom_tightness} and
\eqref{eq:crossing_eps_tn_control},
\eqref{eq:crossing_eps_un_control},
\eqref{eq:crossing_eps_sn_control},
\eqref{eq:crossing_times_control}
we obtain that the sequence $(f_n)$ is not relatively compact.
This is a contradiction,
and completes the proof.
\end{proof}

\begin{lemma}
\label{l:lhd_compat} 
If $A\sw\Pi^\uparrow$ is non-crossing and closed then the relation $\lhd$ is compatible with $(A,d_\Pi)$.
Moreover: suppose that $f_n\to f$ and $g_n\to g$, where $f,g,f_n,g_n\in \Pi^\uparrow$ and $f_n\lhd g_n$ for all $n$.
If $f$ and $g$ do not cross each other then $f\lhd g$.
\end{lemma}
\begin{proof}
Note that the second claim is a stronger statement than the first.
The second claim permits paths within the sequence $(f_n)$ to cross each other,
and paths within the sequence $(g_n)$ to cross each other, requiring only that $f_n\lhd g_n$ for each $n\in\N$.
We will prove the second claim.
Let $f_n\to f$ and $g_n\to g$, where $f,g,f_n,g_n\in \Pi^\uparrow$ and $f_n\lhd g_n$ for all $n$.
By (both parts of) Lemma \ref{l:crossing_eps}, if $f\nlhd g$ then $f$ and $g$ cross.
The result follows.
\end{proof}

We commented below Definition \ref{d:lhd} that
the relation $\lhd$ is (in general) not a partial order.
In Lemma \ref{l:Amax_order} we showed that $\lhd$ is a total order on $A_{\max}$, provided $A\sw\Pi^\uparrow$ is non-crossing.
However, the set $A_{\max}$ is typically not a closed subset of $\Pi^\uparrow$,
even if $A$ is a closed subset of $\Pi^\uparrow$.
For example, consider when $A$ contains the paths $f(t)=k$ for $k\in[0,1]$ and $t\geq\sigma_f=0$, 
plus the single path $g(t)=0$ for $t\geq \sigma_g=-1$.
For this reason Lemmas \ref{l:Amax_order} and \ref{l:lhd_compat} are both important to us,
but we must take care when using them together.
The following technical lemma will be used in Section \ref{sec:meas_2}
to help prove that $\mathscr{W}_{\det}$ is measurable.

\begin{lemma}
\label{l:lhd_crossing_eps}
Let $\eps>0$.
Suppose that $f_n\to f$ and $g_n\to g$,
where $f,g,f_n,g_n\in\Pi^\uparrow$,
with $\s_f\vee\s_g\in \R$.
If $f_n\blacktriangleleft_\eps g_n$ for all $n$ then $f\blacktriangleleft_\eps g$.
\end{lemma}
\begin{proof}
We remark that the condition $\s_f\vee\s_g\in \R$ is necessary,
in fact this is all that prevents $\blacktriangleleft_\eps$
from being compatible with $(\Pi^\uparrow,d_\Pi)$.
Let $f,g,f_n,g_n\in \Pi^\uparrow$ be as given and 
write $\s=\s_f\vee\s_g$.
Suppose that $f_n\blacktriangleleft_\eps g_n$ for all $n$.
Lemma \ref{l:lhd_psuedo_compat} gives that 
$L_{s\bullet}(f)\cap R_{s\bullet}(g)=\emptyset$ for all $s\bullet\geq \s+$.
Noting that $\s_n\to s$, for all sufficiently large $n$ we have $\s_n\in\R$.
Let us pass to a subsequence and assume that $\s_n\in\R$ for all $n$.        

We have $g_n(\s_n-)+\eps\leq g_n(\s_n+)$.
It follows from Lemma \ref{l:appdx_1_sw_limits} that 
$g(\s-)+\eps\leq g(\s+)$.
Similarly, it follows from 
$f_n(\s_n+)+\eps\leq f_n(\s_n-)$ that
$f(\s+)+\eps\leq f(\s-)$.
From $g_n(\s_n-)+\eps\leq f_n(\s_n-)$ we obtain
\begin{equation}
\label{eq:lhd_crossing_eps_gn_eps_fn}
\liminf_{n\to\infty} g_n(\s_n-) + \eps \leq \limsup_{n\to\infty} f_n(\s_n-).
\end{equation}
By Lemma \ref{l:appdx_1_sw_limits} we have that
$\liminf_{n\to\infty} g_n(\s_n-)$ lies between $g(\s-)$ and $g(\s+)$.
We have already shown that $g(\s-)<g(\s+)$, 
so in fact 
$g(\s-)\leq \liminf_{n\to\infty} g_n(\s_n-)$.
Similarly, Lemma \ref{l:appdx_1_sw_limits} gives that
$\limsup_{n\to\infty} f_n(\s_n-)$ lies between $f(\s-)$ and $f(\s+)$.
We have already shown that $f(\s+)<f(\s-)$,
so in fact
$\limsup_{n\to\infty} f_n(\s_n-) \leq f(\s-)$.
From these facts and \eqref{eq:lhd_crossing_eps_gn_eps_fn} we obtain
$g(\s-)+\eps\leq f(\s-)$.

By compactness and Lemma \ref{l:appdx_2_fntn} 
the sequence $(f_n(\s_n+))$ has a limit point $a$ between $f(\s-)$ and $f(\s+)$.
We have shown that $f$ jumps leftwards at $\s$, thus $f(\s+)\leq a$.
Similarly, $(g_n(\s_n+))$ has a limit point $b$ between $g(\s-)$ and $g(\s+)$.
We have shown that $g$ jumps rightwards at $\s$, thus $b\leq g(\s+)$.
Using that $f_n(\s_n)+\eps\leq g_n(\s_n+)$ we obtain that
$a+\eps\leq b$, hence $f(\s+)+\eps\leq g(\s+)$.
This completes the proof.
\end{proof}

\section{Deterministic weaves}
\label{sec:weaves_det}

Recall the space $\mathscr{W}_{\det}$ introduced in \eqref{eq:Wdet}.
An element of $\mathscr{W}_{\det}$ is known as a \textit{deterministic weave} and is, 
by definition, a deterministic element of $\mc{K}(\Pi^\uparrow)$ that is pervasive and non-crossing.
We study deterministic weaves in this section,
although some results will involve probability within their proofs.
The results in this section will feed into the proofs of our main results, in Section \ref{sec:weaves_random}.
Our long term strategy is to establish what can be said in general about the internal structure of deterministic weaves,
to translate this information into statements about the geometric structure of $\mathscr{W}_{\det}$,
and finally lift such results into $\mathscr{W}$.

Definition \ref{d:weave} defines webs and flows as, respectively, minimal and maximal elements of the space of weaves $\mathscr{W}$ under $\preceqd$.
Recall that elements of $\mathscr{W}$ are formally
probability measures on $\mc{K}(\Pi)$.
We identify $\mathscr{W}_{\det}$ with the subset of $\mathscr{W}$
consisting of point-mass measures.
It is not immediately clear what Definition \ref{d:weave} means for deterministic weaves:
extremal points of $(\mathscr{W}_{\det},\preceq)$
are not a priori extremal points of $(\mathscr{W},\preceqd)$, nor vice versa
We will resolve these difficulties in Lemma \ref{l:det_vs_random_webs_flows},
which shows that
a random weave $\mc{W}$ is a web (resp.~flow) if and only if 
$\P\l[\mc{W}\text{ is almost surely minimal (resp.~maximal) in }(\mc{W}_{\det},\preceq)\r]=1$.
The proof of Lemma \ref{l:det_vs_random_webs_flows} will rely on key results established in Section \ref{sec:weaves_det}.
Therefore, until we have proved Lemma \ref{l:det_vs_random_webs_flows} we will avoid calling any deterministic or random elements of $\mc{K}(\Pi)$ a `web' or `flow'.

However, we will use the deterministic maps 
$A\mapsto\web_D(A)$ and $A\mapsto\flow(A)$ 
defined in \eqref{eq:web_op} and \eqref{eq:flow_op}
from this point on.
The meaning of these maps on $\mathscr{W}_{\det}$ is clear,
where in the former case $D\sw\R^2$ is also taken to be deterministic.


\subsection{Weaves and the non-crossing property}
\label{sec:weaves_noncr}

Weaves provide a structure inside of which {\cadlag} paths behave rather better than within arbitrary subsets of $\mc{K}(\Pi)$.
We remark that if $f\in\Pi^\uparrow$ does not cross a weave $\mc{A}$
then it is trivial to see that $\mc{A}\cup\{f\}$ is also a weave.
Combining these two facts,
to some extent weaves are able to control the behaviour of paths that do not cross them.
We begin to explore this idea within the present section.
We start with a key technical lemma that uses all of the defining properties of weaves: 
compactness, pervasiveness and the non-crossing property.
It captures what happens when we approximate the middle of a jump with paths beginning earlier in time.

\begin{lemma}
\label{l:approx_before_jump}
Let $\mc{A}$ be a deterministic weave
and let $f\in\Pi^\uparrow$ be a path that does not cross $\mc{A}$.
\begin{enumerate}
\item
Suppose $f(t-)<f(t+)$.
Then
there exists $h\in\mc{A}$
such that $h(t-)\leq f(t-)$ and $f\lhd h$.
\item
Suppose $f(t+)<f(t-)$.
Then
there exists $h\in\mc{A}$
such that $f(t-)\leq h(t-)$ and $h\lhd f$.
\end{enumerate}
\end{lemma}
\begin{proof}
The second statement follows from the first by considering space reflected about the origin
(and is written out in full for clarity)
so we will prove only the first statement.
Suppose $f(t-)<f(t+)$.
We will now argue that it suffices to prove that
\begin{equation}
\label{eq:approx_before_jump_x}
\text{for any }x\in(f(t-),f(t+)
\text{ there exists } h\in\mc{A}
\text{ such that }h(t-)\leq f(t-)\text{ and }f\lhd h.
\end{equation}
With \eqref{eq:approx_before_jump_x} in hand,
let us write $h^{(x)}$ for the path $h$ generated from $x$, 
and
note that compactness of $\mc{A}$ implies the existence of a subsequential limit
$h^{(x)}\to h'$ as $x\searrow f(t-)$.
As $f\lhd h^{(x)}$ we have $h^{(x)}(t-)\leq f(t-)<f(t+)\leq h^{(x)}(t+)$,
so
Lemma \ref{l:appdx_1_sw_limits} ensures that $h'(t-)\leq f(t-)$.
Lemma \ref{l:lhd_compat} ensures that $f\lhd h'$.
Thus $h'$ has the desired properties.

It remains to establish \eqref{eq:approx_before_jump_x}.
Let $x'\in\R$ be such that $f(t-)<x<f(t+)$,
and suppose that $f$ does not cross $\mc{A}$.
Let $(a_n)\sw(0,\infty)$ be such that $a_n\to 0$.
Let $z_n=(x,t-a_n)$ and by pervasiveness let $h_n\in\mc{A}(z_n)$.
By compactness of $\mc{A}$, pass to a subsequence and assume without loss of generality that $h_n\to h\in\mc{A}$.
By Lemma \ref{l:appdx_2_fntn} we have 
\begin{equation}
\label{eq:hhxhh}
h(t-)\wedge h(t+)\leq x\leq h(t-)\vee h(t+).
\end{equation}
We have that $f$ and $h_n\in\mc{A}$ do not cross.
By Lemma \ref{l:lhd_noncr} this means that $f\lhd h_n$ or $h_n\lhd f$.
We now consider two cases.

Consider first if $f\lhd h_n$ for infinitely many $n\in\N$.
Then Lemma \ref{l:lhd_compat} implies that $f\lhd h$.
In this case $f(t+)\leq h(t+)$, so $x<h(t+)$ which by \eqref{eq:hhxhh} implies $h(t-)\leq x'$. 
Hence $h(t-)\leq y$, and we have established all the required properties of $h$.

If the above case does not occur then there exists $N\in\N$ such that $h_n\lhd f$ for all $n\geq N$.
Without loss of generality we may pass to a subsequence and assume $h_n\lhd f$ for all $n\in\N$.
As $\s_{h_m}<t$ and $L_{t-}(f)=[-\infty,f(t-))$,
Lemma \ref{l:lhd_left_right}
implies that $h_n(t-)\leq f(t-)$.
We have also that $h_n((t-a_n)-)\vee h_n((t-a_n)+)\geq x$.
Clearly $(t-a_n)\pm<t-$ for all $n$.
Hence by Lemma \ref{l:appdx_2_fntn}
we have 
$h(t-)\geq x>f(t-)\geq h(t+),$
which as $f(t-)<f(t+)$ means that $f$ and $h$ cross (by Definition \ref{d:crossing}).
This is a contradiction, so in fact this case does not occur.
This completes the proof.
\end{proof}

\begin{lemma}
\label{l:pincer_t_plus}
Let $\mc{A}$ be a deterministic weave.
Let $f,h\in\Pi^\uparrow$ be such that $\mc{A}\cup\{f\}$ is non-crossing, and $\mc{A}\cup\{h\}$ is non-crossing.
Let $\s=\s_f\vee\s_h$.
Suppose that $f(t\star)<h(t\star)$ for some $t\star\geq\s+$.
Then for all $\eps>0$ 
there exists $g\in\mc{A}$ with $\s_g\leq t+\eps$ such that $f\lhd g$ and $g\lhd h$.

Further, we may choose $g$ such that $g\nsubseteq f$, $f\nsubseteq g$, $g\nsubseteq h$ and $h\nsubseteq g$.
\end{lemma}
\begin{proof}
By the {\cadlag} property of $f$ and $h$
for any $\eps>0$ there exists $t'$ such that $|t-t'|<\eps/2$ and $f(t'+)<h(t'+)$.
Hence we may choose $\eps>0$
chosen sufficiently small that $f(s\pm)<h(s\pm)$ for all $s\in[t',t'+\eps/2)$,
with $|t-t'|<\eps/2$ and $t'>\s$.
By the {\cadlag} property of $f$ and $h$, 
choose $s\in(t',t'+\eps)$ such that both $f$ and $h$ are continuous at $s$.
Note that $s\leq t+\epsilon$.
Recall that when $f$ is continuous at $s$ we write $f(s)=f(s\pm)$.
Let $x\in(f(s),h(s))$ and by pervasiveness of $\mc{A}$ let $g\in\mc{A}(x,s)$.
Hence $g(s-)\wedge g(s+)\leq x\leq g(s-)\vee g(s+)$.
By continuity of $f$ at $s$ we have $L_{s-}(f)=L_{s+}(f)=[-\infty,f(s))$.
Note that $f(s)<h(s)$ implies that at least one of $g(s-)$ and $g(s+)$ is strictly greater than $f(t\star)$.
By Lemma \ref{l:lhd_left_right} we thus have $f\lhd g$,
and as $\s_f<s$ this means $f\nsubseteq g$ and $g\nsubseteq f$.
A symmetrical argument (reflect space about the origin) shows that $g\lhd h$,
with $g\nsubseteq h$ and $h\nsubseteq g$.
\end{proof}

\begin{lemma}
\label{l:noncr_transitive_weave}
Let $\mc{A}\sw\Pi^\uparrow$ be a deterministic weave. If $B, C\sw\Pi^\uparrow$ are such that $\mc{A}\cup B$ is non-crossing, and $\mc{A}\cup C$ is non-crossing,
then $\mc{A}\cup B\cup C$ is non-crossing.
\end{lemma}
\begin{proof}
Note that our conditions imply that $B$ is non-crossing, and $C$ is non-crossing.
It suffices to prove the case where $\mc{B}=\{f\}$ and $\mc{C}=\{h\}$ are singletons,
from which the general case follows immediately.
To this end, suppose that $f,g\in\Pi^\uparrow$ are such that $\mc{A}\cup\{f\}$ is non-crossing and $\mc{A}\cup\{h\}$ is non-crossing.

We will argue by contradiction.
Suppose that $\mc{A}\cup\{f\}\cup\{h\}$ contains a pair of paths that cross.
From our assumptions, the only possibility is that
$f$ and $h$ cross.
By Lemma \ref{l:lhd_noncr}
$f$ and $h$ cross if and only if $L(g)\cap R(f)\neq\emptyset$ and $L(f)\cap R(h)\neq\emptyset$.
Hence there exists $t\star,s\bullet\in\R_\mfs$ such that
$t\star<s\bullet$, with
\begin{alignat*}{2}
L_{t\star}(h)&=[-\infty,h(t\star)), \qquad 
R_{t\star}(f)&&=(f(t\star),\infty], \\
L_{s\bullet}(f)&=[-\infty,f(s\bullet)), \qquad
R_{s\bullet}(h)&&=(h(s\bullet),\infty],
\end{alignat*}
$f(t\star)<h(t\star)$ and $h(s\bullet)<f(s\bullet)$.
Let $\s=\s_f\vee\s_h$.

Consider, first, if $t\star\geq\s+$.
In this case, 
Lemma \ref{l:pincer_t_plus} implies that there exists $g\in\mc{A}$
such that $f\lhd g$ and $g\lhd h$, with $\s_g<s\bullet$.
As $s\bullet\geq\s+$ this means $f(s\bullet)\leq g(s\bullet)\leq h(s\bullet)$, 
which is a contradiction to $h(s\bullet)<f(s\bullet)$.

It remains to consider the case $t\star=\s-$.
In this case $\s_f=\s$ or $\s_h=\s$.
Without loss of generality (or consider space reflected about the origin)
let us assume that $\s_f=\s$.
As $R_{\s-}=(f(\s-),\infty]$ 
and $\s_f=\s$ it follows from \eqref{eq:LR_sets} that 
$f(\s-)<f(\s+)$.
Lemma \ref{l:approx_before_jump} implies the existence of $g\in\mc{A}$
such that $f\lhd g$ and $g(\s-)<f(\s+)$.
Hence $g(\s-)<h(\s-)$.
If $\s_h<\s$ then it follows immediately
by Lemma \ref{l:lhd_left_right} that $g\lhd h$.
Alternatively, 
if $\s_h=s$ then $L_{t\star}(h)=[-\infty,h(\s-))$
so in this case too we have $g\lhd h$.
We now have $f\lhd g$ and $g\lhd h$.
As $s\bullet\geq\s+$ this means $f(s\bullet)\leq g(s\bullet)\leq h(s\bullet)$, 
which is a contradiction to $h(s\bullet)<f(s\bullet)$.
\end{proof}

\begin{lemma}
\label{l:noncr_two_weaves}
Let $\mc{A},\mc{B}$ be deterministic weaves and suppose that $\mc{A}\cup\mc{B}$ is non-crossing.
Let $C\sw\Pi^\updownarrow$ be non-crossing. 
Then $\mc{A}\cup C$ is non-crossing if and only if $\mc{B}\cup C$ is non-crossing.
\end{lemma}
\begin{proof}
It suffices to consider the case $C=\{f\}$ where $f\in\Pi^\updownarrow$,
from which the general case follows immediately.
Assume that $f\in\Pi^\updownarrow$ does not cross $\Ai$.
Let $b\in\Bi$.
By assumption $b$ does not cross $\Ai$.
By Lemma \ref{l:noncr_transitive_weave}, we have that $\mc{A}\cup\{b\}\cup\{f\}$ is non-crossing, 
so in particular $f$ and $b$ do not cross each other.
Since $b\in\Bi$ was arbitrary, $f$ does not cross $\mc{B}$.
\end{proof}

In the proof of Lemma \ref{l:noncr_transitive_weave} we saw that
Lemma \ref{l:pincer_t_plus}
was a natural counterpart to Lemma \ref{l:approx_before_jump}.
The underlying principle is as follows.
If two paths $f$ and $g$ are such that $f\nsubseteq g$ and $g\nsubseteq f$ then either:
$f(t+)\neq g(t+)$ for some $t$,
or at least one of $f$ and $g$ has made a jump at its initial time,
in a direction away from the other.
Lemma \ref{l:pincer_t_plus} applied to the former case,
Lemma \ref{l:approx_before_jump} to the latter,
resulting in a path $h$ that lay between $f$ and $g$.
Variations upon this theme will feature in the proof of several future results, 
including the next lemma.

The following lemma is stated for paths $f,h\in\Pi^\uparrow$ that do not cross a weave $\mc{A}$,
but at this stage it is perhaps best understood by considering the special case $f,h\in\mc{A}_{\max}$.
Note that for maximal paths, 
the condition $f\nsubseteq h$ and $h\nsubseteq f$ is simply the requirement that $f\neq h$.
In this case Lemma \ref{l:lhd_intermediates} provides a key piece of information about the geometric structure of $(\mc{A}_{\max}, \lhd)$,
namely that any two distinct points, within the total order, will always have another point strictly in between them.
This lemma will be a key tool in Section \ref{sec:path_extn}.

\begin{lemma}
\label{l:lhd_intermediates}
Let $\mc{A}$ be a deterministic weave.
Suppose $f,h\in\Pi^\uparrow$ do not cross $\mc{A}$,
with $f\lhd h$, $f\nsubseteq h$ and $h\nsubseteq f$.
Then there exists $g\in\mc{A}_{\max}$ such that $f\lhd g$ and $g\lhd h$,
and $f\nsubseteq g$, $g\nsubseteq f$, $g\nsubseteq h$, $h\nsubseteq g$.
\end{lemma}
\begin{proof}
Let $f,h$ be as given in the lemma and set $\s=\s_f\vee \s_h$.
By part 2 of Lemma \ref{l:lhd_possibilities}
our conditions on $f$ and $h$ imply
that $h\nlhd f$.
From Lemma \ref{l:lhd_LR} we thus have
$L(h)\cap R(f)\neq\emptyset$.
In particular there exists $t\star\geq\s-$ such that
$L_{t\star}(h)=[-\infty,h(t\star))$ and $R_{t\star}(f)=(f(t\star),\infty]$ with
$f(t\star)<h(t\star)$.

Consider first if $t\star\geq\s+$. Then Lemma \ref{l:pincer_t_plus} implies the existence of
$g\in\mc{A}$ with $f\lhd g$ and $g\lhd f$,
also $f\nsubseteq g$, $g\nsubseteq f$, $g\nsubseteq h$ and $h\nsubseteq g$.
Without loss of generality we may take $g\in\mc{A}_{\max}$,
which completes the proof in this case.

It remains to consider the case $t\star=\s-$.
In this case we have $\s=\s_f$ or $\s=\s_h$.
Without loss of generality let us assume that $\s=s_f$
(or consider space reflected about the origin).
As $R_{\s-}(f)$ is non-empty this implies that $f(t-)<f(t+)$.

\begin{remark}
Let us briefly comment on the strategy for the remainder of the proof.
Although $f$ jumps at $\s$, Lemma \ref{l:approx_before_jump} is not suitable for use here
because (if used to construct $g$) it allows the possibility that $g\sw f$.
Instead, we require a more sophisticated version of the approximation scheme 
used in the proof of Lemma \ref{l:approx_before_jump},
but the path we are looking for here is \textit{not} the limiting path;
rather it is some path that occurs sufficiently close to the limit.
We require a path with several different properties.
To find it, we will repeatedly show that one such desired property can fail only for finitely many $n$,
then (without loss of generality) pass to a subsequence on which the property holds for all $n$.
\end{remark}

As $t\star=\s-$ we have $f(\s-)<h(\s-)$, so 
\begin{equation}
\label{eq:intermediate_ffh}
f(\s-)\;<\;f(\s+)\wedge h(\s-).
\end{equation}
Let $\eps>0$ be such that $f(\s-)+2\eps\leq f(\s+)\wedge h(\s-)$
and let $(a_n)\sw(0,\infty)$ be such that $a_n\to 0$.
Let $(y,s_n)=(f(\s-)+\eps,\s-a_n)$
and note that $(y,s_n)\to (y,\s)$ where $y=f(\s-)+\eps$.
By pervasiveness of $\mc{A}$ let $g_n\in\mc{A}((y,s_n))$
so that
\begin{equation}
\label{eq:intermediate_gngnygngn}
g_n(s_n-)\wedge g_n(s_n+) \;\leq\; y \;\leq\; g_n(s_n-)\vee g_n(s_n+).
\end{equation}
Without loss of generality we may take $g_n\in\mc{A}_{\max}$.
By compactness of $\mc{A}$ we may pass to a subsequence and assume that $g_n\to g\in\mc{A}$.

Consider if $f\sw g_n$ for infinitely many $n$.
For such $n$, noting from \eqref{eq:intermediate_ffh} that $f$ jumps rightwards at $\s$, we have $g_n(\s-)\leq f(\s-)$.
From \eqref{eq:intermediate_gngnygngn} we have $g_n(s_n-)\vee g_n(s_n+)\geq y=f(\s-)+\eps$,
and $s_n\pm<\s-$ with $s_n\to\s$,
so Lemma \ref{l:appdx_2_fntn} gives that
$g$ jumps leftwards at $\s$, from right of $f(\s-)+\eps$ to left of $f(\eps)$.
This would make $f$ and $g\in\mc{A}$ cross, which is a contradiction.
Hence in fact $f\sw g_n$ for at most finitely many $n$,
so we may pass to a subsequence and assume 
$f\nsubseteq g_n$ for all $n$.
If $g_n\sw f$ then we would have $\s_{g_n}\geq\s_f$, which is not the case because $\s_{g_n}\leq s_n<\s_f$.
Hence we also have $g_n\nsubseteq f$ for all $n$.

Consider if $g_n\lhd f$ for infinitely many $n\in\N$.
For such $n$, noting that we have $f\nsubseteq g_n$ and $g_n\nsubseteq f$,
by part 2 of Lemma \ref{l:lhd_possibilities} we have $f\nlhd g_n$.
Hence, for such $n$, using that $R_{\s-}(f)=(f(\s-),\infty]$,
by Lemma \ref{l:lhd_left_right} we must have $g_n(\s-)\leq f(\s-)$.
From \eqref{eq:intermediate_gngnygngn} we have  $y\leq g_n(s_n-)\vee g_n(\s_n+)$, where $s_n<\s$ with $s_n\to \s$.
By Lemma \ref{l:appdx_2_fntn}, taking a limit along a subsequence of such $n$ would result in
$g(\s-)\geq y>f(\s-)\geq g(\s+)$,
in which case $f$ and $g$ cross (by jumping over each other in opposite directions at time $\s$).
This may not occur.
Hence in fact $g_n\lhd f$ for at most finitely many $n$.
By Lemma \ref{l:lhd_noncr} for all $n$ we have $f\lhd g_n$ or $g_n\lhd f$.
We may thus pass to a subsequence and assume that $f\lhd g_n$ for all $n$.

We now have $f\lhd g_n$, $f\nsubseteq g_n$ and $g_n\nsubseteq f$.
We will move on to establishing properties of $g_n$ with $h$.
Here we divide into two cases, based upon whether $h(\s-)\leq h(\s+)$ or $h(\s+)<h(\s-)$.
\begin{itemize}
\item
Firstly, consider if $h(\s+)<h(\s-)$.

Consider if $h(\s-)\leq g_n(\s-)$ for infinitely many $n$.
From \eqref{eq:intermediate_gngnygngn} we have $g_n(s_n-)\wedge g_n(s_n+)\leq y$ and $s_n\pm<\s-$ with $s_n\to\s$,
so by Lemma \ref{l:appdx_2_fntn} we obtain that
$g$ jumps rightwards at $\s$, from left of $y$ to right of $h(s-)$.
This means that $g\in\mc{A}$ crosses $h$, which is a contradiction.
Therefore we may pass to a subsequence and assume that $g_n(\s-)<h(\s-)$ for all $n$.

We have $\s_g\leq s_n<\s$ and $\s_h\leq \s$.
If $h\sw g_n$ then, noting that $h$ jumps leftwards at $\s$, we would have $h_n(\s-)\leq g_n(\s-)$,
which is a contradiction.
Hence $h\nsubseteq g_n$.
Similarly, if $g_n\sw h$ then $\s_h\leq\s_g<\s$, so we would have $g_n(\s-)=h_n(\s-)$,
which is a contradiction, so $g_n\nsubseteq h$.

Consider if $h\lhd g_n$. 
We have already seen that $h\nsubseteq g_n$ and $g_n\nsubseteq h$,
so by part 2 of Lemma \ref{l:lhd_possibilities} we have $g_n\nlhd h$.
We have $L_{\s-}(h)=[-\infty,h(\s-))$ so Lemma \ref{l:lhd_left_right} gives that $h(\s-)\leq g_n(\s-)$,
which again may not occur.
Hence in fact $g_n\lhd h$.

\item
Secondly, consider if $h(\s-)\leq h(\s+)$.

Our assumption $f\lhd h$ implies that $f(\s\bullet)\leq h(s\bullet)$ for all $s\bullet\geq\s+$,
so from what we have already proved (in the case $t\star\geq \s+$) 
we may assume without loss of generality that
$f(s\bullet)=h(s\bullet)$ for all $s\bullet\geq\s+$.
Using that $h(\s-)\leq h(\s+)$ we must therefore have $\s_h<\s$, 
as otherwise by \eqref{eq:intermediate_ffh} we would have $h\sw f$.

By left continuity of $h$ at $\s-$ there exists some $\delta>0$ such that
$h(s\bullet)\geq h(\s-)-\eps/2$ for all $s\bullet\in[\s-,(\s-\delta)+]$.
From \eqref{eq:intermediate_gngnygngn} for all sufficiently large $n$ we have
$s_n\in(\s-\delta,s)$
and
\begin{equation}
\label{eq:intermediate_gngnyh}
g_n(s_n-)\wedge g_n(s_n+) \;\leq\; y \;<\; h(\s-)-\eps/2 \;\leq\; h(s_n\pm).
\end{equation}
It follows immediately that, for such $n$, $g_n\nsubseteq h$ and $h\nsubseteq g_n$.
For sufficiently large $n$ we also have $s_n>\s_h$,
in which case \eqref{eq:intermediate_gngnyh} gives $g_n(s_n-)\in L_{s_n-}(h)$ or $g_n(s_n+)\in L_{s_n+}(h)$.
Hence Lemma \ref{l:lhd_left_right} gives $g_n\lhd h$.
We may thus pass to a subsequence and assume that 
$g_n\nsubseteq h$, $h\nsubseteq g_n$ and $g_n\lhd h$ for all $n$.
\end{itemize}
In both cases we have now shown,
for all $g_n$ within the subsequence that we have passed into,
that
$f\lhd g_n$, $g_n\lhd h$, 
and that none of $f,g_n,h$ are $\sw$-comparable with each other.
Therefore, any such $g_n$ has the required properties 
and the proof of the present lemma is complete.
\end{proof}

\subsection{Weaves of bi-infinite paths}
\label{sec:biinf_weaves}

In this section we establish some geometric properties of deterministic weaves that comprise entirely of bi-infinite paths.
One key result is that if $\mc{A}\sw\Pi^\updownarrow$ is a deterministic weave then set of ramification points of $\mc{A}$ has zero (two dimensional) Lebesgue measure.
This result will be later extended to all deterministic weaves, in Lemma \ref{l:ramification_meas_zero}.
Whenever we refer to Lebesgue measure in this section, we mean two dimensional Lebesgue measure on $\Rc$.

\begin{lemma}
\label{l:Hmeas} 
For each $f\in\Pi$, the set $H(f)$ has zero Lebesgue measure
\end{lemma}

\begin{proof}
This lemma is almost self-evident but
in view of the example of Jordan curves with positive Lebesgue measure
we will give a short proof.
Since {\cadlag} functions have only countably many discontinuities, 
the result holds for the part of $H(f)$ corresponding to jumps of $f$. 
The remaining part of $H(f)$ can be shown to have zero Lebesgue measure via Fubini's theorem.
We leave the details to the reader.
\end{proof}

\begin{lemma}
\label{l:biinf_lr_approx}
Let $\mc{A}\sw\Pi^\updownarrow$ be a deterministic weave and $h\in \mc{A}$.
\begin{enumerate}
\item
If $h(t\star)>-\infty$ for some $t\star\in\R_\mfs$ then
there exist a strictly monotone sequence $(f_n)$ with $f_n\lhd f_{n+1}$ such that
$f_n\to h$.
\item
If $h(t\star)<\infty$ for some $t\star\in\R_\mfs$ then
there exist a strictly monotone sequence $(g_n)$ with $g_{n+1}\lhd g_{n}$ such that
$g_n\to h$.
\end{enumerate}
\end{lemma}
\begin{proof}
We will show only the existence and properties of $(f_n)$.
The corresponding statements for $(g_n)$ follow by symmetry.
By Lemma \ref{l:Amax_order} $(\mc{A},\lhd)$ is totally ordered.
Let $h\in\mc{A}$ and set $L=\{f'\in\mc{A}\-f'\lhd h, f'\neq h\}$.
As $h(\star)>-\infty$,
by the {\cadlag} property of $h\in\Pi^\updownarrow$ there exists some $s\in\R$ such that 
$h$ is continuous at $s$ and $h(s)>-\infty$.
Taking $g\in\mc{A}((h(s)/2,s))$ gives $g\in L$, so $L$ is non-empty.
By compactness of $\mc{A}$ the set $\ov{L}$ is compact, 
which by Lemma \ref{l:lhd_compat} 
implies that $(\ov{L},\lhd)$ contains a unique maximal element $f$.
By Lemma \ref{l:lhd_compat} we have $f\lhd h$.

Suppose, in preparation for an argument by contradiction, that $f\neq h$.
Since $f,h$ are both bi-infinite there exists $t\in\R$ such that $f(t+)<h(t+)$.
By Lemma \ref{l:pincer_t_plus} and using that $\mc{A}\sw\Pi^\updownarrow$, 
there exists $g\in\mc{A}$ such that $f\neq g$ and $g\neq h$.
This is a contradiction to maximality of $f$ in $\ov{L}$.

We thus have $f=h$, which by definition of $L$ implies that there exists $(f_n)\sw L$ such that $f_n\to h$,
with $f_n\neq h$ for all $n$.
Without loss of generality we may choose a strictly monotone subsequence, which completes the proof.
\end{proof}

Lemma \ref{l:biinf_lr_approx} fails for general deterministic weaves, 
which may contain paths that are isolated points
(from the left, right or both).
For example see the weave $\mc{A}$ on the right hand side of Figure \ref{fig:web_dc}.
We will shortly show, as a consequence of Lemma \ref{l:biinf_vs_01},
that if a deterministic weave consists entirely of bi-infinite paths
then it does not contain any isolated points.

\begin{lemma}
\label{l:lhd_vs_M1}
Let $\mc{A}\sw\Pi^\updownarrow$ be a deterministic weave. 
The order topology induced on $\mc{A}$ by the total order $\lhd$ coincides with its topology as a subspace of $\Pi$.
\end{lemma}
\begin{proof}
By Lemma \ref{l:Amax_order}, $(\mc{A},\lhd)$ is totally ordered.
Recall that the order topology on $\mc{A}$ is generated by the open rays
\begin{align*}
R_f&=\{g\in\mc{A}\-g\lhd f\text{ and }f\neq g\}, \\
R'_f&=\{g\in\mc{A}\-f\lhd g\text{ and }f\neq g\}
\end{align*}
where $f\in\mc{A}$.
We will show that $R_f$ is open in the M1 topology on $\mc{A}$. 
The same result follows for $R'_f$ by a symmetrical argument.
Note that if $f$ is the bi-infinite path with constant value at $-\infty$ then 
$R_f=\emptyset$ and $R'_f=\mc{A}\sc\{f\}$,
which are automatically open.
Similar considerations apply to if $f$ is the bi-infinite path with constant value $\infty$.
We may therefore restrict to $f\in\mc{A}$ such that $f(t\star)>-\infty$ for some $t\star\in\R_\mfs$.

By Lemma \ref{l:Amax_order} $\mc{A}$ is totally ordered, from which it follows that
$\mc{A}\sc R_g=\{f\in\mc{A}\-g\lhd f\}$.
By Lemma \ref{l:lhd_compat} this is a closed subset of $\mc{A}$ in the M1 topology,
thus $R_g$ is an open subset of $\mc{A}$ in the M1 topology.
It follows that any subset of $\mc{A}$ that is open in the order topology is also open in the M1 topology,
and it remains to prove the converse.

It suffices to show that if $B\sw\mc{A}$ is closed in the M1 topology, then it is also closed in the order topology.
Let us write $\stackrel{M1}{\to}$ for convergence in the M1 topology and $\stackrel{\lhd}{\to}$ for convergence in the order topology.
Let $B\sw\mc{A}$ be closed in the M1 topology i.e.~if $f_n\in B$ and $f_n\stackrel{M1}{\to} f\in\mc{A}$ then $f\in B$.
Suppose that $h_n\in B$ and $h_n\stackrel{\lhd}{\to} h\in\mc{A}$.
By Lemma \ref{l:biinf_lr_approx} there exists $f_n,g_n\in\mc{A}$ such that $f_n\lhd f_{n+1}$, $g_{n+1}\lhd g_n$ for all $n$,
and $f_n\stackrel{M1}{\to} h$, $g_n\stackrel{M1}{\to} h$ as $n\to\infty$.
For each $m\in\N$ the set
$$O_m=\{h'\in\mc{A}\-f_m\lhd h', h'\lhd g_m, f_m\neq h', h'\neq g_m\}$$
is an open interval in the order topology,  
hence there exists $N_m\in\N$ such that for all $n\geq N_m$ we have $h_n\in O_{N_m}$.
Without loss of generality we may assume $N_m\to\infty$ as $m\to\infty$.
We thus have that for all $n\geq N_m$
\begin{equation}
\label{eq:fmhngm}
f_m\lhd h_n 
\quand
h_n\lhd g_m.
\end{equation}
Let $h'$ be any limit point of $(h_n)$ in the M1 topology,
thus $h'\in\Pi^\updownarrow$
Letting $m\to\infty$ in \eqref{eq:fmhngm}, by Lemma \ref{l:lhd_compat} we obtain 
$h\lhd h'$ and $h'\lhd h$.
Since both are bi-infinite, we have $h=h'$.
It follows that $h_n\stackrel{M1}{\to} h$,
which (since $B$ is closed in the M1 topology) shows that $h\in B$, and thus completes the proof.
\end{proof}

\begin{lemma}
\label{l:biinf_vs_01}
Let $\mc{A}\sw\Pi^\updownarrow$ be a deterministic weave. 
There exists an
order preserving homeomorphism $\phi$ between the totally ordered spaces $(\mc{A},\lhd)$ and $([0,1],\leq)$.
\end{lemma}

\begin{proof}
Note that the result of Lemma \ref{l:lhd_vs_M1} is implicit in the statement of the present lemma.
Throughout the proof we will use the result of Lemma \ref{l:Amax_order}, that $(\mc{A},\lhd)$ is totally ordered.
For $f\in\Pi^\updownarrow$ we define
\begin{align*}
H^-(f)&=\{(x,t)\-t\in x<f(t-)\wedge f(t+)\}, \\
H^+(f)&=\{(x,t)\-t\in x>f(t-)\vee f(t+)\}.
\end{align*}
Note that $\ov{\R}\times\ov{\R}=H^-(f)\cup H(f)\cup H^+(f)$ 
and that this union is disjoint.
Recall that $d_{\Rc}$ is a metric that generates the topology on $\R^2_{\rm c}$ 
and recall that for $A\sub\R^2_{\rm c}$ the open $\epsilon$-expansion of $A$ is given by 
$A^{(\eps)}=\{z\in\R^2_{\rm c}:\dist_{\Rc}(z,A)<\eps\}$,
where $\dist(z,A)=\inf_{a\in A} d_{\Rc}(z,a)$.
Let $\mu$ be a measure on $\Rc$ that is absolutely continuous with respect to Lebesgue measure, with full support.
Let $\phi(f)=\mu(H^-(f))$.

It is immediate that $\phi$ is non-decreasing and that $\phi(f_{-\infty})=0$, $\phi(f_\infty)=1$
where $f_{\pm\infty}$ are the constant paths at $\pm\infty$ (it is trivial to check that $f_{\pm\infty}\in\mc{A}$).
We next show that $\phi$ is continuous. 
Assume that $f_n, f\in\mc{A}$ and that $f_n\to f$. 
From our remarks above Proposition \ref{p:J1M1} we have $H(f_n)\to H(f)$ in the Hausdorff metric induced by $d$.
Hence, for each $\eps>0$, 
for sufficiently large $n$ we have both $H^-(f_n)\sw H^-(f)^{(\eps)}$ and $H^+(f_n)\sw H^+(f)^{(\eps)}$.
Thus
$$
\limsup_{n\to\infty}\mu\big(H^-(f_n)\big)\leq\mu\big(H^-(f)^{(\eps)}\big)
\quand
\limsup_{n\to\infty}\mu\big(H^+(f_n)\big)\leq\mu\big(H^+(f)^{(\eps)}\big).
$$
Letting $\eps\down 0$ and using the fact that
$\phi_\mu(\pi)=\mu\big(H^-(\pi)\big)=1-\mu\big(H^+(\pi)\big)$ by
Lemma~\ref{l:Hmeas}, we see that $\phi_\mu(\pi_n)\to\phi_\mu(\pi)$ as
$n\to\infty$.

Our next goal is to show that $\phi_\mu$ is a bijection.
From what we have already proved, $\phi$ is surjective and non-decreasing,
so it suffices to prove that if $f\lhd g$ with $f\neq g$ then $\phi(f)<\phi(g)$.
If $f\lhd g$ are not equal then, since both are bi-infinite, 
there exists $t\in\R$ such that $f(t+)<g(t+)$,
from which it follows by right continuity that the set
$O=\{(x,s)\-f(t+)<x<g(t+)\}$ has non-empty interior,
and thus positive Lebesgue measure.
Since $O\sw H^-(g)\sc H^-(f)$ and $\mu$ is absolutely continuous with Lebesgue measure,
we have $\phi(f)<\phi(g)$, as required.

We have now shown that $\phi$ is a continuous bijection from the compact space $\mc{A}$ to (the Hausdorff topological space) $[0,1]$,
which implies that $\phi$ is a homeomorphism.
The fact that $\phi$ is non-decreasing and bijective implies that $\phi$ is order preserving.
\end{proof}

\begin{lemma}
\label{l:biinf_ramification_meas_zero}
The function $(\mc{A},z)\to\1\{z\text{ is ramified in }\mc{A}\}$
is measurable from $\mc{K}(\Pi)\times\Rc\to \{0,1\}$.
Moreover, for any deterministic weave $\mc{A}\sw\Pi^\updownarrow$
the set of ramification points of $\mc{A}$ has zero Lebesgue measure.
\end{lemma}
\begin{proof}
Let $\mathscr{A}=\{B\in\mc{K}(\Pi)\-\exists f\in A\text{ with }B\sw \{f\}_\uparrow\}$
and let $\ram(\mc{A})\sw\Rc$ denote the set of ramification points of $\mc{A}$.
Note that $z\in\Rc$ is non-ramified if and only if $A(z)\in\mathscr{A}$.
It is straightforward to check that $\mathscr{A}\sw\mc{K}(\Pi)$ is closed,
as a consequence of Lemma \ref{l:appdx_1_sw_limits}.
From Lemma \ref{l:meas_Az} the map
$(A,z)\mapsto A(z)$ from $\mc{K}(\Pi)\times\Rc\to\mc{K}(\Pi)$ is measurable.
We have
$\1\{z\in\ram(\mc{A})\}=\1\{\mc{A}(z)\in\mathscr{A}\}$,
hence that the map $(\mc{A},z)\mapsto \1\{z\in\ram(A)\}$ is measurable.
It follows immediately that $\ram(\mc{A})$ is a measurable subset of $\Rc$,
for any $\mc{A}\in\mathscr{W}_{\det}$.

It remains to show that the measure of $\ram(A)$ is zero.
Take $\phi:\mc{A}\to[0,1]$ as in the statement of Lemma \ref{l:biinf_vs_01},
and let $z\in\ov{\R}\times\ov\R_\mfs$. 
By Lemma \ref{l:appdx_1_sw_limits} the set $\mc{A}(z)$ is closed.
By definition of $\lhd$ (and the fact that $\mc{A}\sw\Pi^\updownarrow$) we have that $\mc{A}(z)$ is an interval of the totally ordered space $(\mc{A},\lhd)$.
Thus $\mc{A}(z)$ is a closed interval of $(\mc{A},\lhd)$ and $\phi(\mc{A}(z))$ is a closed interval of $[0,1]$.

Let $U$ be uniformly distributed on $[0,1]$.
Note that $z\in\Rc$ is ramified 
if and only if the closed interval $\mc{A}(z)$ is more than just a single point,
which occurs if and only if $\P[U\in \phi(\mc{A}(z))]>0$.
Let $\mu$ be a measure on $\Rc$ that is absolutely continuous with respect to Lebesgue measure.
Let $Z$ be a random variable with law $\mu$, independently of $U$.
By Lemma \ref{l:meas_Az} $\mc{A}(Z)$ is a $\mc{K}(\Pi)$ valued random variable.
Then $\P[U\in\phi(\mc{A}(Z))]=\P[\phi^{-1}(U)\in\mc{A}(Z)]$ is the probability that 
the random path $\phi^{-1}(U)$ passes through the random point $Z$.
Recalling that $\mu$ is absolutely continuous with respect to Lebesgue measure, by Lemma \ref{l:Hmeas} this probability is zero.
Thus $Z$ is almost surely not ramified,
which implies that the set of ramification points has zero Lebesgue measure.
\end{proof}

\begin{lemma}
\label{l:biinf_noncr}
Let $\mc{A}\sw\Pi^\updownarrow$ be a deterministic weave.
If $h\in\Pi^\updownarrow$ does not cross $\mc{A}$ then $h\in \mc{A}$.
\end{lemma}

\begin{proof}
We will argue by contradiction.
Suppose that $h\in\Pi^\updownarrow$ does not cross $\mc{A}$ and that $h\notin \mc{A}$.
Since $\mc{A}$ is closed, this means that $h$ is an isolated point of $\mc{A}\cup\{h\}$.
Since $h$ does not cross $\mc{A}$ it is straightfoward to check that $\mc{A}\cup\{h\}$ is a deterministic weave.
Thus, from Lemma \ref{l:biinf_vs_01} we have that $\mc{A}\cup\{h\}$ does not contain any isolated points.
This is a contradiction, which completes the proof.
\end{proof}

\begin{remark}
\label{r:biinf_weave_min_max}
By Lemma \ref{l:biinf_noncr} a deterministic weave $\mc{A}\sw\Pi^\updownarrow$
always contains a constant path at spatial location $-\infty$, similarly at $+\infty$.
These are the minimal and maximal elements of $(\mc{A},\lhd)$.
\end{remark}

\begin{lemma}
\label{l:flow_regeneration}
Let $\mc{A}\sw\Pi^\updownarrow$ be a deterministic weave. 
If $D\sw\R^2$ is dense then $\mc{A}=\ov{\mc{A}(D)}$.
\end{lemma}
\begin{proof}
It is immediate that $\ov{\mc{A}(D)}\sw\mc{A}$. 
Thus $\ov{\mc{A}(D)}$ is compact. 
Since $D$ is dense, it is easily seen from Lemma \ref{l:appdx_1_sw_limits}
that $\ov{\mc{A}(D)}$ is pervasive, and since $\ov{\mc{A}(D)}\sw\mc{A}$ it is also non-crossing.
Hence $\ov{\mc{A}(D)}$ is a deterministic weave.
If $h\in\mc{A}$ then since $\ov{\mc{A}(D)}\sw\mc{A}$
we have that $h$ does not cross $\ov{\mc{A}(D)}$.
Hence by Lemma \ref{l:biinf_noncr} we have $h\in\ov{\mc{A}(D)}$.
Thus $\mc{A}=\ov{\mc{A}(D)}$.
\end{proof}

\subsection{Extensions of paths in weaves}
\label{sec:path_extn}

Our results on weaves rely fundamentally on the fact that, within a weave, 
half-infinite paths may be extended into bi-infinite paths, without inducing crossing.
We also require that such extensions preserve compactness; 
this point will be addressed later on in Lemma \ref{l:flow_is_weave}.
As one might expect, bi-infinite extensions of paths
are closely related to maximal paths under the partial order $\sw$ of path extension.
From Lemma \ref{l:Amax_order}, 
if $\mc{A}$ is a weave then
$(\mc{A}_{\max},\lhd)$ is totally ordered.
It is tempting to hope that bi-infinite extensions of paths could be constructed via taking a suitable limit
of paths in $\mc{A}$,
but in general this is not possible because deterministic weaves are closed sets.
A more delicate operation is required.

\begin{defn}
\label{d:d_cut}
Let $\mc{A}$ be a deterministic weave.
We say that a subset $X\sw\mc{A}_{\max}$ is a \textit {Dedekind cut} of $\mc{A}_{\max}$
if 
(i) whenever $f,g\in\mc{A}_{\max}$ with $f\lhd g$ and $g\in X$ we have $f\in X$ and
(ii) $X$ has no maximal element.
\end{defn}

Let us outline the main results within this section.
Recall that $\flow(\mc{A})$ denotes the set of bi-infinite {\cadlag} paths that do not cross $\mc{A}$.
Dedekind cuts are best known as part of Dedekind's construction of $\ov{\R}$ from $\Q$.
A related situation presents itself here,
in which $\mc{A}_{\max}$ plays the role of $\Q$
and $\flow(\mc{A})$ plays the role of $\ov{\R}$.
Specifically:
in Lemma \ref{l:dcut_Xh} we show that if $h\in\Pi^\uparrow$ does not cross $\mc{A}$
then the paths of $\mc{A}_{\max}$ that lie strictly to the left of $h$
are a Dedekind cut of $\mc{A}_{\max}$.
In Theorem \ref{t:path_extension},
which will be proved across several lemmas, 
we show that Dedekind cuts of $\mc{A}_{\max}$ are in bijective correspondence with
bi-infinite paths that do not cross $\mc{A}$.
Thus, each $f\in\Pi^\uparrow$ gives rise to a Dedekind cut, 
which in turn gives rise to a bi-infinite path $h$, 
extending $f$ without crossing $\mc{A}$.

\begin{lemma}
\label{l:dcut_Xh}
Let $\mc{A}$ be a deterministic weave and let $h\in\Pi^\uparrow$ be a path that does not cross $\mc{A}$.
Then 
\begin{equation}
\label{eq:Xh_def}
X_h=\{f\in\mc{A}_{\max}\-f\lhd h,\; f\nsubseteq h,\; h\nsubseteq f\}
\end{equation}
is a Dedekind cut of $\mc{A}_{\max}$.
If $h,h'\in\Pi^\uparrow$ do not cross $\mc{A}$, and are such that $h\nsubseteq h'$ and $h'\nsubseteq h$, then 
$X_h\neq X_{h'}$.
\end{lemma}
\begin{proof}
Let us first check that $X_h$ is a Dedekind cut of $\mc{A}_{\max}$.
We must check that $X_h$ satisfies conditions (i) and (ii) of Definition \ref{d:d_cut}.
For (i),
suppose that $f\in X$ and $g\in\mc{A}_{\max}$ with $g\lhd f$.
We have $f\lhd h$.
Lemma \ref{l:Amax_order} (which includes that $\lhd$ is transitive) gives that $g\lhd h$, as required.
For (ii),
let $f\in X_h$.
Lemma \ref{l:lhd_intermediates} implies that there exists $g\in\mc{A}_{\max}$, such that 
$f\nsubseteq g$, $g\nsubseteq h$, $g\nsubseteq h$,$h\nsubseteq g$, with $f\lhd g$ and $g\lhd h$.
As $f,g,h\in\mc{A}_{\max}$ this implies that $f,g,h$ are all distinct.
Thus $f$ is not a maximal element of $X$, as required.

It remains to check that $X_h\neq X_{h'}$ whenever $h,h'\in\Pi^\uparrow$ 
do not cross $\mc{A}$ and satisfy $h\nsubseteq h'$ and $h'\nsubseteq h$.
Noting that both $h$ and $h'$ do not cross $\mc{A}$,
by Lemma \ref{l:noncr_transitive_weave} $h$ and $h'$ also do not cross each other. 
Lemma \ref{l:lhd_noncr} gives that $h\lhd h'$ or $h'\lhd h$.
Without loss of generality suppose that $h\lhd h'$.
Then by Lemma \ref{l:lhd_intermediates} there exists $g\in\mc{A}_{\max}$
such that $g\nsubseteq h$, $h\nsubseteq g$, $g\nsubseteq h'$, $h'\nsubseteq g$ with $h\lhd g$ and $g\lhd h'$.
It follows that $g\in X_h$ and $g\notin X_{h'}$, as required.
\end{proof}

We now consider the inverse map to \eqref{eq:Xh_def},
in that we seek to reconstruct a bi-infinite path from its corresponding Dedekind cut.
This part is rather technical and will involve the topology introduced in Section \ref{sec:Rpm} on $\ov\R_\mfs$.
which we invite the reader to recall at this point.
In particular 
recall from Lemma \ref{l:Rpm} that 
$t_n\star_n\to t+$
if and if and only if $t_n\to t$ and $t_n\star_n\geq t+$ for all sufficiently large $n$;
similarly
$t_n\star_n\to t-$
if and if and only if $t_n\to t$ and $t_n\star_n\leq t-$ for all sufficiently large $n$.

If $\mc{A}$ is a weave and $X$ is a Dedekind cut of $\mc{A}_{\max}$ then we set
\begin{equation}
\label{eq:LX_def}
\mathscr{L}(X)=\bigcup_{f\in X}L(f)
\end{equation}
where $L(f)$ is defined in \eqref{eq:LR_sets}.
Note that $\mathscr{L}(X)$ is a subset of $\ov{\R}\times\R_{\mfs}$,
which we equip with the product topology.
Recall also that $L_{t\star}(f)=\{x\in\ov{\R}\-(x,t\star)\in\mathscr{L}(X)\}$
which, according to \eqref{eq:LR_sets} is either empty or equal to $[-\infty,f(t\star))$.
Roughly, our strategy is to show that 
the right-hand boundary of $\mathscr{L}(X)$ is the graph (in space-time $\ov{\R}\times\R_{\mfs}$) of a {\cadlag} path. 
With this in mind, given a Dedekind cut $X$ of $\mc{A}_{\max}$ let $\mathscr{P}_X:\R_{\mfs}\to\ov{\R}$ be given by
\begin{equation}
\label{eq:PX_def}
\mathscr{P}_X(t\star)
=
\sup\l\{x\in\ov{\R}\-(x,t\star)\in\ov{\mathscr{L}(X)}\r\}.
\end{equation}
Taking the closure of $\mathscr{L}(X)$ in \eqref{eq:PX_def} is crucial,
because being a {\cadlag} path corresponds to being a continuous function on $\R_\mfs$,
and being a continuous function corresponds to having a closed graph.

\begin{lemma}
\label{l:LX_leftwards_structure}
Let $\mc{A}$ be a deterministic weave and let $X$ be a Dedekind cut of $\mc{A}_{\max}$.
If $(y,t\star)\in\ov{\mathscr{L}(X)}$ and $x\leq y$ then $(x,t\star)\in\ov{\mathscr{L}(X)}$.
\end{lemma}
\begin{proof}
Let $(y,t\star)\in\ov{\mathscr{L}(X)}$ and $x<y$
By \eqref{eq:LX_def} there exists $f_n\in X$ and $(y_n,t_n\star_n)\in\ov{\R}\times\R_{\mfs}$
such that $(y_n,t_n\star_n)\to y$ and $y_n\in L_{t_n\star_n}(f_n)$.
Thus $x\leq y_n$ for all sufficiently large $n\in\N$.
For such $n$ we have $x_n\in L_{t_n\star_n}(f_n)$,
which implies $(x_n,t_n\star_n)\in \mathscr{L}(X)$,
hence $(x,t\star)\in\ov{\mathscr{L}(X)}$.
\end{proof}

\begin{lemma}
\label{l:lhd_X_Xc}
Let $\mc{A}$ be a deterministic weave and let $X$ be a Dedekind cut of $\mc{A}_{\max}$.
The following hold:
\begin{enumerate}
\item Let $f\in X$ and $g\in\mc{A}_{\max}\sc X$. Then $f\lhd g$.
\item Suppose that $(x,t\star)\in(\ov{\R}\times\R_{\mfs})\sc\ov{\mathscr{L}(X)}$.
Then for all $\eps>0$ there exists $t'\in\R$
and $g\in\mc{A}_{\max}\sc X$ such that $|t-t'|\leq\eps$ and $g\in\mc{A}((x,t'))$.
\end{enumerate}
\end{lemma}
\begin{proof}
For the first part, take $f\in X$ and $g\in \mc{A}_{\max}\sc X$.
Note that $f,g\in\mc{A}$ so $f$ and $g$ may not cross,
which by Lemma \ref{l:lhd_noncr} implies that $f\lhd g$ or $g\lhd f$.
If $f\sw g$ or $g\sw f$ then by maximality $f=g$, in which case $f\lhd g$.
Alternatively, if both $f\nsubseteq g$ and $g\nsubseteq f$ then $g\lhd f$ would imply $g\in X$, 
because $X$ is a Dedekind cut; so we must have $f\lhd g$.
Thus, in all cases we have $f\lhd g$.

Let us now consider the second claim.
First consider the case $t\star=t+$.
Let $\eps>0$.
As $(x,t+)\notin\ov{\mathscr{L}(X)}$ there exists $\eps_0>0$ such that
\begin{equation}
\label{eq:spacetime_box_not_L}
\big([x-\eps_0,x+\eps_0]\times[t+,(t+\eps_0)+]\big) \cap \ov{\mathscr{L}(X)}=\emptyset.
\end{equation}
Without loss of generality, assume $\eps\in(0,\eps_0)$.
By pervasiveness of $\mc{A}$ there exists $g\in\mc{A}((x,t'))$ where $t'=t+\eps/2$.
Without loss of generality we may take $g\in\mc{A}_{\max}$.
It is clear that $|t-t'|\leq\eps$.

Consider if $g\in X$.
Note that we have $g(t'+)\vee g(t'-)\geq x$.
If $g(t'+)\geq x$ then $[-\infty,x)\sw L_g(t'+)$
which would imply $(x,t'+)\in \ov{L(g)}\sw\ov{\mathscr{L}(X)}$, contradicting \eqref{eq:spacetime_box_not_L}, so this may not happen.
The remaining case is that $g(t'+)<x\leq g(t'-)$,
in which case $[-\infty,x)\sw L_{t'-}(g)$, implying that $(x,t')\in\ov{L(g)}\sw\ov{\mathscr{L}(X)}$,
contradicting \eqref{eq:spacetime_box_not_L}, 
so this may not happen either.
We conclude that $g\notin X$.
This completes the proof of the case $t\star=t+$.
The case $t=t-$ is similar, using in place of \eqref{eq:spacetime_box_not_L}
that for some $\eps_0>0$ we have
$[x-\eps_0,x+\eps_0]\times[(t-\eps_0)-,t-] \cap \ov{\mathscr{L}(X)}=\emptyset$.
\end{proof}

\begin{lemma}
\label{l:PX_cadlag}
Let $\mc{A}$ be a deterministic weave and let $X$ be a Dedekind cut of $\mc{A}_{\max}$.
Then $\mathscr{P}_X$ is a bi-infinite {\cadlag} path.
\end{lemma}
\begin{proof}
Let us write $h=\mathscr{P}_X$ for the duration of this proof.
We must show that $h$ is a continuous map from $\R_{\mfs}$ to $\ov{\R}$.
By the closed graph theorem, 
the function $h:\R_\mfs\to\ov{\R}$ is continuous if and only if its graph 
$\mathscr{H}=\{(h(t\star),t\star)\-t\star\in\R_\mfs\}$ 
is a closed subset of $\ov{\R}\times\R_{\mfs}$.
Let $t_n\star_n\to t\star$ in $\R_\mfs$.
By compactness of $\ov{\R}$ the sequence $(h(t_n\star_n),t_n\star_n)$ is relatively compact.
Let $(x,t\star)$ be a limit point of this sequence,
and (with slight abuse of notation) let us pass to a subsequence such that $h(t_n\star_n)\to x$.
To establish the present lemma we must show that $x=h(t\star)$.

By \eqref{eq:PX_def},
for each $n\in\N$ there exists a sequence $(x_{n,m},t_{n,m}\star_{n,m})_{m\in\N}\sw \mathscr{L}(X)$ 
such that 
$(x_{n,m},t_{n,m}\star_{n,m})\to (h(t_n\star_n),t_n\star_n)$
as $m\to\infty$.
By a diagonal argument there exists a strictly increasing function $m:\N\to\N$ such that
$(x_{n,m(n)},t_{n,m(n)}\star_{n,m(n)})\to (x,t\star)$.
Hence $(x,t\star)\in\ov{\mathscr{L}(X)}$,
which implies that $x\leq h(t\star)$.
If $h(t\star)=-\infty$ then we now have $x=h(t\star)$, 
so in what follows we may assume that 
$-\infty<h(t\star)$.

We will now argue by contradiction:
suppose that $x<h(t\star)$.
Let $\eps>0$ be such that $x+3\eps\leq h(t\star)$,
and note that $x+\eps<h(t\star)-\eps$.
We consider the cases $t\star=t+$ and $t\star=t-$ in turn.

\medskip

Suppose, first, that $t\star=t+$. Let us briefly outline the strategy. 
We will construct a sequence of 
$f_n\in X$ that come close to the space-time point $(h(t+),t+)$,
and a sequence of $g_n\in\mc{A}_{\max}\sc X$ that come close to $(h(t_n\star_n),t_n\star_n)\approx (x,t+)$.
Note that $x<h(t+)$, whilst
Lemma \ref{l:lhd_X_Xc} gives $f_j\lhd g_k$ for all $j,k\in\N$.
This combination causes $(f_n)$ and $(g_n)$ to become tangled up in each other,
so much so that their limit points $f,g\in\mc{A}$ will cross, 
by jumping over each other in opposite directions at time $t$,
resulting in a contradiction.
We now proceed with the proof.

By Lemma \ref{l:Rpm}, the fact that $t_n\star_n\to t+$ implies that for sufficiently large $n$ we must have $t_n\star_n\geq t+$.
As $h(t_n\star_n)\to x<h(t+)$, in fact for sufficiently large $n$ we have $t_n\star_n>t+$,
and also $h(t_n\star_n)\leq x+\eps$.
Without loss of generality we pass to a subsequence and assume that both these properties hold for all $n$.

By \eqref{eq:PX_def} there exists $(y_n,s_n\bullet_n)_{n\in\N}\sw\mathscr{L}(X)$ such that
$(y_n,s_n\bullet_n)\to(h(t+),t+)$.
The fact that $s_n\bullet_n\to t+$ implies that for sufficiently large $n$ we must have $s_n\bullet_n\geq t+$,
and for sufficiently large $n$ we also have $y_n\geq h(t\star)-\eps$.
so without loss of generality we pass to a subsequence and assume that both these properties hold for all $n$.

Consider if 
\begin{equation}
\label{eq:ht+_sup_L_bdry}
h(t+)=\sup\{x\-(x,t+)\in\mathscr{L}(X)\}.
\end{equation}
In this case there exists $f\in X$ such that $L_{t+}(f)=[-\infty,f(t+))$ and $0<h(t+)-f(t+)\leq\eps$,
so $f(t+)\geq x+2\epsilon$.
Right continuity (i.e.~forwards in time) of $f$ thus implies $h(t_n\star_n)\geq x+\eps$ for all sufficiently large $n$,
which contradicts the fact that $h(t_n\star_n)\to x$.
So this case may not occur.

Therefore, $h(t+)\neq \sup\{x\-(x,t+)\in\mathscr{L}(X)\}$,
which implies that for sufficiently large $n$ we have $s_n\bullet_n\neq t+$
(because $(y_n,s_n\bullet_n)\to (h(t+),t+)$ and $(y_n,s_n\bullet_n)\in\mathscr{L}(X)$).
We have already seen that $s_n\bullet_n\geq t+$, so
without loss of generality we pass to a subsequence and assume that $s_n\bullet_n>t+$ for all $n$.

As $(y_n,s_n\bullet_n)\in\mathscr{L}(X)$ there exists $f_n\in X$ such that $(y_n,s_n\bullet_n)\in L(f_n)$.
Hence $f_n(s_n\bullet_n)\geq y_n\geq h(t\star)-\epsilon$.
By compactness of $\mc{A}$, 
without loss of generality we pass to a subsequence and assume that $f_n\to f\in\mc{A}$.

Let $x_n=h(t_n\star_n)+2^{-n}$. 
As $h(t_n\star_n)\to x$, without loss of generality we pass to a subsequence and assume that
$x_n\leq x+\eps$ for all $n$.
Lemma \ref{l:LX_leftwards_structure} gives that $(x_n,t_n\star_n)\notin \ov{\mathscr{L}(X)}$.
Thus, by the second part of Lemma \ref{l:lhd_X_Xc},
for each $n\in\N$ there exists $t_n'\in\R_\mfs$ and $g_n\in\mc{A}_{\max}\sc X$
such that $|t_n-t_n'|\leq 2^{-n}$, $t'_n\star_n>t+$ and $g_n\in\mc{A}((x_n,t'_n))$.
By compactness of $\mc{A}$, 
without loss of generality we pass to a subsequence and assume that $g_n\to g\in\mc{A}$.
By the first part of Lemma \ref{l:lhd_X_Xc} we have $f_i\lhd g_j$ for all $i,j\in \N$.

We have that $t'_n\star_n$ and $s_n\bullet_n$ are both strictly greater than $t+$,
and both tend to $t+$ as $n\to\infty$.
Consequently, 
passing to further subsequences,
there exists a strictly increasing function $n\mapsto n'$ such that 
\begin{equation}
\label{eq:st_time_order}
s_{n+1}\bullet_{n+1} < t'_{n'}\star_{n'} < s_n\bullet_n
\end{equation}
for all $n$.

We now examine the sequence $(g_n)$ as $n\to\infty$.
We will show that
\begin{alignat}{3}
g_{n'}(t'_{n'}-)\vee g_{n'}(t'_{n'}+) & \leq x_{n'} && \leq x+\eps, \label{eq:gln_1} \\
g_{n'}(s_n\bullet_n) & \geq y_n && \geq h(t\star)-\eps. \label{eq:gln_2}
\end{alignat}
Equation \eqref{eq:gln_1} follows because $g_{n'}\in\mc{A}((x_{n'},t'_{n'}))$.
To see equation \eqref{eq:gln_2}: we have that $f_{n}\lhd g_{n'}$
and that $[-\infty,y_{n})\sw L_{s_{n}\bullet_{n}}(f_{n})$.
Lemma \ref{l:lhd_LR} implies $L(f_{n})\cap R(g_{n'})=\emptyset$,
and $\sigma_{g_{n'}}<s_n\bullet_n$ so we must have $g_{n'}(s_n\bullet_n)\geq y_n$.

From Lemma \ref{l:appdx_1_sw_limits},
combined with \eqref{eq:st_time_order}, \eqref{eq:gln_1} and \eqref{eq:gln_2}
we obtain that the limit $g$ makes a rightwards jump at time $t$,
from below $x+\epsilon$ at time $t-$ to above $h(t\star)-\epsilon$ at time $t+$.

We now turn our attention to $(f_n)$, in similar style.
Here, we show that
\begin{alignat}{3}
f_{n+1}(s_{n+1}\bullet_{n+1}) & \geq y_{n+1} && \geq h(t\star)-\eps \label{eq:fn_1} \\
f_{n+1}(t'_{n'}-)\wedge f_{n+1}(t'_{n'}+) & \leq  x_n && \leq x+\eps \label{eq:fn_2} 
\end{alignat}
Equation \eqref{eq:fn_1} follows from the fact that $(y_n,s_n\bullet_n)\in L(f_n)$.
To see equation \eqref{eq:fn_2}: we have that $f_{n+1}\lhd g_{n'}$.
If $g_{n'}(t'_{n'}+)\leq g_{n'}(t'_{n'}-)$ then we have
$f_{n+1}(t'_{n'}+)\leq g_{n'}(t'_{n'}+)\leq x_n$.
Alternatively, if $g_{n'}(t'_{n'}-)<g_{n'}(t'_{n'}+)$ then we have
$[-\infty,x_n)\sw R_{t'_{n'}-}(g_{n'})$,
and Lemma \ref{l:lhd_LR} gives $L(f_{n+1})\cap R(g_{n'})=\emptyset$,
which implies $f_{n+1}(t'_{n'}-)\leq x_n$.
In both cases we have \eqref{eq:fn_2}.

From Lemma \ref{l:appdx_1_sw_limits}
combined with \eqref{eq:st_time_order}, \eqref{eq:fn_1} and \eqref{eq:fn_2}
we obtain that the limit $f$ makes a leftwards jump at time $t$,
from above $h(t\star)-\epsilon$ at time $t-$ to below $x+\epsilon$ at time $t+$.
Thus $f$ and $g$ cross (by jumping in opposite directions over each other at time $t$).
As both $f,g\in\mc{A}$, this is a contradiction.
This completes the proof of the case $t\star=t+$.

\medskip

It remains to consider the case $t\star=t-$.
The argument is essentially the same, 
except that Lemma \ref{l:Rpm} requires that we now 
approach $t-$ from the left (i.e.~from backwards in time) rather than $t+$ from the right.
In outline: construct a sequence of 
$f_n\in X$ that come close to the space-time point $(h(t-),t-)$,
and a sequence of $g_n\in\mc{A}_{\max}$ that come close to $(h(t_n\star_n),t_n\star_n)\approx (x,t-)$.
Note that $x<h(t-)$, whilst
Lemma \ref{l:lhd_X_Xc} gives $f_j\lhd g_k$ for all $j,k\in\N$.
This combination causes $(f_n)$ and $(g_n)$ to become entangled with each other,
so that once again their limit points $f,g\in\mc{A}$ will cross -- resulting in a contradiction.

There is one point at which a difference worthy of comment emerges.
This concerns \eqref{eq:ht+_sup_L_bdry}.
If $h(t-)=\sup\{x\-(x,t-)\in\mathscr{L}(X)\}$ then 
there exists $f\in X$ such that $L_{t-}(f)=[-\infty,f(t-))$ and $0<h(t-)-f(t-)\leq\epsilon$.
\begin{itemize}
\item If $\sigma_f<t$ then a similar argument to that in the same paragraph as \eqref{eq:ht+_sup_L_bdry} applies, 
using left continuity of $f$ instead of right continuity;
this reaches a contradiction.
\item If $\sigma_f=t$ then we require a new step within the argument,
one that features only here because of the `extra' behaviour of $L_{t\star}(f)$ when $t\star=\sigma_f-$, see \eqref{eq:LR_sets}.
In particular, for $t=\sigma_f$ the fact that $L_{t-}(f)=[-\infty,f(t-))$ implies that $f(t+)<f(t-)$.
We then proceed as before to construct $g\in\mc{A}$ such that $g(t-)\leq x+\epsilon$ and $g(t+)\geq h(t-)-\epsilon$.
Thus $f$ and $g$ cross, reaching a contradiction.
\end{itemize}
If $h(t-)\neq\{x\-(x,t-)\in\mathscr{L}(X)\}$
then we can (and moreover can only) approximate $(h(t-),t-)\in\ov{\mathscr{L}(X)}$ 
using space-time points in $\mathscr{L}(X)$ with times strictly less than $t-$.
In this case we may proceed as before.
This completes the proof.
\end{proof}

\begin{lemma}
\label{l:PX_lhd_noncr}
Let $\mc{A}$ be a deterministic weave and let $X$ be a Dedekind cut of $\mc{A}_{\max}$.
Then $\mathscr{P}_X$ does not cross $\mc{A}$.
Moreover, if $f\in X$ then $f\lhd \mathscr{P}_X$,
and if $g\in \mc{A}_{\max}\sc X$ then $\mathscr{P}_X\lhd g$.
\end{lemma}
\begin{proof}
Let us write $h=\mathscr{P}_X$ for the duration of this proof.
From Lemma \ref{l:PX_cadlag} we have $h\in\Pi^\updownarrow$.
Note that $h$ crosses $\mc{A}$ if and only if $h$ crosses $\mc{A}_{\max}$.
We will show, in turn, that 
(a) $f\in X \ra f\lhd h$
and 
(b) $g\in\mc{A}_{\max}\sc X \ra h\lhd g$.
With this in hand it follows from Lemma \ref{l:lhd_noncr} that $h$ does not cross $\mc{A}_{\max}$,
thus $h$ does not cross $\mc{A}$.

We begin with (a). 
Let $f\in X$ and $t\star\geq\s_f-$.
If $L_{t\star}(f)=[-\infty,f(t\star))$ then
$(f(t\star),t\star)\in\mathscr{L}(X)$ and hence $f(t\star)\leq h(t\star)$,
which implies $L_{t\star}(f)\cap R_{t\star}(h)=\emptyset$.
Alternatively, if $L_{t\star}(f)=\emptyset$ then it is immediate that $L_{t\star}(f)\cap R_{t\star}(h)=\emptyset$.
Hence $L(f)\cap R(g)=\emptyset$, which by Lemma \ref{l:lhd_left_right} implies that $f\lhd h$.

We now move on to (b). 
Let $g\in\mc{A}_{\max}\sc X$.
We will argue by contradiction.
Suppose that $h\nlhd g$.
Then by Lemma \ref{l:lhd_LR} we have $L(h)\cap R(g)\neq\emptyset$.
In particular, 
for some $t\star\geq\s_g-$ we have
$L_{t\star}(h)=[-\infty,h(t\star))$ and
$R_{t\star}(g)=(g(t\star),\infty]$
with $g(t\star)<h(t\star)$.

Consider first if $t\star\geq\s_g+$.
Then, using the {\cadlag} property of $g$ and $h$ 
there exists $\eps>0$ an interval $[a+,b-]$ with $a<b$
such that 
\begin{equation}
\label{eq:PX_lhd_0}
g(s\bullet)+\eps\leq h(\s\bullet)
\end{equation}
for all $s\bullet\in[a-,b+]$
(to see this: if $\star=+$ take $a=t$, if $\star=-$ take $b=t$).
Let us briefly note our strategy here: we will use \eqref{eq:PX_lhd_0} to
show that $g$ lies to the left of some path in $X$.
Fix some $s\bullet\in\R_\mfs$ with $a+<s\bullet<b-$.
By \eqref{eq:PX_def} there exists $(y_n,s_n\bullet_n)\in\mathscr{L}(X)$
such that $(y_n,s_n\bullet_n)\to (h(s\bullet),s\bullet)$ as $n\to\infty$.
For all sufficiently large $n\in\N$ we have 
\begin{equation}
\label{eq:PX_lhd_1}
h(s\bullet)-\eps/2 \leq y_n
\qquad\text{ and }\qquad
a-\leq s_n\bullet_n\leq b+.
\end{equation}
By Lemma \ref{l:PX_cadlag} we have $h(s_n\bullet_n)\to h(s\bullet)$,
so for all sufficiently large $n$ we also have
\begin{equation}
\label{eq:PX_lhd_2}
|h(s\bullet)-h(s_n\bullet_n)|\leq\eps/2.
\end{equation}
Fix $n\in\N$ large enough that \eqref{eq:PX_lhd_1} and \eqref{eq:PX_lhd_2} both hold.
By \eqref{eq:LX_def} there exists $f_n\in X$ such that 
$(y_n,s_n\bullet_n)\in L_{s_n\bullet_n}(f)$,
which means that $y_n<f(s_n\bullet_n)$.
Combining this inequality with \eqref{eq:PX_lhd_0}, \eqref{eq:PX_lhd_1} and \eqref{eq:PX_lhd_2} we obtain that
$$
g(s_n\bullet_n)
\;\leq\; h(s_n\bullet_n)-\eps
\;\leq\; h(s\bullet)-\eps/2
\;\leq\; y_n
\;<\; f_n(s_n\bullet_n).$$
Hence $g(s_n\bullet_n)\in L_{s_n\bullet_n}(f)$,
which by Lemma \ref{l:lhd_left_right} means that $g\lhd f_n$.
As $X$ is a Dedekind cut and $f_n\in X$, we thus have $g\in X$, which is a contradiction.

It remains to consider the case $t\star=\s_g-$.
In this case by \eqref{eq:LR_sets} we have $R_{\s-}(g)=(g(\s-),\infty]$ and $g(\s-)<g(\s+)$.
We have also that $g(\s-)<h(\s-)$.
By \eqref{eq:PX_def} there exists $(y_n,s_n\bullet_n)\in\mathscr{L}(X)$ such that
$(y_n,s_n\bullet_n)\to (h(\s-),\s-)$,
and $f_n\in X$ such that $y_n\in L_{s_n\bullet_n}(f)=[-\infty,f_n(s_n\bullet_n))$.
By compactness of $\mc{A}$ we may pass to a subsequence and assume that $f_n\to f\in\mc{A}$.
Without loss of generality we may assume that $s_n\bullet_n\leq s-$,
so $\s_{f_n}\leq \s_g$.

Suppose that $f_n(\s-)\geq g(\s-)$: then $f_n(\s-)\in R_{\s-}(g)$ which, 
by Lemma \ref{l:lhd_left_right} would give $g\lhd f_n$,
and as $X$ is a Dedekind cut this would give $g\in X$, which is false.
Hence in fact $f_n(\s-)<g(\s-)$.
We now have $s_n\bullet_n\leq \s-$ with $\s_n\bullet_n\to \s-$, along with $f_n(s_n\bullet_n)\geq y_n$
and $f_n(\s-)\leq g(\s-)$.
By Lemma \ref{l:appdx_2_fntn}, this implies that $f$ jumps leftwards at $\s$,
from right of $h(\s-)$ to left of $g(\s-)$.
This implies that $g$ and $f$ cross, by jumping over each other in opposite directions as $\s$,
which is a contradiction as both $f,g\in\mc{A}$.
This completes the proof.
\end{proof}

\begin{lemma}
\label{l:PX_inverse_of_Xh}
Let $\mc{A}$ be a deterministic weave and suppose that $h\in\Pi^\updownarrow$ does not cross $\mc{A}$.
Define $X=X_h$ according to \eqref{eq:Xh_def}.
Then $\mathscr{P}_X=h$.
\end{lemma}
\begin{proof}
Let us write $h'=\mathscr{P}_X$ for the duration of this proof.
We must show that $h=h'$.
We will argue by contradiction.
Noting that $h$ is bi-infinite, 
if $h\neq h'$ then there exist $t+\in\R_\mfs$ such that $h(t+)\neq h'(t+)$.
Without loss of generality (or consider space reflected about the origin) 
we may assume that $h(t+)<h'(t+)$.
Thus $h\lhd h'$.
By Lemma \ref{l:pincer_t_plus},
again using that $h$ and $h'$ are bi-infinite,
there exists $g\in\mc{A}$
such that $h\lhd g$ and $g\lhd h'$ with $g\nsubseteq h$ and $g\nsubseteq h'$.
Without loss of generality we may take $g\in\mc{A}_{\max}$.
Hence either $g\in X$ or $g\in\mc{A}_{\max}\sc X$.
We consider these two cases separately.

If $g\in X$ then,
recalling that
$X=\{f\in\mc{A}_{\max}\-f\lhd h\text{ and }f\nsubseteq h\}$,
we have $g\lhd h$.
From Lemma \ref{l:lhd_noncr}, noting that $h$ is bi-infinite,
we thus obtain $g\sw h$, which is a contradiction.
If $g\in \mc{A}_{\max}\sc X$ then
by Lemma \ref{l:PX_lhd_noncr} we have $h'\lhd g$.
From Lemma \ref{l:lhd_noncr}, noting that $h'$ is bi-infinite,
we thus obtain $g\sw h'$, which is a contradiction.
Having reached a contradiction in both cases,
we conclude that in fact $h=h'$.
\end{proof}

\begin{lemma}
\label{l:Xh_onto}
Let $\mc{A}$ be a deterministic weave and
let $X$ be a Dedekind cut of $\mc{A}_{\max}$.
There exists $h\in\Pi^\updownarrow$ such that $X=X_h$,
where $X_h$ is given by \eqref{eq:Xh_def}.
\end{lemma}
\begin{proof}
Let us write $h=\mathscr{P}_X$ for the duration of this proof.
By Lemma \ref{l:PX_cadlag} we have $h\in\Pi^\updownarrow$.
Define $X_h$ as in \eqref{eq:Xh_def},
that is 
$X_h=\{f\in\mc{A}_{\max}\-f\leq h\text{ and }f\nsubseteq h\}$.
Note that since $h\in\Pi^\updownarrow$ we may discard the condition $h\nsubseteq f$,
because if $h\sw f$ then $f=h$ and thus also $f\sw h$.
We must show that $X=X_h$.

Let $f\in X$. 
Lemma \ref{l:PX_lhd_noncr} gives that $f\lhd h$.
Moreover Lemma \ref{l:PX_lhd_noncr} gives that
$g\lhd h$ for all $g\in X$.
Thus, if $f\subseteq h$ then we would also have $g\lhd f$ for all $g\in X$,
which would make $f$ a maximal element of $X$; 
this a contradiction as $X$ is a Dedekind cut.
Hence $f\nsubseteq h$.
We thus have $X\sw X_h$.
We now move on to the reverse inclusion.

Let $f\in X_h$,
so we have $f\in\mc{A}_{\max}$, $f\lhd h$ and $f\nsubseteq h$.
If $f\notin X$ then Lemma \ref{l:PX_lhd_noncr} gives that $h\lhd f$,
from which part 2 of Lemma \ref{l:lhd_possibilities} implies that $f\sw h$ or $h\sw f$;
this is a contradiction.
Hence $X_h\sw X$, so $X=X_h$ as required.
\end{proof}

\begin{theorem}
\label{t:path_extension}
Let $\mc{A}$ be a deterministic weave.
Let $\mathscr{X}$ denote the set of Dedekind cuts of $\mc{A}_{\max}$.
\begin{enumerate}
\item
The map $h\mapsto X_h$ given by \eqref{eq:Xh_def}
is a bijection between $\flow(\mc{A})$ and $\mathscr{X}$.
The inverse map $X\mapsto \mathscr{P}_X$ is given by \eqref{eq:PX_def}.
\item
For any $f\in\Pi^\uparrow$ that does not cross $\mc{A}$,
there exists $h\in\flow(\mc{A})$ such that $f\sw h$.
\end{enumerate}
\end{theorem}
\begin{proof}
Recall from \eqref{eq:flow_op} that by definition $\flow(\mc{A})=\{h\in\Pi^\updownarrow\-h\text{ does not cross }\mc{A}\}$.
Lemma \ref{l:dcut_Xh} gives that the range of the map $h\mapsto X_h$ (with domain $\flow(\mc{A})$) is within $\mathscr{X}$
and that this map is injective.
Lemma \ref{l:Xh_onto} gives that $h\mapsto X_h$ has range $\mathscr{X}$.
Lemmas \ref{l:PX_cadlag} and \ref{l:PX_lhd_noncr} ensure that the range of the map $X\mapsto\mathscr{P}_X$ 
(with domain $\mathscr{X}$)
is within $\flow(\mc{A})$,
so Lemma \ref{l:PX_inverse_of_Xh} gives that $h\mapsto X_h$ and $X\mapsto\mathscr{P}_X$ are inverses of each other,
between $\flow(\mc{A})$ and $\mathscr{X}$.
This establishes the first claim of the present theorem.

To see the second claim, let $f\in\Pi^\uparrow$ and suppose that $f$ does not cross $\mc{A}$.
By Lemma \ref{l:dcut_Xh} we have $X_f\in\mathscr{X}$,
from which part 1 of the present theorem gives that $h=\mathscr{P}_{X_f}\in\flow(\mc{A})$
does not cross $\mc{A}$.
It remains to show that $f\sw h$.
We will argue by contradiction. 

Suppose that $f\nsubseteq h$,
which as $h\in\Pi^\updownarrow$ implies that $h\nsubseteq f$.
We have that $\mc{A}\cup\{f\}$ is non-crossing and that $\mc{A}\cup\{h\}$ is non-crossing.
It follows by Lemma \ref{l:noncr_transitive_weave} that $\mc{A}\cup\{f,h\}$ is non-crossing,
so in particular $f$ does not cross $h$.
By Lemma \ref{l:lhd_noncr} we have
$f\lhd h$ or $h\lhd f$.
We treat these two cases in turn.

Consider, first, if $f\lhd h$.
Then Lemma \ref{l:lhd_intermediates} gives $g\in\mc{A}_{\max}$
such that $f\lhd g$, $g\lhd h$, with $f,g,h$ all incomparable under $\sw$.
Hence $g\in\mc{A}_{\max}\sc X_f$, which by Lemma \ref{l:PX_lhd_noncr} gives $h\lhd g$.
By Lemma \ref{l:Amax_order} we thus have $g\sw h$ or $h\sw g$, which is a contradiction.

The argument when $h\lhd f$ is similar.
Now Lemma \ref{l:lhd_intermediates} gives $g\in\mc{A}_{\max}$
such that $h\lhd g$, $g\lhd f$ with $f,g,h$ all incomparable under $\sw$.
Hence $g\in X_f$,
which by Lemma \ref{l:PX_lhd_noncr} gives $g\lhd h$,
and Lemma \ref{l:Amax_order} again arrives at a contradiction.
This completes the proof.
\end{proof}

\subsection{The flow map}
\label{sec:flow_op}

In this section we establish further properties of the $\flow$ operation defined in \eqref{eq:flow_op},
including Lemma \ref{l:flow_map_cts} which shows that $\mc{A}\mapsto\flow(\mc{A})$ is continuous on $\mathscr{W}_{\det}$.
Another key result, contained within Lemma \ref{l:relcomp_weaves_to_flows}, 
is that the path extension of Theorem \ref{t:path_extension} preserves compactness.
Here, for the first time, 
we see interaction between all the main concepts of the present article:
relative compactness in $\mc{K}(\Pi^\uparrow)$,
path extension,
and the pervasiveness and non-crossing properties of weaves.

\begin{lemma}
\label{l:relcomp_weaves_to_flows}
Let $\mathscr{A}\sw\mc{K}(\Pi^\uparrow)$ be a relatively compact subset of $\mc{K}(\Pi^\uparrow)$,
where each $\mc{A}\in\mathscr{A}$ is a deterministic weave.
Suppose that any limit point of $\mathscr{A}$ is non-crossing.
Then $\{\flow(\mc{A})\-\mc{A}\in\mathscr{A}\}$ is a relatively compact subset of $\mc{K}(\Pi^\updownarrow)$.
\end{lemma}
\begin{proof}
As $\Pi^\updownarrow$ is a closed subset of $\Pi^\uparrow$,
also $\mc{K}(\Pi^\updownarrow)$ is a closed subset of $\mc{K}(\Pi^\uparrow)$.
It therefore suffices to consider relative compactness in $\mc{K}(\Pi^\uparrow)$.
By Lemma \ref{l:relcom_Pi_KPi} the set 
$\{\flow(\mc{A})\-\mc{A}\in\mathscr{A}\}$ 
is a relatively compact subset of $\mc{K}(\Pi^\uparrow)$
if and only if
$\mc{F}=\bigcup_{\mc{A}\in\mathscr{A}}\flow(\mc{A})$
is a relatively compact subset of $\Pi^\uparrow$.
We will argue by contradiction.

Suppose that $\mc{F}$ is not relatively compact.
Then, by Proposition \ref{p:relcom_tightness} there exists 
$T,\kappa>0$ and sequences $(\delta_n)\sw(0,1)$, $(f_n)\sw\mc{F}$ such that
$\delta_n\searrow 0$ and $w_{T,\delta_n}(f_n)\geq\kappa$.
For convenience, recall from \eqref{eq:moduli_cadlag} that
\begin{align*}
w_{T,\delta}(f)=
\sup\big\{
d_{\ov{\R}}\big(f(t_2\star_2),[f(t_1\star_1),f(t_3\star_3)]\big)\-
&t_1\star_1,t_2\star_2,t_3\star_3\in I_\mfs(f), \\
&-T<t_1<t_2<t_3<T,\ t_3-t_1<\de\big\}.
\end{align*}
So, we have $t^n_i\star^n_i\in I_{\mfs}(f_n)$ such that $-T<t_1^n<t_2^n<t_3^n<T$ and $t^n_3-t^n_1<\delta$,
with
\begin{equation}
\label{eq:not_compact}
d_{\ov{\R}}\l(f_n(t^n_2\star^n_2),\l[f_n(t^n_1\star^n_1), f(t^n_3\star^n_3)\r]\r)\geq\kappa.
\end{equation}
By the {\cadlag} property of the bi-infinite path $f_n$, we may assume without loss of generality (reducing $\kappa>0$ if necessary) that $f_n$ is continuous at $t^n_i\star^n_i$ for $i=1,2,3$.
Using that $\Rc$ is compact, without loss of generality we may pass to a subsequence and assume additionally that 
$f_n(t^n_i)\to y_i\in\ov{\R}$ in $\Rc$.
Using that $[-T,T]$ is compact, we may pass to a further subsequence and assume additionally that
$t^n_i\to t_i$, with $t_i\in[-T,T]$.
Since $0\leq t^n_3-t^n_1<\delta_n$ in fact $t_1=t_2=t_3=t$.
To summarise, we thus have
\begin{equation}
\label{eq:not_compact_time_conv}
t^n_1<t^n_2<t^n_3\text{ for all $n$ and }t^n_i\to t\in[-T,T]\text{ as }n\to\infty.
\end{equation}
Without loss of generality (or consider the same setup with space reflected about the origin) we pass to a further subsequence and assume additionally that 
$f_n(t^n_1)\leq f_n(t^n_2)$,
which by \eqref{eq:not_compact} implies that
\begin{equation}
\label{eq:not_compact_2}
f_n(t^n_i)+\kappa\leq f_n(t^n_2)
\quad\text{ for }i=1,3.
\end{equation}

Note that $f_n\in\mc{A}$, for some $\mc{A}\in\mathscr{A}$,
and let us write $\mc{A}_n$ for such $\mc{A}$.
By pervasiveness of $\mc{A}_n$ there exist
\begin{equation}
\label{eq:fn_gnhn}
g_n\in\mc{A}_n\l((t^n_1,f_n(t^n_1)+\tfrac\kappa3)\r)
\quand
h_n\in\mc{A}_n\l((t^n_2,f_n(t^n_2)-\tfrac\kappa3)\r).
\end{equation}
Since $\mc{A}_n$ is a weave, $g_n$ and $h_n$ do not cross each other.
By Lemma \ref{l:relcom_Pi_KPi}, relative compactness of $\mathscr{A}$ gives that
$\mc{B}=\bigcup_{\mc{A}\in\mathscr{A}}\mc{A}$ is a relatively compact subset of $\Pi^\updownarrow$.
Hence we may pass to a further subsequence and assume additionally that
$g_n\to g\in\ov{\mc{B}}$ and $h_n\to h\in\ov{\mc{B}}$.
The sequence $(\mc{A}_n)$ is a subset of $\mathscr{A}$ and therefore is relatively compact,
so we may pass to a subsequence and assume that $\mc{A}_n\to\mc{A}$.
Hence $g,h\in\mc{A}$.
The set $\mc{A}$ is a limit point of $\mathscr{A}$,
therefore (as a hypothesis of the present lemma) $\mc{A}$ is non-crossing.
Hence $g$ and $h$ may not cross.

Let us briefly comment on the strategy for the remaining part of the proof:
we will establish a contradiction through showing that $g$ and $h$ cross at time $t$,
at which time they will jump past each other in opposite directions.
By Lemma \ref{l:lhd_noncr} we have that $g_n,h_n$ are comparable under $\lhd$ to $f_n$.
Since $f_n$ is continuous at $t^n_i$,
by \eqref{eq:fn_gnhn} and Lemma \ref{l:lhd_LR} we have 
$f_n\lhd g_n$ and $h_n\lhd f_n$.
Hence 
$f_n(t^n_2)\leq g_n(t^n_2\pm)$ and $h_n(t^n_3\pm)\leq f_n(t^n_3)$.
By definition of $g_n,h_n$ we have
$g_n(t^n_1-)\wedge g_n(t^n_1+)\leq f_n(t^n_1)+\frac\kappa3$ and $f_n(t^n_2)-\frac\kappa3\leq h_n(t^n_2-)\vee h_n(t^n_2+)$.
Combining these facts with \eqref{eq:not_compact_2} we thus have $t^n_i\bullet^n_i$ such that
\begin{alignat*}{4}
& g_n(t^n_1\bullet^n_1)+\tfrac{2\kappa}{3} 
&& \leq f_n(t^n_1)+\kappa 
&& \leq f_n(t^n_2)
&& \leq g_n(t^n_2\bullet^n_2)\\
& h_n(t^n_3\bullet^n_3)+\kappa
&& \leq f_n(t_n^3)+\kappa
&& \leq f_n(t^n_2) 
&& \leq h_n(t^n_2\bullet^n_2)+\tfrac{\kappa}{3}.
\end{alignat*}
Hence, 
\begin{alignat*}{4}
& g_n(t^n_1\bullet^n_1)+\tfrac{\kappa}{6} 
&& \leq f_n(t^n_2)-\tfrac\kappa2
&& \leq g_n(t^n_2\bullet^n_2)-\tfrac\kappa2\\
& h_n(t^n_3\bullet^n_3)+\tfrac\kappa2
&& \leq f_n(t^n_2)-\tfrac\kappa2 
&& \leq h_n(t^n_2\bullet^n_2)-\tfrac{\kappa}{6}.
\end{alignat*}
Using that $[-T-,T+]\sw\R_{\mfs}$ is compact (which follows from Lemma \ref{l:Rpm}) we may pass to a subsequence and assume that
for $i=1,2,3$ the sequence $(t^n_i\bullet^n_i)$ converges as $n\to\infty$.
We will now send $n\to\infty$.
Recalling that $f_n(t^n_2)\to y_2$, we thus obtain from Lemma \ref{l:appdx_1_sw_limits} and \eqref{eq:not_compact_time_conv} that
\begin{alignat*}{3}
& g(t-) &&< y_2-\tfrac\kappa2 &&< g(t+) \\
& h(t+) &&< y_2-\tfrac\kappa2 &&< h(t-), 
\end{alignat*}
which implies that $g$ and $h$ cross. This is a contradiction since $g,h\in\mc{A}$.
\end{proof}

The M1 topology and particular form of $w_{T,\delta}$ is crucial to the above argument.
We do not know of an analogous result for the J1 topology.

\begin{lemma}
\label{l:flow_is_weave} 
Let $\mc{A}$ be a deterministic weave.
Then $\flow(\mc{A})$ is a deterministic weave and $\flow(\mc{A})\sw\Pi^\updownarrow$.
\end{lemma}
\begin{proof}
Note that equation \eqref{eq:flow_op} gives that $\flow(\mc{A})\sw\Pi^\updownarrow$.
Let us first establish that $\flow(\mc{A})$ is compact.
By Lemma \ref{l:relcomp_weaves_to_flows} (taking $\mathscr{A}=\{\mc{A}\}$)
we have that $\flow(\mc{A})$ is relatively compact.
We will show that $\flow(\mc{A})$ is closed (and thus compact).
Suppose that $g_n\to g$ where $g_n\in\flow(\mc{A})$.
For any $f\in\mc{A}$ there exists $f'\in\flow(\mc{A})$
such that $f\sw f'$.
Thus $f'$ and $g_n$ do not cross, for all $n$.
Lemma \ref{l:noncr_biinf_limits} gives that $f'$ and $g$ do not cross,
which implies that $f$ and $g$ do not cross.
Thus $f$ does not cross $\mc{A}$, which implies that $f\in\flow(\mc{A})$.

To check that $\flow(\mc{A})$ is a weave, 
it remains to show that $\flow(\mc{A})$ is pervasive and non-crossing.
If $g,h\in\flow(\mc{A})$ then by \eqref{eq:flow_op} we have that 
$\mc{A}\cup\{g\}$ is non-crossing and $\mc{A}\cup\{h\}$ is non-crossing. 
By Lemma \ref{l:noncr_transitive_weave} we thus have that $\{g,h\}$ is non-crossing,
so in fact $\flow(\mc{A})$ is non-crossing.
If $z\in\R^2_c$ then pervasiveness of $\mc{A}$ implies that 
there exists some $f\in\mc{A}$ such that $f\in H(f)$.
Theorem \ref{t:path_extension} implies that there exists $f'\in\flow(\mc{A})$ with $f\sw f'$,
hence $z\in H(f')$.
Thus $\flow(\mc{A})$ is pervasive.
\end{proof}

\begin{lemma}
\label{l:flow_map_cts}
Let $\mc{A}_n,\mc{A}$ be deterministic weaves with $\mc{A}_n\to\mc{A}$.
Then $\flow(\mc{A}_n)\to\flow(\mc{A})$.
\end{lemma}
\begin{proof}
The set $\mathscr{A}=\{\mc{A}_n\-n\in\N\}$ is a relatively compact subset of $\mc{K}(\Pi^\uparrow)$.
The only limit point of $\mathscr{A}$ is $\mc{A}$, which is non-crossing.
Thus by Lemma \ref{l:relcomp_weaves_to_flows}
the set $\{\flow(\mc{A}_n)\-n\in\N\}$ is a relatively compact subset of $\mc{K}(\Pi^\uparrow)$.
Suppose that we have $\mc{A}_n\to\mc{F}$ along some subsequence 
and pass to this subsequence, with mild abuse of notation.
To prove the present lemma we must show that $\mc{F}=\flow(\mc{A})$.

By Lemma \ref{l:flow_map_cts} we have that $\flow(\mc{A})$ is a deterministic weave.
Let us show that $\mc{F}$ is also a deterministic weave, with $\mc{F}\sw\Pi^\updownarrow$.
From the previous paragraph we have $\mc{F}\in\mc{K}(\Pi^\uparrow)$.
As $\mc{A}_n\to\mc{F}$ and each $\mc{A}_n$ is a subset of the closed set $\Pi^\updownarrow$,
we have $\mc{F}\sw\Pi^\updownarrow$.
If $f,g\in\mc{F}$ then we have $f_n,g_n\in\flow(\mc{A}_n)$ such that $f_n\to f$ and $g_n\to g$.
We have that $f_n$ and $g_n$ do not cross, 
so by Lemma \ref{l:noncr_biinf_limits} $f$ and $g$ do not cross.
Thus $\mc{F}$ is non-crossing.
For all $z\in\R^2_c$ there exists $f_n\in\mc{A}_n$ such that $z\in H(f_n)$.
By Lemma \ref{l:relcom_Pi_KPi}, relative compactness of $(\mc{A}_n)$ implies relative compactness of $(f_n)$,
thus we may pass to a subsequence and assume $f_n\to f\in\mc{F}$.
Lemma \ref{l:appdx_1_sw_limits} gives that $z\in H(f)$.
Thus $\mc{F}$ is pervasive.
We have now shown that $\mc{F}\sw\Pi^\updownarrow$ is a deterministic weave.

Our next goal is to show that $\mc{F}$ and $\flow(\mc{A})$ are non-crossing.
Let $f\in\mc{F}$ and $g\in\mc{A}$.
Then there exists $f_n\in\flow(\mc{A}_n)$ and $g_n\in\mc{A}_n$ such that
$f_n\to f$ and $g_n\to g$.
By Theorem \ref{t:path_extension}, there exists $g'_n\in\flow(\mc{A}_n)$
such that $g_n\sw g'_n$.
We have $\flow(\mc{A}_n)\to\mc{F}$, so
by Lemma \ref{l:relcom_Pi_KPi} the set $\mc{F}\cup(\bigcup_{n\in\N}\flow(\mc{A}_n))$
is compact,
which implies that $\{g'_n\-n\in\N\}$ is relatively compact.
Hence there exists $g'\in\mc{F}$ such that $g'_n\to g'$.
By Lemma \ref{l:appdx_1_sw_limits} we have $g\sw g'$.
Both $f_n$ and $g'_n$ are elements of $\mc{A}_n$, hence they do not cross each other.
By Lemma \ref{l:noncr_biinf_limits}, $f$ and $g'$ do not cross each other.
As $g\sw g'$ this means that $f$ and $g$ do not cross. 

We now have that $\mc{F}$ and $\mc{A}$ do not cross,
and that both are deterministic weaves consisting entirely of bi-infinite paths.
Lemma \ref{l:biinf_noncr} gives that $\mc{F}=\mc{A}$.
\end{proof}

\begin{lemma}
\label{l:noncr_same_flow}
Let $\mc{A},\mc{B}$ be deterministic weaves and assume that $\mc{A}\cup\mc{B}$ is non-crossing.
Then $\mc{A}\cup\mc{B}$ is a deterministic weave and $\flow(\mc{A})=\flow(\mc{B})$.
\end{lemma}
\begin{proof}
It is trivial to check that $\mc{A}\cup\mc{B}$ is a deterministic weave.
Lemma \ref{l:flow_is_weave} gives that $\flow(\mc{A})$ and $\flow(\mc{B})$ are deterministic weaves,
composed entirely of bi-infinite paths.
Lemma \ref{l:noncr_two_weaves} gives that a path $f\in\Pi^\updownarrow$ crosses $\mc{A}$ if and only if it crosses $\mc{B}$,
so by Lemma \ref{l:biinf_noncr} we have $\flow(\mc{A})=\flow(\mc{B})$.
\end{proof}

\begin{lemma}
\label{l:flow_maximal}
Let $\mc{A}$ be a deterministic weave.
Then 
$\flow(\mc{A})$ is a maximal element of $(\mathscr{W}_{\det},\preceq)$
and 
$\mc{A}\preceq\flow(\mc{A})$.
\end{lemma}
\begin{proof}
The reader may wish to check \eqref{eq:preceq} for the definition of the partial order $\preceq$.
Note that Lemma \ref{l:flow_is_weave} gives that $\flow(\mc{A})\in\mathscr{W}_{\det}$.
Let us first show that $\mc{A}\preceq \flow(\mc{A})$.
Theorem \ref{t:path_extension} gives that $\mc{A}\sw(\flow(\mc{A}))_\uparrow$.
Now consider $f\in\mc{A}_\uparrow\cap\flow(\mc{A})$.
For such $f$ we have $f\in\Pi^\updownarrow$, which implies $f\in\mc{A}$.
Thus $\mc{A}\preceq\flow(\mc{A})$.

It remains to show that
$\flow(\mc{A})$ is a maximal element of $\mathscr{W}_{\det}$.
Lemma \ref{l:flow_is_weave} gives that $\flow(\mc{A})\in\mathscr{W}_{\det}$.
Suppose that $\mc{B}\in\mathscr{W}_{\det}$, comparable under $\preceq$ to $\flow(\mc{A})$.
We must show that $\mc{B}\preceq\flow(\mc{A})$.
From Lemma \ref{l:preceq_noncr} we have that $\flow(\mc{A})\cup\mc{B}$ is non-crossing.
From what we have already proved we have that $\mc{A}\preceq\flow(\mc{A})$, 
so using Lemma \ref{l:preceq_noncr} again gives that $\mc{A}\cup\flow(\mc{A})$ is non-crossing.
Thus by Lemma \ref{l:noncr_transitive_weave} we have that $\mc{A}\cup\mc{B}$ is non-crossing.
Lemma \ref{l:noncr_same_flow} gives that $\flow(\mc{A})=\flow(\mc{B})$.
From what we have already proved we now have that
$\mc{B}\preceq \flow(\mc{B})=\flow(\mc{A})$, as required.
This completes the proof.
\end{proof}

\begin{lemma}
\label{l:ramification_meas_zero}
Let $\mc{A}$ be a deterministic weave. The set of ramification points of $\mc{A}$ has Lebesgue measure zero.
\end{lemma}
\begin{proof}
By Lemma \ref{l:flow_is_weave} we have that $\flow(\mc{A})=\{f\in\Pi^\updownarrow\-f\text{ does not cross }\mc{A}\}$
is a weave. 
Lemma \ref{l:biinf_ramification_meas_zero} gives that the set of ramification points of $\flow(\mc{A})$ has measure zero.
By Theorem \ref{t:path_extension} any ramification point of $\mc{A}$ is also a ramification point of $\mc{F}$.
The result follows.
\end{proof}

In view of Lemma \ref{l:ramification_meas_zero}, for any deterministic weave $\mc{A}$ 
there exists a dense countable subset $D\sw\Rc$ that is non-ramified. 
In Lemma \ref{l:ramification_meas_zero_det} we will address, for random waves,
how such a subset may be chosen to be a random variable.

\subsection{The web map}
\label{sec:web_op}

In this section we study properties of the $\web$ operation defined in \eqref{eq:web_op}.
Some results in this section are analogues of properties that were proven 
(for the $\flow$ operation) in Section \ref{sec:flow_op}.
We also address the dependence, or rather the lack thereof, of \eqref{eq:web_op} on the set $D$.
The web operation involves both the `downset' operation $A_\uparrow=\{f\in\Pi^\uparrow\-f\sw g\text{ for some }g\in A\}$
and taking closure in $\Pi^\uparrow$,
so we begin with the interaction between these two operations.

\begin{lemma}
\label{l:relcomp_uparrow}
If $A\sw\Pi^\uparrow$ is relatively compact then $A_\uparrow$ is relatively compact and $(\ov{A})_\uparrow=\ov{(A_\uparrow)}$.
\end{lemma}
\begin{proof}
The first claim follows immediately from Proposition \ref{p:relcom_tightness}.
It remains to establish that $(\ov{A})_\uparrow=\ov{(A_\uparrow)}$.
To this end, 
suppose that $f\in(\ov A)_\uparrow$.
Then there exists $(g_n)\sw A$ such that $g_n\to g\in\Pi^\uparrow$ and $f\sw g$.
Let $z\in\Rc$ denote the initial point of $f$.
By Lemma \ref{l:appdx_1_sw_limits} $z\in\ov{\cup_n H(g_n)}$.
Hence, we may pass to a subsequence of the $(g_n)$ and choose $z_n\in H(g_n)$ such that $z_n\to z$.
It follows from Lemma \ref{l:appdx_1_sw_limits} that $g_n|_{z_n}\to g|_{z}=f$,
so $f\in\ov{( A_\uparrow)}$. Thus $(\ov A)_\uparrow \sw \ov{( A_\uparrow)}$.

In preparation for proving the reverse inclusion, let us first show that 
if $B\sw\Pi$ is compact then $B_\uparrow$ is closed.
Take such an $B$, and let $(f_n)\sw \mc{B}_\uparrow$ with $f_n\to f$.
We have $(g_n)\sw B$ such that $f_n\sw g_n$.
By compactness, and passing to a subsequence of $(g_n)$, 
we have that $g_n\to g\in B$. 
By Lemma \ref{l:sw_compat} we have $f\sw g$,
thus $f\in B_\uparrow$,
which establishes that $B_\uparrow$ is closed.

We now show the reverse inclusion.
Suppose that $f\in\ov{(A_\uparrow)}$.
Then there exists $(g_n)\sw A$ and $f_n\sw g_n$ such that $f_n\to f$.
In particular, $g_n\in \ov A$ which implies $f_n\in(\ov A)_\uparrow$.
From the previous paragraph we have that $(\ov A)_\uparrow$ is closed,
so $f\in (\ov A)_\uparrow$.
Thus $(\ov A)_\uparrow \supseteq \ov{( A_\uparrow)}$,
as required.
\end{proof}

\begin{lemma}
\label{l:weave_uparrow}
Suppose that $\mc{A}$ is a deterministic weave. Then $\mc{A}_\uparrow$ is a deterministic weave
and $\mc{A}_\uparrow\preceq\mc{A}$.
\end{lemma}
\begin{proof}
Relative compactness of $\mc{A}_\uparrow$ is given by Lemma \ref{l:relcomp_uparrow}.
Since $\mc{A}$ is a weave, $\mc{A}$ is closed, so from Lemma \ref{l:relcomp_uparrow} we have that $\mc{A}_\uparrow$ is closed,
which thus implies compactness.
The non-crossing property of $\mc{A}_\uparrow$ is inherited from $\mc{A}$ by Definition \ref{d:crossing}.
Using that $\mc{A}_\uparrow=(\mc{A}_\uparrow)_\uparrow$
it is straightforward to check that \eqref{eq:preceq} gives $\mc{A}_\uparrow\preceq\mc{A}$.
\end{proof}


\begin{lemma}
\label{l:web_op_D}
Let $\mc{A}$ be a deterministic weave and $D,D'$ be dense non-ramified subsets of $\R^2$.
Define $\web_D(\mc{A})$ as in \eqref{eq:web_op}.
Then $\web_D(\mc{A})=\web_{D'}(\mc{A})$. 
\end{lemma}
\begin{proof}
Let $\mc{A}$ be a weave and let $D$, $D'$ be dense non-ramified subsets of $\R^2$.
We note that it suffices to prove that
\begin{equation}
\label{eq:web_op_ADAD'}
(\mc{A}|_{D})_\uparrow\sw \ov{(\mc{A}|_{D'})_\uparrow}.
\end{equation}
With \eqref{eq:web_op_ADAD'} in hand,
by symmetry we also have that $(\mc{A}|_{D'})_\uparrow\sw \ov{(\mc{A}_{D})_\uparrow}$.
Taking closures shows that $\web_D(\mc{A})=\web_{D'}(\mc{A})$.

Let $f$ be an element of the left hand side of \eqref{eq:web_op_ADAD'}.
Then there exists $z\in D$ and $g\in\mc{A}(z)$ with $f\sw g|_z$.
Using that $D'$ is dense,
take $z_n\in D'$ such that $z_n\to z$,
and using that $\mc{A}$ is pervasive 
take $g'_n\in\mc{A}(z_n)$.
By compactness of $\mc{A}$ we may pass to a subsequence and assume that
$g'_n\to g\in\mc{A}(z)$,
which by Lemma \ref{l:appdx_1_sw_limits} implies that
$g'_n|_{z_n}\to g|_z$.

Since $g,g'\in\mc{A}(z)$ and $z$ is non-ramified,
we have that $g|_z=g'|_z$.
Thus $f\sw g'|_z$.
Let $w$ denote the initial point of $f$.
Using that $g'_n|_{z_n}\to g'|_z$, 
by Lemma \ref{l:relcom_Pi_KPi} we thus have
$$w\in H(g'|_z)\sw\ov{\bigcup_{n=1}^\infty H(g'_n|_{z_n})}.$$
If $w\in H(g'_n|_{z_n})$ for some $n\in\N$
then we have $f\sw g'_n|_{z_n}\in\mc{A}|_{D'}$,
and we are done.
Otherwise, 
there exists a subsequence of $n$ such that 
$w_n\in H(g'_n|_{z_n})$ and $w_n\to w$,
and without loss of generality we pass to this subsequence.
We thus have $g'_n|_{w_n}\sw g'_n|_{z_n}$,
so $g'_n|_{w_n}\in(\mc{A}_{D'})_{\uparrow}$.
Noting that $w_n\to w$ and $g'_n\to g$,
it follows form Lemma \ref{l:appdx_1_sw_limits} that
$g'_n|_{w_n}\to f$.
Hence $f\in\ov{(\mc{A}_{D'})_{\uparrow}}$.
This establishes \eqref{eq:web_op_ADAD'} and completes the proof.
\end{proof}

\begin{remark}
\label{r:web_op_det}
From the point onwards we will often invoke Lemma \ref{l:web_op_D} implicitly, 
through writing $\web_D(\mc{A})=\web(\mc{A})$.
Lemmas \ref{l:ramification_meas_zero} and Lemma \ref{l:web_op_D} combine to show that 
$\web:\mathscr{W}_{\det}\to\mathscr{W}_{\det}$ is a deterministic function.
\end{remark}

\begin{lemma}
\label{l:noncr_same_web}
Let $\mc{A},\mc{B}$ be deterministic weaves and assume that $\mc{A}\cup\mc{B}$ is non-crossing.
Then 
$\web(\mc{A})=\web(\mc{B})$.
\end{lemma}
\begin{proof}
Since $\mc{A}\cup\mc{B}$ is non-crossing, it is trivial to check that $\mc{A}\cup\mc{B}$ is a weave.
By Lemma \ref{l:ramification_meas_zero} there exists 
$D\sw\R^2$ that is dense and non-ramified with respect to $\mc{A}\cup\mc{B}$.
Note that this implies $D$ is also non-ramified with respect to both $\mc{A}$ and $\mc{B}$.
Fix some $z\in D$ and consider $f\in\mc{A}(z)$ and $g\in\mc{B}(z)$.
Note that $(\mc{A}\cup\mc{B})(z)=\mc{A}(z)\cup\mc{B}(z)$.
We have $f,g\in(\mc{A}\cup\mc{B})(z)$ and $z$ is non-ramified, so we have $f|_{z}=g|_{z}$.
Since $z,f,g$ were arbitrary, by \eqref{eq:web_op} this implies $\web_{D}(\mc{A})=\web_D(\mc{B})$.
\end{proof}

\begin{lemma}
\label{l:web_is_weave}
Let $\mc{A}$ be a deterministic weave.
Then $\web(\mc{A})$ is a deterministic weave. 
\end{lemma}
\begin{proof}
Fix a dense non-ramified $D\sw\R^2$.
By Lemma \ref{l:web_op_D} it suffices to check that $\mc{W}=\web_D(\mc{A})$ is a weave.
Noting that 
$(\mc{A}|_D)_\uparrow\sw A_\uparrow$,
compactness of $\mc{A}$ and Lemma \ref{l:relcomp_uparrow} implies compactness of $\mc{W}$. 
Similarly, $\mc{W}$ inherits the non-crossing property from $\mc{A}$ 
by Lemma \ref{l:weave_uparrow}.
Lastly, for any $z\in \R^2$ there exists $(z_n)\sw D$ such that $z_n\to z$.
Take $f_n\in\mc{A}(z)$ and note $g_n=f_n|_{z_n}\in\mc{W}$.
By compactness $g_n$ has a sub-sequential limit point $g\in\mc{W}$, 
and by Lemma \ref{l:appdx_1_sw_limits} we have $g\in\mc{W}(z)$.
Thus $\mc{W}$ is pervasive,
so we have that $\mc{W}$ is a weave.
\end{proof}

\begin{lemma}
\label{l:web_minimal}
Let $\mc{A}$ be a deterministic weave. 
Then $\web(\mc{A})$ is a minimal element of $(\mathscr{W}_{\det},\preceq)$
and $\web(\mc{A})\preceq\mc{A}$.
\end{lemma}
\begin{proof}
Let $D\sw\R^2$ be non-ramified with respect to $\mc{A}$ and 
let us write $\mc{W}=\web_D(\mc{A})$.
Note that Lemma \ref{l:web_is_weave} gives that $\mc{W}\in\mathscr{W}_{\det}$.
Let us first show that $\mc{W}\preceq\mc{A}$.
It is immediate from \eqref{eq:web_op} that $\mc{W}\sw\ov{(\mc{A}_\uparrow)}$,
so by Lemma \ref{l:relcomp_uparrow} and compactness of $\mc{A}$ we have $\mc{W}\sw \mc{A}_\uparrow$.
Proposition \ref{p:relcom_tightness} implies that $\mc{A}|_D$ is relatively compact,
thus from Lemma \ref{l:relcomp_uparrow} and \eqref{eq:web_op} 
we have 
$\mc{W}_\uparrow
=(\ov{(\mc{A}|_D)_\uparrow})_\uparrow
=\ov{(\mc{A}|_D)_\uparrow}=\mc{W}$,
so trivially $\mc{W}_\uparrow\cap\mc{A}\sw\mc{W}$.
According to \eqref{eq:preceq} we now have $\mc{W}\preceq \mc{A}$.

Suppose that $\mc{B}$ is a deterministic weave, comparable to $\web(\mc{A})$.
We must show that $\web(\mc{A})\preceq \mc{B}$.
The argument is analogous to that of Lemma \ref{l:flow_maximal}.
From Lemma \ref{l:preceq_noncr} we have that $\web(\mc{A})\cup\mc{B}$ is non-crossing.
From what we have already proved we have that $\web(\mc{A})\preceq\mc{A}$ 
so using Lemma \ref{l:preceq_noncr} again gives that $\mc{A}\cup\web(\mc{A})$ is non-crossing.
Thus by Lemma \ref{l:noncr_transitive_weave} we have that $\mc{A}\cup\mc{B}$ is non-crossing.
Lemma \ref{l:noncr_same_web} now gives that $\web(\mc{A})=\web(\mc{B})$.
From what we have already proved we have that
$\web(\mc{A})=\web(\mc{B})\preceq \mc{B}$, as required.
\end{proof}

\section{Random weaves}
\label{sec:weaves_random}

We now turn our attention to random weaves.
We will give the proof of our main results
(stated in Section \ref{sec:results_weaves})
in Sections \ref{sec:proof_t_weave_structure}--\ref{sec:proof_t_cont_weaves}.
We require some technical matters to be dealt with first,
largely concerning measurability,
before we are in a position to rigorously work with random weaves.
The proofs in Sections \ref{sec:meas_2} and \ref{sec:preceq_random}
are not necessary for the reader wishing to understand 
the proofs of our main results in later sections.

\subsection{On measurability}
\label{sec:meas_2} 

To make sense of the statements of our main results in Section \ref{sec:results_weaves} we require that several objects are measurable.
For example, if $\mc{A}$ and $\mc{B}$ are random weaves the we need $\{\mc{A}\preceq\mc{B}\}$ to be an event
and we need $\web(\mc{A})$ and $\flow(\mc{A})$ to be random variables.
In Appendix \ref{sec:meas_1} we show that several basic maps associated to the space $\mc{K}(\Pi^\uparrow)$ are measurable,
whereas here we consider measurability with a focus specific to weaves.
We have also seen in Lemma \ref{l:flow_map_cts} that the map $\mc{A}\mapsto\flow(\mc{A})$ is 
a continuous map from $\mathscr{W}_{\det}$ to itself.
The map $\web(\cdot)$ is not continuous on $\mathscr{W}_{\det}$,
as shown by example in Figure \ref{fig:web_dc}.
We defined $\web$ and $\flow$ as maps with domain $\mathscr{W}_{\det}$
but we check measurability in terms of the topology on $\mc{K}(\Pi^\uparrow)$,
so the following lemma is a technical necessity. 

\begin{lemma}
\label{l:meas_Wdet}
It holds that $\mathscr{W}_{\det}$ is a measurable subset of $\mc{K}(\Pi^\uparrow)$.
\end{lemma}
\begin{proof}
We have
$\mathscr{W}_{\det}=
\{A\in\mc{K}(\Pi^\uparrow)\-A\text{ is pervasive}\} \cap
\{A\in\mc{K}(\Pi^\uparrow)\-A\text{ is non-crossing}\}$.
Using Lemma \ref{l:appdx_1_sw_limits} it is straightforward to check that 
$\{A\in\mc{K}(\Pi^\uparrow)\-A\text{ is pervasive}\}$ is closed.
It remains to show that 
$\{A\in\mc{K}(\Pi^\uparrow)\-A\text{ is non-crossing}\}$
is measurable.

Consider if $A_n,A\in\mc{K}(\Pi)$ are such that $A_n\to A$,
and $A_n$ is non-crossing for each $n$.
Suppose that $A$ fails to be non-crossing,
in particular suppose that $f,g\in A$ cross each other.
We have $f_n,g_n\in\mc{A}_n$ such that $f_n\to f$ and $g_n\to g$.
By Lemma \ref{l:lhd_noncr} we have $f_n\lhd g_n$ or $g_n\lhd f_n$,
at least one of which must hold for infinitely many $n$.
It follows by part 2 of Lemma \ref{l:crossing_eps} that $\mc{A}$ contains a pair of paths $f',g'$
such that $f'\blacktriangleleft_\eps g'$, for some $\eps>0$.
Writing 
\begin{align*}
N &= \{A\in\mc{K}(\Pi^\uparrow)\-A\text{ is non-crossing}\}, \\
M_\eps &= \{A\in\mc{K}(\Pi^\uparrow)\-\text{there exists }f,g\in A\text{ with }f\blacktriangleleft_{1/n} g\} \\
L &= \{A\in\mc{K}(\Pi^\uparrow)\- \text{there exists }f\in A\text{ with }\s_f=-\infty\text{ or } \s_f=+\infty\}.
\end{align*}
we thus obtain
$
\ov{N} = N \cup
\l( \ov{N} \cap \bigcup_{n\in\N} M_{1/n} \r).
$
By part 1 of Lemma \ref{l:crossing_eps} we have $N\cap M_{\eps}=\emptyset$, hence
$N=\ov{N}\sc\l(\bigcup_{n\in\N} M_{1/n}\r).$
Lemma \ref{l:lhd_crossing_eps} implies that $M_{1/n} \cup L$ is closed, for each $n\in\N$,
so noting that $L$ is also closed
we obtain that $M_{1/n}=(M_{1/n} \cup L)\sc L$ is measurable.
Hence $N$ is measurable.
This completes the proof.
\end{proof}

\begin{lemma}
\label{l:meas_web_op}
The map $\mc{A}\mapsto\web(\mc{A})$ is measurable from $\mathscr{W}_{\det}$ to itself.
\end{lemma}
\begin{proof}
From Lemma \ref{l:meas_Wdet} we have that $\mathscr{W}_{\det}$ is measurable.
Recall that in Remark \ref{r:web_op_det}
we noted that for $\mc{A}\in\mathscr{W}_{\det}$ 
the value of $\web_D(\mc{A})$ does not depend upon $D$,
provided that $D\sw\Rc$ is dense and non-ramified.
We thus write $\web(\mc{A})=\web_D(\mc{A})$.

Let $\mu$ be a measure on $\Rc$ with full support and no atoms.
Let $(z_i)_{i=1}^\infty$ be a sequence of independent random variables with distribution $\mu$,
on the probability space $(\Omega,\mc{F},\P)$.
The following argument is somewhat unusual,
so let us give an outline.
We will show that the map from $\mc{A}\in\mathscr{W}_{\det}$ to the law of $\web_D(\mc{A})$ is a measurable map,
and that this law is precisely the probability measure with a point-mass at $\web(\mc{A})$.
The stated result then follows, using that 
the map from point-mass measures, to their associated points, is a measurable map.

From Lemma \ref{l:meas_A|z} the map
$(A,\omega)\mapsto A|_{z_i(\omega)}$
is measurable from $\mc{K}(\Pi)\times\Omega\to \mc{K}(\Pi)$.
It follows from Lemma \ref{l:meas_Auparrow} that
\begin{equation}
\label{eq:meas_web_op_n}
(A,\omega)\mapsto M_n(A,\omega)=\l(\bigcup_{i=1}^n A|_{z_i(\omega)}\r)_\uparrow
\end{equation}
is measurable, for each $n\in\N$,
as a function $M_n:\mathscr{W}_{\det}\times\Omega\to\mathscr{W}_{\det}$.

For all $\mc{A}\in\mathscr{W}_{\det}$ and $\omega\in\Omega$ we have
$M_n(\mc{A},\omega) \sw M_{n+1}(\mc{A},\omega) \sw \mc{A}_{\uparrow}$. 
Lemma \ref{l:relcomp_uparrow} gives that $\mc{A}_\uparrow$ is a compact subset of $\Pi$,
which implies that the sequence $(M_n(\mc{A},\omega))_{n=1}^\infty$ is a relatively compact subset of $\mc{K}(\Pi)$.
Let $\mc{B}$ be a limit point in $\mc{K}(\Pi^\uparrow)$ of $(M_n(\mc{A},\omega))$ as $n\to\infty$.

We aim to show that $\P[B=\web(\mc{A})]=1$.
For each $\mc{A}\in\mathscr{W}_{\det}$ we have that 
\begin{equation}
\label{eq:zi_D_condition}
\P[(z_i)\text{ is dense in }\R^2\text{ and non-ramified in }\mc{A}]=1.
\end{equation}
We write $D=(z_i(\omega))_{i\in\N}$.
Let us condition on the event in \eqref{eq:zi_D_condition} occurring.
Then, if $g\in B$ we have $z_{i_n}\sw\R^2$, $f_n\in \mc{A}(z_{i_n})$ 
and $g_n\in\Pi^\uparrow$ such that $g_n\sw f_n|_{z_{i_n}}$ and $g_n\to g$.
It follows immediately that $g\in\web_D(\mc{A})$, thus $B\sw\web_D(\mc{A})$.
Similarly, if $g\in\web_D(\mc{A})$ then 
there exists $z_i\in\R^2$, $f_i\in A(z_n)$ and $g_i\in\Pi^\uparrow$ such that $g_i\sw f_i$ and $g_i\to g$. 
Hence for all $i$ there exists $n\in\N$ such that $g_i\in M_n(\mc{A},\omega)$. 
Thus $g\in B$ and $B=\web_D(\mc{A})$,
which by Lemma \ref{l:web_op_D} is equal to $\web_D(\mc{A})$.
We thus have $\P[B=\web(\mc{A})]=1$.

Let $\mathscr{P}(\mathscr{W}_{\det})$ denote the space of probability measures on $\mathscr{W}_{\det}$
and let $\mathscr{P}_0(\mathscr{W}_{\det})$ denote the closed subspace of point-mass probability measures.
Let $\delta_A\in \mathscr{P}_0(\mathscr{W}_{\det})$ be the probability measure that is a point-mass on $A\in\mathscr{W}_{\det}$.
Let $\mathscr{L}_n^{\mc{A}}\in\mathscr{P}(\mathscr{W}_{\det})$ denote the law of the random variable $\omega\mapsto M_n(\mc{A},\omega)$.
From what we have proved, it follows that $\mathscr{L}_n^{\mc{A}}$ converges weakly to $\delta_{\web(\mc{A})}$
the probability measure on $\mc{K}(\Pi^\uparrow)$ that is a point-mass on $\web(\mc{A})$.

For measurable $S\sw\mathscr{W}_{\det}$ we have
$$\mc{L}_n^{\mc{A}}(S)=\P[M_n(\mc{A},\cdot)\in S]=\int_{\omega\in\Omega} \1_S(M_n(\mc{A},\omega))\,d\P(\omega),$$
from which it follows that $\mc{A}\mapsto \mc{L}^\mc{A}_n$ is a measurable function
from $\mathscr{W}_{\det}$ to $\mathscr{P}_0(\mathscr{W}_{\det})$.
Hence $\mc{A}\mapsto \delta_{\web(\mc{A})}$ is also measurable.
It is easily seen that the map $m:\mathscr{P}_0(\mathscr{W}_{\det})\to \mathscr{W}_{\det}$
given by $\delta_A\mapsto A$ is continuous, and thus measurable.
Compositions of measurable functions are measurable, hence the map
$A\mapsto m(\delta_{\web{A}})=\web(\mc{A})$
is measurable.
\end{proof}

\begin{lemma}
\label{l:preceq_meas_simplex}
The set
$\{(A,B)\in\mc{K}(\Pi)^2\- A\preceq B\}$
is a measurable subset of $\mc{K}(\Pi)^2$.
\end{lemma}
\begin{proof}
Recall the definition of $\preceq$ on $\mc{K}(\Pi)$ from \eqref{eq:preceq}.
It suffices to show that
\begin{align*}
C_1&=\{(A,B)\in\mc{K}(\Pi)^2\-A\sw B_\uparrow\} \\
C_2&=\{(A,B)\in\mc{K}(\Pi)^2\-B\cap A_\uparrow\sw A\}
\end{align*}
are both measurable subsets of $\mc{K}(\Pi)^2$,
where $\mc{K}(\Pi)^2$ is equipped with the product topology and corresponding Borel $\sigma$-field.
We note that $A_\uparrow$ is closed whenever $A\in\mc{K}(\Pi)$ is closed.
Using Lemma \ref{l:appdx_1_sw_limits} it is straightforward to check that 
the set $C_1=\{(A,B)\in\mc{K}(\Pi)^2\-\forall f\in A\;\exists g\in B\text{ such that }f\sw g\}$ is closed.
We now move on to $C_2$.
To this end note that
$B\cap A_\uparrow\nsubseteq A$ if and only if
there exists $f\in A_\uparrow$ and $g\in B$ such that $g\sw f$ and $f\notin A$.
The condition $f\notin A$ is equivalent to $d_{\mc{K}(\Pi)}(\{f\},A)>0$ which,
as $A$ is closed,
is in turn equivalent to $d_{\mc{K}(\Pi)}(\{f\},A)\geq\epsilon$ for some $\epsilon>0$.
We thus have that
$C_2=\bigcup_{n\in\N} S_{1/n}$
where 
$$
S_\epsilon=\l\{(A,B)\in\mc{K}(\Pi)^2\- \exists f\in A_\uparrow, g\in B\text{ such that }
g\sw f\text{ and }d_{\mc{K}(\Pi)}(\{f\},A)\geq\epsilon\r\}.
$$
Similar to above, using Lemma \ref{l:appdx_1_sw_limits} it is straightforward to check that 
$S_\epsilon$ is closed, for any $\epsilon>0$.
Thus $C_2$ is measurable.
\end{proof}


\subsection{On partial ordering of random weaves}
\label{sec:preceq_random}

In this section we show that 
$\preceqd$ is a partial order on (the laws of) random weaves, 
as defined shortly below \eqref{eq:preceq}.
More precisely, recall that $\mc{P}(M)$ denotes the space of probability measures on a metric space $M$.
We have shown in Lemma \ref{l:preceq_deterministic} that $\preceq$ given by \eqref{eq:preceq} defines a partial order on 
$\mc{K}(\Pi^\uparrow)$.
In Section \ref{sec:terminology} we defined an extension of $\preceq$ to $\mc{P}(\mc{K}(\Pi^\uparrow))$,
namely if $\mc{A}$ and $\mc{B}$ are $\mc{K}(\Pi^\uparrow)$ valued random variables then we write $\mc{A}\preceqd\mc{B}$
if there exists a coupling of $\mc{A}$ and $\mc{B}$ such that $\P[\mc{A}\preceq\mc{B}]=1$.
We aim to show that that $\preceqd$ is a partial order on $\mc{P}(\mc{K}(\Pi^\uparrow))$.

If $\preceq$ was compatible with $(\mc{K}(\Pi),d_\Pi)$,
in the sense of Definition \ref{d:compatible},
then we could use a classical result e.g.~Theorem 2.4 in \cite{Liggett1985} to 
obtain the extension to $\mc{P}(\mc{K}(\Pi))$.
However, as we saw in Remark \ref{r:preceq_not_compatible}
compatibility fails in this situation.
Instead we require an original argument 
that uses compactness and the precise form of \eqref{eq:preceq}.
We first give a preliminary lemma.

\begin{lemma}
\label{l:preceq_antisym}
Suppose that $\mc{D},\mc{D}'$ are $\mc{K}(\Pi^\uparrow)$ valued random variables, with the same marginal distributions,
coupled such that $\P[\mc{D}\preceq\mc{D}']=1$. 
Then $\P[\mc{D}=\mc{D}']=1.$
\end{lemma}

\begin{proof}
By Proposition \ref{p:J1M1} the metric space $\Pi^\uparrow$ is separable, 
which implies that its topology has a countable base:
there exists a family $(U_i)_{i\in\N}$ of non-empty open subsets of $\Pi^\uparrow$
such that any open subset $O\sw\Pi$ can be written as $O=\cup_{i\in I} U_i$ for some $I\sw \N$.

Assume the conditions of the lemma on $\mc{D},\mc{D}'$.
The proof comes in two parts, corresponding respectively to the inequalities $\P[\mc{D}'\sc \mc{D}\neq\emptyset]>0$ and $\P[\mc{D}\sc \mc{D}'\neq\emptyset]>0$, each of which will be shown to be impossible through an argument by contradiction.

\textbf{Part 1.} Suppose $\P[\mc{D}'\sc \mc{D}\neq\emptyset]>0$.
The reader may wish to glance at Remark \ref{r:preceq_random_explanation},
immediately below the present proof,
for a toy example to illustrate our strategy here.
For $A\sw\Pi$, we write 
\begin{align}
A^\circ &= \{b\in\Pi^\uparrow\-\text{there exists }a\in A\text{ such that }a\sw b\}. \label{eq:project_up} 
\end{align}
On the event that $\{\mc{D}'\sc \mc{D}\}$,
let $d'\in \mc{D}'\sc \mc{D}$ and let $\mc{B}_\epsilon(d')$ be the open ball in $\Pi^\uparrow$ of radius $\epsilon$ about $d'$.
We will now show that, almost surely, $\mc{B}_\epsilon(d')^\circ\cap \mc{D}$ is empty, for sufficiently small $\epsilon>0$.
Suppose that $\mc{B}_\epsilon(d')^\circ\cap \mc{D}\neq\emptyset$ for all $\epsilon>0$.
Then, taking $\epsilon=1/n$, we have sequences $f_n\in\mc{B}_{1/n}(d')$ and $g_n\in \mc{B}_{1/n}(d')^\circ\cap \mc{D}$, 
with $f_n\sw g_n$.
By compactness of $\mc{D}$ we may pass to subsequence and assume convergence $f_n\to d'\in \mc{D}'$ and $g_n\to d\in \mc{D}$.
By Lemma \ref{l:appdx_1_sw_limits} we then have $d'\sw d$, so $d'\in\mc{D}_\uparrow$.
We have $\P[\mc{D}\preceq\mc{D}']=1$ upon which event $\mc{D}_\uparrow\cap\mc{D}'\sw\mc{D}$, 
so $d'\in\mc{D}$ which is a contradiction.
Thus, almost surely, for some (random) $\epsilon>0$, we have $\mc{B}_\epsilon(d')^\circ\cap \mc{D}=\emptyset$.
Clearly also $d'\in\mc{B}_\epsilon(d')$.


Let 
\begin{equation*}
O=
\begin{cases}
\mc{B}_\epsilon(d') & \text{ on the event that } \mc{D}'\sc \mc{D}\neq\emptyset \\
\emptyset & \text{ otherwise}.
\end{cases}
\end{equation*}
From the previous paragraph have that $O^\circ\cap \mc{D}=\emptyset$ and with positive probability $d'\in O$.
Since $O$ is open, almost surely we may write  $O=\cup_{i\in I} U_i$ for some random $I\sw\N$.
The set $I$ is non-empty with positive probability, 
hence there is some deterministic $i\in I$ such that with positive probability $d'\in U_i\sw O$.
Let us write $U=U_i$ for such an $i$. 

On the event that $d'\in U\sw O$ we have that $U^\circ\sw O^\circ$, 
which implies that $U^\circ\cap \mc{D}=\emptyset$ (because $O^\circ\cap\mc{D}=\emptyset$)
and $U^\circ\cap \mc{D}'\neq\emptyset$ (because it contains $d'$).
Hence the event $\{U^\circ\cap \mc{D}'\neq\emptyset\text{ and }U^\circ\cap \mc{D}=\emptyset\}$ has positive probability.
We thus have
\begin{align}
0<
\P\l[U^\circ\cap \mc{D}'\neq\emptyset\text{ and }U^\circ\cap \mc{D}=\emptyset\r] \notag
&=\P\l[U^\circ\cap \mc{D}'\neq\emptyset\r]-\P\l[U^\circ\cap \mc{D}'\neq \emptyset\text{ and }U^\circ\cap \mc{D}\neq \emptyset\r] \notag \\
&=\P\l[U^\circ\cap \mc{D}\neq\emptyset\r]-\P\l[U^\circ\cap\mc{D}'\neq \emptyset\text{ and }U^\circ\cap \mc{D}\neq \emptyset\r] \notag \\
&= \P\l[U^\circ\cap \mc{D}\neq\emptyset\text{ and }U^\circ\cap \mc{D}'=\emptyset\r]. \label{eq:U_swap}
\end{align}
The second line of \eqref{eq:U_swap} follows because $\mc{D}$ and $\mc{D}'$ have the same marginal distribution, and the other steps are elementary.

Consider when the event $\{U^\circ\cap \mc{D}\neq\emptyset\text{ and }U^\circ\cap \mc{D}'=\emptyset\}$ occurs,
which by \eqref{eq:U_swap} has positive probability.
Then we have $h\in U^\circ\cap \mc{D}$, 
but $\P[\mc{D}\preceq\mc{D}']=1$ upon which event we have $\mc{D}\sw\mc{D}'_\uparrow$, 
hence there exists $h'\in \mc{D}'$ such that $h\sw h'$.
By \eqref{eq:project_up} we have that $h'\in U^\circ$,
which is a contradiction to $U^\circ\cap \mc{D}'=\emptyset$.
Hence in fact $\P[\mc{D}'\sc \mc{D}\neq\emptyset]=0$, as required.

\textbf{Part 2.} Suppose $\P[D\sc D'\neq\emptyset]>0$.
The argument is similar to Case 1 but somewhat simpler, and we will make use of Case 1 within it.
Note that we should expect an asymmetric argument due to the parity inherent in $\P[\mc{D}\preceq\mc{D}']=1$.
To make the comparison clear we will recycle much of our notation.

On the event $\mc{D}\sc \mc{D}'\neq\emptyset$, take $d\in \mc{D}\sc \mc{D}'$.
Suppose that $B_{\epsilon}(d)\cap \mc{D}'\neq\emptyset$ for all $\epsilon>0$.
Taking $\epsilon=1/n$, there exists $f_n\in \mc{D}'\cap B_{1/n}(d)$.
By compactness we may pass to a subsequence and assume convergence $f_n\to f\in \mc{D}'$. 
This implies $f=d$, which is a contradiction since $f\in\mc{D}'$.
Hence there exists a random $\epsilon>0$ such that
$B_{\epsilon}(d)\cap \mc{D}'=\emptyset$.

Let $O$ be equal to $B_{\epsilon}(d)$ on the event $\{\mc{D}\sc \mc{D}'\neq\emptyset\}$ and $O=\emptyset$ otherwise.
Thus $O\cap \mc{D}'=\emptyset$ and with positive probability $d\in O$.
By the same argument as in Part 1, there exists deterministic $i\in\N$ such that with positive probability $d\in U_i\sw O$.
Thus, setting $U=U_i$, we have that with positive probability $d\in U\cap (\mc{D}\sc \mc{D}')$. 
We have $\P[\mc{D}\preceq\mc{D}']=1$ upon which event $\mc{D}\sw\mc{D}'_\uparrow$.
Thus, when $d\in U\cap (\mc{D}\sc \mc{D}')$ there exists $d'\in \mc{D}'$ such that $d\sw d'$, implying that both
$U^\circ\cap (\mc{D}\sc \mc{D}')\neq \emptyset$ and $U^\circ\cap \mc{D}'\neq\emptyset$.
From Part 1 we have that almost surely $\mc{D}'\sw \mc{D}$.
Hence with positive probability we have
both $U^\circ\cap (\mc{D}\sc \mc{D}')\neq \emptyset$ and $U^\circ\cap (\mc{D}'\cap \mc{D})\neq\emptyset$.
Thus, noting that $\mc{D}=(\mc{D}\cap \mc{D}')\cup(\mc{D}\sc \mc{D}')$,
\begin{align*}
0 
< \P\l[U^\circ\cap(\mc{D}\cap \mc{D}')\neq\emptyset\r] 
&< \P\l[U^\circ\cap \mc{D}\neq\emptyset\r] \\
&= \P\l[U^\circ\cap \mc{D}'\neq\emptyset\r].
\end{align*}
Here, the second line follows because $\mc{D}$ and $\mc{D}'$ have identical distribution and $U$ is deterministic.
It follows that $\P[U^\circ \cap (\mc{D}'\sc \mc{D})\neq\emptyset]>0$, 
but from Part 1 we know that $\P[\mc{D}'\sc \mc{D}=\emptyset]=1$, so we have reached a contradiction.
This completes the proof.
\end{proof}

\begin{remark}
\label{r:preceq_random_explanation}
The proof of Lemma \ref{l:preceq_antisym} is technical but it has a simple idea at its heart.
Consider a toy example: take two uniform random variables $X,X'$ on $S=\{1,2,3,4,5,6\}$
and suppose that $X$ and $X'$ are coupled in a way that satisfies $\P[X\leq X']=1$.
We aim to show that $\P[X=X']=1$.
For $k\in S$,
\begin{align}
\P[X=k,X'\neq k]
&=\P[X=k]-\P[X=k,X'=k] \notag\\
&=\P[X'=k]-\P[X'=k,X=k] \notag\\
&=\P[X'=k,X\neq k]. \label{eq:toy_swap}
\end{align}
Note the similarity of \eqref{eq:toy_swap} to \eqref{eq:U_swap}.
Taking $k=1$ and using that $\P[X\leq X']=1$ we obtain $\P[X=1, X'>1, X\leq X']=\P[X'=1, X>1, X\leq X']$
which becomes $\P[X=1,X'>1]=0$, thus $\P[X=1]=\P[X=1,X'=1]$.
This is clearly a step in the right direction and is
in similar style to the more complex reasoning involving $\preceq$ below \eqref{eq:U_swap}.
The finiteness of $S$ is also helpful here
whereas in Lemma \ref{l:preceq_random} we must rely on second countability of $\Pi^\uparrow$.
We leave it for the reader to complete this toy example and deduce that $\P[X=X']=1$.
\end{remark}

\begin{lemma}
\label{l:preceq_random}
The relation $\preceqd$ is a partial order on the space of $\mc{K}(\Pi^\uparrow)$ valued random variables.
\end{lemma}

\begin{proof}
We will check that $\preceqd$ on $\mc{P}(\mc{K}(\Pi^\uparrow))$ is reflexive, transitive and antisymmetric, in turn.
Lemma \ref{l:preceq_deterministic} has already shown that these properties hold in the deterministic case
i.e.~$\preceqd$ is a partial order on $\mathscr{W}_{\det}$.

By reflexivity of $\preceq$ on $\mc{W}_{\det}$ we have that $\P[\mc{A}\preceq\mc{A}]=1$ for any $\mc{K}(\Pi^\uparrow)$ valued random variable $\mc{A}$, so $\preceqd$ is reflexive.
For transitivity, let us assume that $\mc{A},\mc{B},\mc{C}$ are $\mc{K}(\Pi^\uparrow)$ valued random variables,
and that we have couplings $(\mc{A},\mc{B})$ and $(\mc{B},\mc{C})$ such that $\P[\mc{A}\preceq\mc{B}]=1$ and (on a possibly different probability space) $\P[\mc{B}\preceq\mc{C}]=1$. 
It follows that there exists a joint coupling $(\mc{A},\mc{B},\mc{C})$ on which
$\P[\mc{A}\preceq\mc{B}\text{ and }\mc{B}\preceq\mc{C}]=1$.
On the event $\{\mc{A}\preceq\mc{B}\text{ and }\mc{B}\preceq\mc{C}\}$ transitivity of $\preceq$ on $\mc{W}_{\det}$ implies that $\mc{A}\preceq\mc{C}$, so we obtain $\P[\mc{A}\preceq\mc{C}]=1$, as required.

It remains to show antisymmetry.
Suppose that $\mc{A},\mc{A}',\mc{B},\mc{B}'$ are $\mc{K}(\Pi^\uparrow)$ valued random variables
such that $\P[\mc{A}\preceq\mc{B}]=1$ and $\P[\mc{B}'\preceq\mc{A}']=1$, 
where $\mc{A}$ and $\mc{A}'$ have the same marginal distribution 
and $\mc{B}$ and $\mc{B}'$ have the same marginal distribution.
We must show that there exists a coupling under which $\P[\mc{A}=\mc{B}]=1$.
Since $\mc{A}$ and $\mc{A}'$ have the same marginal distribution, 
it follows that exists a coupling $(\mc{A},\mc{A}',\mc{B},\mc{B}')$ such that
$\P[\mc{A}=\mc{A}',\; \mc{A}\preceq\mc{B},\; \mc{B}'\preceq\mc{A}']=1$.
By transitivity of $\preceq$  on $\mathscr{W}_{\det}$ this means that $\P[\mc{B}'\preceq\mc{B}]=1$.
By Lemma \ref{l:preceq_antisym} we have $\P[\mc{B}=\mc{B}']=1$ which means 
$\P[\mc{A}\preceq\mc{B}, \mc{B}\preceq\mc{A}]=1$ and by antisymmetry of $\preceq$ on $\mc{W}_{\det}$ we 
obtain that $\P[\mc{A}=\mc{B}]=1$, as required.
\end{proof}

\begin{lemma}
\label{l:det_vs_random_webs_flows}
Let $\mc{A}$ be a weave.
Then $\mc{A}$ is a web if and only if $\P[\mc{A}$ is a minimal element of $\mathscr{W}_{\det}]=1$.
Similarly, $\mc{A}$ is a flow if and only if $\P[\mc{A}$ is a maximal element of $\mathscr{W}_{\det}]=1$.
\end{lemma}

\begin{proof}
Let us first give the argument for webs.
Let $\mc{A}$ be a weave. 
We must show that $\mc{A}$ is a web if and only if $\P[\mc{A}$ is a minimal element of $\mc{W}_{\det}]=1$.
It is trivial to see that almost sure pervasiveness and the non-crossing property pass from either side of the `if and only if' statement to the other side, so it remains only to handle minimality.
To be explicit we must show that for a random weave $\mc{A}$ the following statements are equivalent:
\begin{enumerate}
\item If $\mc{B}$ is a weave and there exists a coupling between $\mc{A}$ and $\mc{B}$ such that 
$\P[\mc{A}\preceq\mc{B}]=1$ or $\P[\mc{B}\preceq\mc{A}]=1$, then $\P[\mc{A}\preceq\mc{B}]=1$.
\item $\P[\mc{A}\text{ is a minimal element of }\mc{W}_{\det}]=1$.
\end{enumerate}
By the definition of $\preceqd$ from below \eqref{eq:preceq},
the first statement is precisely the claim that the law of $\mc{A}$ is minimal in $\mc{P}(\mc{K}(\Pi^\uparrow))$.

Let us first show that (2) implies (1). If $\mc{A}$ is almost surely minimal in $\mc{W}_{\det}$ then
for any coupling of $\mc{A}$ to another weave $\mc{B}$, 
$\{\mc{A}\text{ and }\mc{B}\text{ are comparable}\}\sw\{\mc{A}\preceq\mc{B}\}$,
where both the left and right hand side are events.
Thus (1) holds.

Conversely, let us assume (1).
Let
$$
\mc{A}'=
\begin{cases}
\mc{A} & \text{ on the event }\{\web(\mc{A})\prec\mc{A}\}, \\
\web(\mc{A}) & \text{ on the event }\{\mc{A}\preceq\web(\mc{A})\}.
\end{cases}
$$
Lemmas \ref{l:meas_web_op} and \ref{l:preceq_random} combine to show that $\mc{A}'$ is a random variable.
By Lemma \ref{l:web_is_weave} we have $\P[\web(\mc{A})\preceq \mc{A}']=1$, which by (1) implies that $\P[\mc{A}'\preceq \web(\mc{A})]=1$,
so in fact $\P[\mc{A}\preceq\web(\mc{A})]=1$.
By Lemma \ref{l:web_is_weave} we thus have $\P[\mc{A}=\web(\mc{A})]=1$,
from which Lemma \ref{l:web_minimal} gives (2).

In the case of flows we may use a similar argument,
reversing the direction of the sign of $\preceq$.
Lemma \ref{l:web_is_weave} is replaced by Lemma \ref{l:flow_is_weave}, and
Lemma \ref{l:web_minimal} is replaced by Lemma \ref{l:flow_maximal}.
We leave the details to the reader.
\end{proof}


\begin{remark}
Suppose that $\mc{A}$ is a weave.
Our results in Sections \ref{sec:meas_2} and \ref{sec:preceq_random} justify that 
$\web(\mc{A})$ and $\flow(\mc{A})$ are random variables, 
and that if $\mc{B}$ some other random weave, coupled to $\mc{A}$ then $\{\mc{A}\preceq\mc{B}\}$ is an event.
Moreover, $\preceqd$ defines a partial order on the laws of random weaves,
via the relationship 
$\mc{A}\preceqd\mc{B}$ 
if and only if 
there exists a coupling of $\mc{A}$ and $\mc{B}$ such that
$\P[\mc{A}\preceq\mc{B}]=1$.
We will use these results freely from now on 
and will not repeatedly cite them when used within the proofs.
\end{remark}

\subsection{Proof of Theorem \ref{t:weave_structure}}
\label{sec:proof_t_weave_structure}

We are now ready 
to establish our main results concerning random weaves,
which henceforth are simply referred to as weaves.
These results were stated in Section \ref{sec:results_weaves}
and the proofs are spread across Sections \ref{sec:proof_t_weave_structure}--\ref{sec:proof_t_cont_weaves}.
The statements of Theorems \ref{t:weave_structure}-\ref{t:weave_conv} consist of several (numbered) parts.
We will use bold text (see e.g.~the next paragraph) to track when each part is addressed.
Most of our work in Section \ref{sec:weaves_det} leads towards these proofs.
We begin with Theorem \ref{t:weave_structure}, of which
we prove the four statements of the theorem in turn.
Let $\mc{A}$ be a weave.

\textbf{Part 1.}
Let us first assume (a), that $\mc{A}$ is a web.
By Lemma \ref{l:det_vs_random_webs_flows} 
$\mc{A}$ is almost surely a minimal element of $\mathscr{W}_{\det}$.
By Lemma \ref{l:web_minimal} we have $\web(\mc{A})\preceq \mc{A}$,
from which minimality implies that $\mc{A}=\web(\mc{A})$, which gives (b).
Conversely, let us assume (b), that $\mc{A}\eqas\web(\mc{A})$.
Lemma \ref{l:web_minimal} thus gives that $\mc{A}$ is almost surely a minimal element of $\mathscr{W}_{\det}$,
from which Lemma \ref{l:det_vs_random_webs_flows} gives that $\mc{A}$ is a web.
Thus $(a)\iff(b)$.

\textbf{Part 2.}
Again, let $\mc{A}$ be a weave.
We will show that $(a)\ra(b)\ra(c)\ra(a)$.
Let us first assume (a), that $\mc{A}$ is a flow.
By Lemma \ref{l:det_vs_random_webs_flows} 
$\mc{A}$ is almost surely a maximal element of $\mathscr{W}_{\det}$.
By Lemma \ref{l:flow_maximal} we have $\mc{A}\preceq\flow(\mc{A})$,
from which maximality implies that $\mc{A}=\flow(\mc{A})$, giving (b).
Now let us assume (b), that $\mc{A}\eqas\flow(\mc{A})$.
It is immediate from \eqref{eq:flow_op} that $\flow(\mc{A})\sw\Pi^\updownarrow$,
so we have (c).

Lastly, let us assume (c), that $\mc{A}\sw\Pi^\updownarrow$.
Suppose that $\mc{B}$ is a weave with a coupling to $\mc{A}$ such that $\P[\mc{A}\preceq\mc{B}]=1$.
To see that $\mc{A}$ is a flow we must show that this implies $\P[\mc{A}=\mc{B}]=1$.
Using that $\mc{A}\preceq\mc{B}$ almost surely,
it follows from \eqref{eq:preceq} that almost surely
$\mc{A}_\uparrow\cap\mc{B}\sw\mc{A}\sw\mc{B}_\uparrow$.
As $\mc{A}\in\Pi^\updownarrow$ we thus have that almost surely $\mc{A}\sw\mc{B}$.
We require the reverse inclusion, so let $f\in\mc{B}$.
On the almost sure event that $\mc{A}\preceq\mc{B}$,
by Lemma \ref{l:preceq_noncr} we have that $f$ does not cross $\mc{A}$,
so by Theorem \ref{t:path_extension} there exists $f'\in\flow(\mc{A})$
such that $f\sw f'$.
Lemma \ref{l:biinf_noncr} gives that $f'\in\mc{A}$,
which implies that $f\in\mc{A}_\uparrow\cap\mc{B}$.
Thus $\mc{A}\eqas\mc{B}$, as required.

\textbf{Part 3.}
Lemmas \ref{l:flow_maximal} and \ref{l:web_minimal} give that 
$\P[\web(\mc{A})\preceq \mc{A}\preceq\flow(\mc{A})]=1$.

\textbf{Part 4.}
The existence claim is established by part 3 of the present proof.
It remains to prove the uniqueness claim,
which we will give in turn for webs and then flows.

Let $\mc{W},\mc{W}'$ be webs
and suppose that
$\mc{W}\preceqd\mc{A}$ and $\mc{W}'\preceqd\mc{A}$.
Then there exists (pairwise) couplings such that 
$\P[\mc{W}\preceq\mc{A}]=1$ and $\P[\mc{W}'\preceq\mc{A}]=1$.
We seek to show that $\mc{W}\eqd \mc{W}'$.
It follows that there exists a three-way coupling of $\mc{W},\mc{W}'$ and $\mc{A}$ such that
$\P[\mc{W}\preceq\mc{A}\text{ and }\mc{W}'\preceq\mc{A}]=1$.
By Lemma \ref{l:preceq_noncr} we have that $\mc{W}\cup\mc{A}$ is almost surely non-crossing,
and $\mc{W}'\cup\mc{A}$ is almost surely non-crossing.
By Lemma \ref{l:noncr_transitive_weave} we have that $\mc{W}\cup\mc{W}'$ is almost surely non-crossing.
By Lemma \ref{l:noncr_same_web} we thus have $\web(\mc{W})\eqas\web(\mc{W}')$.
By part 2 of the present proof we thus have $\mc{W}\stackrel{a.s}{=}\mc{W}'$,
hence in particular $\mc{W}$ and $\mc{W}'$ have the same marginal distribution, as required.

It remains to prove a corresponding statement for flows.
Let $\mc{F},\mc{F}'$ be flows
and suppose that
$\mc{A}\preceqd\mc{F}$ and $\mc{A}\preceqd\mc{F}'$.
Then, as above, 
there exists a three-way coupling of $\mc{F},\mc{F}'$ and $\mc{A}$ such that
$\P[\mc{A}\preceq\mc{F}\text{ and }\mc{A}\preceq\mc{F}']=1$.
By the same argument as above, 
again using Lemmas \ref{l:preceq_noncr} and \ref{l:noncr_transitive_weave},
with Lemma \ref{l:noncr_same_flow} in place of Lemma \ref{l:noncr_same_web},
and using part 1 of the present proof in place of part 2,
we obtain that $\mc{F}\eqas\mc{F}'$.
Hence in particular $\mc{F}$ and $\mc{F}'$ have the same marginal distribution, as required.
This completes the proof.

\subsection{Proof of Theorem \ref{t:weaves_characterization}}
\label{sec:proof_l_weaves_equiv}

We require some preparatory lemmas before giving the proof of Theorem \ref{t:weaves_characterization}.

The next lemma gives us the ability use a 
\textit{deterministic} non-ramified dense set of space-time points with (random) weaves.
It is a straightforward consequence of Lemma \ref{l:ramification_meas_zero},
delayed until now because when we stated Lemma \ref{l:ramification_meas_zero} we were focused on deterministic weaves.

\begin{lemma}
\label{l:ramification_meas_zero_det}
Let $\mc{A}$ be a weave.
Then the set $\{z\in\Rc\-\P[z\text{ is ramified in }\mc{A}]>0\}$ has zero Lebesgue measure.
\end{lemma}
\begin{proof}
Let $\ram(\mc{A})$ denote the set of ramification points of a (deterministic or random) weave $\mc{A}$.
We have shown in Lemma \ref{l:ramification_meas_zero}
that the map $(\mc{A},z)\mapsto\1\{z\in\ram(\mc{A})\}$ is measurable from $\mathscr{W}_{\det}\times\Rc\to\{0,1\}$,
and that $\ram(\mc{A})$ is Lebesgue null for all deterministic weaves.
By Fubini's theorem, for any weave $\mc{A}$ we have
$
\int_{\Rc} \P[z\in\ram(\mc{A})]\,dz
=\E\big[\int_{\Rc}\1\{z\in\ram(\mc{A})\}\,dz\big]=0
$
and the result follows.
\end{proof}

\begin{lemma}
\label{l:biinf_insertion}
Let $A\sw\Pi^\updownarrow$ be non-crossing and pervasive
and let $D\sw\R^2$ be dense. 
Let $f,h\in A$ with $f\lhd h$ and suppose $t\star\in\R_\mfs$ is such that $f(t\star)<h(t\star)$.
Then for all $\epsilon>0$ there exists $(x,s)\in D$ and $g\in A((x,s))$ such that $f,h\notin A((x,s))$, $f\lhd g$, $g\lhd h$ and $|t-s|<\epsilon$.
Moreover, if $\star=-$ then we may take $s<t$ and if $\star=+$ then we may take $s>t$.
\end{lemma}
\begin{proof}
By the {\cadlag} property of $f,h$ and denseness of $D$, there exists $(x,s)\in D$ such that
$|s-t|<\epsilon$, with the desired sign for $t-s$ and with $f(s-)\vee f(s+)<x<h(s-)\wedge h(s+)$.
Since $\mc{A}$ is pervasive there exists $g\in A((x,s))$.
It is immediate that $f,h\notin A((x,s))$,
and that $\{f,g,h\}$ is non-crossing.
By Lemmas \ref{l:lhd_noncr} and \ref{l:lhd_left_right} we have $f\lhd g$ and $g\lhd h$.
\end{proof}

Lemma \ref{l:biinf_insertion} is a technical lemma used in the proof of our next lemma.
Recall that in Lemma \ref{l:ramification_meas_zero} we showed that 
if $\mc{A}\in\mathscr{W}_{\det}$ then $\ram(\mc{A})$,
the set of ramification points of $\mc{A}$,
is a measurable and null subset of $\R^2_c$.
The following lemma is stated as a result for deterministic weaves,
which avoids having to find a suitable state space for random null sets.

\begin{lemma}
\label{l:ramific_flow_web_same}
Let $\mc{A}$ be a deterministic weave. 
Then $\mc{A}$, $\flow(\mc{A})$ and $\web(\mc{A})$
all have the same ramification points.
\end{lemma}
\begin{proof}
Let us write $\mc{W}=\web(\mc{A})$ and $\mc{F}=\flow(\mc{A})$.
By Theorem \ref{t:weave_structure}, 
applied to the weave whose law is a point-mass at $\mc{A}$,
we have $\mc{W}\preceq\mc{A}\preceq\mc{F}$.
Note that if $\mc{B},\mc{B}'\in\mathscr{W}_{\det}$ with $\mc{B}\preceq\mc{B}'$
then for any $b\in B$ there exists $b'\in B'$ such that $b\sw b'$.
It follows that
$\ram(\mc{W})\sw\ram(\mc{A})\sw\ram(\mc{F})$.
With this in hand it remains only to show that
for any deterministic weave $\mc{A}$ we have
\begin{equation}
\label{eq:ram_equal_det}
\ram(\mc{F})\sw \ram(\mc{W}).
\end{equation}
To this end,
let $D\sw\Rc$ be dense and non-ramified.
Suppose that $(x,t)\in\Rc$ is ramified in $\mc{F}$.
We thus have bi-infinite $f,g\in\mc{F}(z)$ that are not comparable under $\sw$.
It follows that there exists $s\star\in\R_\mfs$ such that $f(s\star)\neq g(s\star)$, 
and without loss of generality we may assume $f(s\star)<g(s\star)$.

\begin{itemize}
\item 
Consider first if $s\star\geq t+$. 
Take a sequence $z_n=(x_n,t_n)\in D$ such that $z_n\to z$ and $x_n\leq f(t_n-)\wedge f(t_n+)$,
along with a sequence $w_n=(y_n,s_n)\in D$ such that $w_n\to z$ and $y_n\geq g(t_n-)\vee g(t_n+)$.
It it straightforward to check that such sequences exist.
Take $f_n\in\mc{F}(z_n)$ and $g_n\in\mc{F}(w_n)$.
By compactness, passing to a subsequence, we may assume that $f_n\to f'$ and $g_n\to g'$ where $f',g'\in\mc{F}$.
By Lemma \ref{l:appdx_1_sw_limits} we have $f',g'\in\mc{F}(z)$.
Recall from Lemma \ref{l:Amax_order} that $(\mc{F},\lhd)$ is totally ordered.
By Lemma \ref{l:biinf_vs_01} we have $f'\lhd f\lhd g\lhd g'$,
which implies that $f'(s\star)<g'(s\star)$.
By Lemma \ref{l:appdx_1_sw_limits} we have $f_n|_{z_n}\to f'|_{z}$ and $g_n|_{w_n}\to g'|_{z}$.
Note that $f'|_{z}, g'|_{z}\in\mc{W}$ by \eqref{eq:web_op},
and both pass though $z$.
Since $f'|_{z}(s\star)=f'(s\star)<g'(s\star)=g'|_{z}(s\star)$
they cannot be comparable under $\sw$.
Hence, in this case, we have that $z\in\ram(\mc{W})$.

\item
Next, consider if $s\star\leq t-$.
We may assume that $f(u\bullet)=g(u\bullet)$ for all $u\bullet\geq t+$ (or else, the case above applies).
By Lemma \ref{l:biinf_insertion} there exists $z\in D$ and $h\in\mc{A}(z)$ such that $f,g\notin\mc{F}(z)$,  $f\lhd h\lhd g$
and $\sigma_{z}<t$. 
We have $f(u\bullet)=g(u\bullet)$ for all $u\bullet\geq t+$, and $f\lhd h\lhd g$,
which means $f(u\bullet)=h(u\bullet)=g(u\bullet)$ for all such $u\bullet$.
Since $f(t-)\leq h(t-)\leq g(t-)$ we have $h\in\mc{A}(z)$.
The properties of $h$ given in Lemma \ref{l:biinf_insertion} guarantee that $h$ is not equal to $f$ or $g$, 
so there exists some $v_1\bullet_1,v_2\bullet_2\leq t-$ such that 
$f(v_1\bullet_1)<h(v_1\bullet_1)$ and $h(v_2\bullet_2)< g(v_2\bullet_2)$.

We apply Lemma \ref{l:biinf_insertion} twice more, to $(f,h)$ at $v_1\bullet_1$ and to $(h,g)$ at $v_2\bullet_2$.
We thus obtain (respectively) for $i=1,2$, $z_i\in D$ and $h_i\in\mc{A}(z_i)$
such that
$f\lhd h_1\lhd h \lhd h_2\lhd g$, with $f,h\notin\mc{F}(z_1)$, $h,g\notin\mc{F}(z_2)$ and $\sigma_{z_1},\sigma_{z_2}<t$.
The same argument as above shows that $h_1,h_2\in\mc{F}(z)$.
It is clear that $h_1|_{z_1}$ and $h_2|_{z_1}$
are both elements of $\mc{W}$ and both pass through $z$.
To complete the proof, we will show that they are not comparable under $\sw$.

Suppose that $h_1|_{z_1}\sw h_2|_{z_2}$.
Then $\sigma_{z_2}\leq \sigma_{z_1}$, 
and the fact that $h_1\lhd h\lhd h_2$ implies that $h\in\mc{A}(z_1)$, which is a contradiction.
Similarly we cannot have $h_2|_{z_2}\sw h_1|_{z_1}$,
so in this case we also have $z\in\ram(\mc{W})$.
\end{itemize}
This completes the proof of Lemma \ref{l:ramific_flow_web_same}.
\end{proof}


We are now ready to give the proof of Theorem \ref{t:weaves_characterization}.
We prove the two parts of the theorem in turn,
with part 2 first.
Let $\mc{A}$ and $\mc{B}$ be weaves.

\textbf{Part 1.}
Suppose that $\mc{A}\sim\mc{B}$.
Theorem \ref{t:weave_structure} gives that $\flow(\mc{A})\eqd \flow(\mc{B})$,
which implies that there exists a coupling of $\mc{A}$ and $\mc{B}$ such that $\P[\flow(\mc{A})=\flow(\mc{B})]=1$.
Let us write $\mc{F}=\flow(\mc{A})\eqas\flow(\mc{B})$.
Suppose that $\vec{z}\sw\R^2_c$ is finite and almost surely non-ramified in both $\mc{A}$ and $\mc{B}$.
Lemma \ref{l:ramific_flow_web_same} gives that, almost surely,
$\vec{z}$ is non-ramified in $\mc{F}$.

We seek to show that $\mc{A}|_{\vec{z}}\eqas\mc{F}|_{\vec{z}}$.
Write $\vec{z}=(z_1,\ldots,z_m)$ and fix $i\leq m$.
Let $f=\mc{A}|_{z_i}$.
By Theorem \ref{t:path_extension} there exists $f'\in\mc{F}$ with $f\sw f'$.
Therefore $f'|_{z_i}=f$.
As $z$ is almost surely non-ramified in $\mc{F}$, and $\mc{F}\sw\Pi^\updownarrow$,
in fact $\mc{F}(z_i)\eqas\{f'\}$, which implies that almost surely $f'|_{z_i}=\mc{F}|_{z_i}=f$.
We thus have $\mc{A}|_{z_i}\eqas\mc{F}|_{z_i}$.

A symmetric argument shows that $\mc{B}|_{z_i}\eqas\mc{F}|_{z_i}$,
so in fact $\mc{A}|_{z_i}\eqas\mc{B}|_{z_i}$.
In particular $\mc{A}|_{z_i}$ and $\mc{B}|_{z_i}$ have the same marginal distribution,
so $\mc{A}|_{z_i}\eqd \mc{B}|_{z_i}$, as required.

\textbf{Part 2.}
Let us first show that (a) and (b) are equivalent.
Assume (a), that $\mc{A}\sim\mc{B}$.
Theorem \ref{t:weave_structure} gives that $\flow(\mc{A})\eqd \flow(\mc{B})$,
which implies that there exists a coupling of $\mc{A}$ and $\mc{B}$ such that $\P[\flow(\mc{A})=\flow(\mc{B})]=1$.
Let us write $\mc{F}=\flow(\mc{A})\eqas\flow(\mc{B})$.
From \eqref{eq:flow_op} we thus have
$\P[\mc{A}\cup\mc{F}\text{ is non-crossing and }\mc{B}\cup\mc{F}\text{ is non-crossing}]=1$.
Lemma \ref{l:noncr_transitive_weave} gives that $\P[\mc{A}\cup\mc{B}\text{ is non-crossing}]=1$, obtaining (b).

Conversely suppose (b), that $\mc{A},\mc{B}$ are coupled weaves such that $\P[\mc{A}\cup\mc{B}\text{ is non-crossing}]=1$.
By Lemma \ref{l:noncr_same_flow} we have
$\P[\flow(\mc{A})=\flow(\mc{B})]=1$.
In particular $\flow(\mc{A})$ and $\flow(\mc{B})$ have the same marginal distribution,
so $\mc{A}\sim\mc{B}$.

We will next show that (a) and (c) are equivalent.
Assume (a), that $\mc{A}\sim\mc{B}$.
Lemma \ref{l:ramification_meas_zero_det} implies
the existence of a deterministic dense countable $D\sw\R^2$
such that $D$ is almost surely non-ramified in both $\mc{A}$ and $\mc{B}$.
If $\vec{z}\sw D$ is finite then it follows by part 1 of the present proof
that $\mc{A}|_{\vec{z}}$ and $\mc{B}|_{\vec{z}}$ have the same marginal distribution,
which establishes (c).

Conversely suppose (c).
Note that the (deterministic) operation $\mc{A'}\mapsto\web_D(\mc{A'})$
depends only on $\mc{A'}|_D$.
Enumerate $D=(z_i)_{i\in\N}$ and write $\vec{z}_m=(z_1,\ldots,z_m)$.
From (c) we have that $\mc{A}|_{\vec{z}_m}$ and $\mc{B}|_{\vec{z}_m}$
have the same marginal distribution, 
so by countability of $D$ in fact $\mc{A}|_D$ and $\mc{B}|_D$ have the same marginal distribution.
It follows that $\web_D(\mc{A})$ and $\web_D(\mc{B})$ have the same marginal distribution.
By Lemma \ref{l:web_op_D} we thus have that $\web(\mc{A})\eqd \web(\mc{B})$,
so $\mc{A}\sim\mc{B}$,
which establishes (a).
This completes the proof.

\subsection{Proof of Theorems \ref{t:flow_conv} and \ref{t:weave_conv}}
\label{sec:proof_t_conv}

We give the proof of \textbf{Theorem \ref{t:weave_conv}} before that of \ref{t:flow_conv},
because part 2 of Theorem \ref{t:flow_conv} will be proven as a 
specialization of part 2 of Theorem \ref{t:weave_conv}.
Suppose that $\mc{A}_n,\mc{A}$ are weaves.
We prove parts 1 and 2 of Theorem \ref{t:weave_conv} in turn.

\textbf{Part 1.}
Suppose that $\mc{A}_n\to\mc{A}$
and let $\vec{z}_n,\vec{z}\in(\Rc)^m$ be non-ramified.
By Skorohod's Representation Theorem
we may (change probability space, preserving the marginal distributions of each $\mc{A}_n$ and $\mc{A}$)
and assume that $\mc{A}_n\toas  \mc{A}$.
Let us write 
$\vec{z}_n=(z_{n,1},\ldots,z_{n,m})$ and
$\vec{z}=(z_1,\ldots,z_m)$.
Due to non-ramification, for each $i$ the sets $\mc{A}_n|_{z_{n,i}}$ and $\mc{A}|_{z_i}$
almost surely contain a single bi-infinite path,
which we write as $f_{n,i}$ and (respectively) $f_i$.
By Lemma \ref{l:relcom_Pi_KPi} the set $\bigcup_{n\in\N}\mc{A}_n$
is almost surely relatively compact.
Therefore we may pass to a subsequence and assume that
$f_{n,i}\toas   g_i\in\Pi^\updownarrow$ as $n\to\infty$, for all $i$.
Since $\mc{A}_n\toas  \mc{A}$ we have $g_i\in\mc{A}$.
By Lemma \ref{l:appdx_1_sw_limits} we have $g_i\in\mc{A}(z_i)$,
which implies that $g_i\stackrel{a.s.}{=}f_i$.
Thus $\mc{A}_n|_{\vec{z}_n}\toas  \mc{A}|_{\vec{z}}$,
on the probability space generated by Skorohod's Representation Theorem,
which implies convergence in distribution.

\textbf{Part 2.}
Let $\mc{B}$ be a weak limit point of $(\mc{A}_n)$, that is
$\mc{A}_n\tod \mc{B}$ along a subsequence of $n$.
Let us pass to this subsequence, without loss of generality.
Further, suppose that $\mc{B}$ is almost surely non-crossing.
By Skorohod's Representation Theorem,
noting that we are interested to prove distributional properties of $\mc{B}$,
without loss of generality we may assume that $\mc{A}_n\toas  \mc{B}$.
Since $\mc{B}$ is assumed to be almost surely non-crossing,
to show that $\mc{B}$ is a weave we need only show that
$\mc{B}$ is almost surely pervasive.
Let $z\in\R_c$.
Almost surely, for all $n$ there exists $f_n\in\mc{A}_n(z_n)$.
By Lemma \ref{l:relcom_Pi_KPi} the set $\mc{B}\cup\l(\bigcup_{n=1}^\infty\mc{A}_n\r)$
is compact, hence there exists $f\in\Pi^\uparrow$ such that $f_n\to f$.
As  $\mc{A}_n\toas  \mc{B}$ we have $f\in\mc{B}$.
By Lemma \ref{l:appdx_1_sw_limits} we have $f\in\mc{B}(z)$.
Thus $\mc{B}$ is pervasive,
so $\mc{B}$ is a weave.

Suppose additionally that 
$\mc{A}_n|_{\vec{z}}\to\mc{A}|_{\vec{z}}$ for all almost surely non-ramified $\vec{z}\in(\R^2)^m$.
By Lemma \ref{l:ramification_meas_zero_det} the set 
$R=\{z\in\R^2\-
\P\l[z\text{ is ramified in $\mc{A},\mc{B}$ or $\mc{A}_n$}\r]>0\}$
is a Lebesgue null subset of $\R^2$.
For any finite sequence $\vec{z}$ of points in $D=\R^2\sc R$
we have (by assumption) that
$\mc{A}_n|_{\vec{z}}\tod \mc{A}|_{\vec{z}}$.
Since $R$ is null, $D$ is dense in $\R^2$.
From part 1 of the present theorem we have also that
$\mc{A}_n|_{\vec{z}}\stackrel{a.s}{\to}\mc{B}|_{\vec{z}}$,
which implies convergence in distribution.
Hence 
$\mc{A}|_{\vec{z}}\eqd \mc{B}|_{\vec{z}}$
for all finite $\vec{z}\sw D$.
Theorem \ref{t:weaves_characterization} now gives that $\mc{A}\sim\mc{B}$.
This completes the proof of Theorem \ref{t:weave_conv}.

\medskip

We now give the proof of \textbf{Theorem \ref{t:flow_conv}},
proving each of the three parts in turn.
Suppose that $\mc{F}_n,\mc{F}$ are flows.

\textbf{Part 1.}
Note that if $z\in\R_c$ is non-ramified
then $\mc{F}(z)$ contains only a single bi-infinite path.
With this fact in hand, the argument is essentially the same as that of part 1 of Theorem \ref{t:weave_conv}
(from the start of the present section)
and is left to the reader.

\textbf{Part 2.}
Suppose that $\mc{F}'$
is a weak limit point of $(\mc{F}_n)$.
Lemma \ref{l:lhd_compat_biinf} gives that $\mc{F}'$ is almost surely non-crossing.
Hence, by part 2 of Theorem \ref{t:weave_conv}, $\mc{F}'$ is a weave.
As $\Pi^\updownarrow$ is a closed subset of $\Pi$ we have $\mc{F}'\sw\Pi^\updownarrow$ almost surely,
hence by Theorem \ref{t:weave_structure} $\mc{F}'$ is a flow.

Suppose, additionally, that $(\mc{F}_n)$ is tight and $\mc{F}_n|_{\vec{z}}\tod \mc{F}|_{\vec{z}}$
for all non-ramified $\vec{z}\in(\R^2)^m$.
Part 2 of Theorem \ref{t:weave_conv} thus gives that $\mc{F}\sim\mc{F}'$.
By Theorem \ref{t:weave_structure}, 
in particular by the fact that each equivalence class contains a unique maximal element,
we have $\mc{F}\eqd \mc{F}'$.
We now have that $(\mc{F}_n)$ is tight and any weak limit point of $(\mc{F}_n)$ is equal in distribution to $\mc{F}$,
so we have that $\mc{F}_n\to\mc{F}$.

\textbf{Part 3.}
Suppose that $\mc{F}_n\tod \mc{F}$
and that $\mc{A}_n\sim\mc{F}_n$.
By Theorem \ref{t:weaves_characterization}, for each $n\in\N$ 
there exists a coupling of $\mc{F}_n$ to $\mc{A}_n$
such that $\mc{A}_n\cup\mc{F}_n$ is almost surely non-crossing.
Hence $\mc{A}_n\sw(\mc{F}_n)_\uparrow$.
It follows from Proposition \ref{p:relcom_tightness}
that tightness of $(\mc{F}_n)$ implies tightness of $(\mc{A}_n)$.

Let $\mc{A}'$ be a weak limit point of $(\mc{A}_n)$.
We thus have $(\mc{F}_n,\mc{A}_n)\tod (\mc{F},\mc{A}')$.
We 
apply Skorohod's Representation Theorem 
(to the sequence of pairs $(\mc{F}_n,\mc{A}_n)_{n\in\N}$)
and may therefore assume without loss of generality that
$\mc{F}_n\toas  \mc{F}$ and $\mc{A}_n\toas\mc{A}'$.
Note that this preserves the marginal distributions of $(\mc{F}_n,\mc{A}_n)$, for each $n$.
We therefore have that $\mc{A}_n\sw(\mc{F}_n)_\uparrow$ almost surely.

If $f\in\mc{A}$ then (almost surely) there exists $f_n\in\mc{A}_n$ such that
$f_n\to f$.
Hence also there exists $g_n\in\mc{F}_n$ with $f_n\sw g_n$.
Lemma \ref{l:relcom_Pi_KPi} gives that $\mc{F}\cup(\bigcup_{n\in\N}\mc{F}_n)$ is almost surely compact,
so we may pass to a subsequence and assume that $g_n\to g$, where $g\in\mc{F}$.
Lemma \ref{l:appdx_1_sw_limits} gives that $f\sw g$.
Thus almost surely $\mc{A}'\sw\mc{F}_\uparrow$, which implies $\mc{A}'\cup\mc{F}$ is non-crossing.
Lemma \ref{l:appdx_1_sw_limits} and almost sure pervasiveness of $\mc{A}_n$
imply that $\mc{A}'$ is almost surely pervasive.
Thus $\mc{A}'$ is a weave.
By Theorem \ref{t:weaves_characterization} we have that $\mc{A}'\sim\mc{F}$, as required.
This completes the proof of Theorem \ref{t:flow_conv}.

\subsection{Proof of Theorem \ref{t:dual_webs}}
\label{sec:proof_t_dual_webs}

Let $\mc{A}$ be a weave, let $\mc{W}=\web(\mc{A})$ and $\mc{F}=\flow(\mc{A})$.
Let $D$ be a countable, dense and almost surely non-ramified subset of $\Rc$.
By Theorem \ref{t:weave_structure} the event
$$\{
\mc{W}=\web_D(\mc{A}), \;
\mc{W}\preceq\mc{A}\preceq\mc{F}, \;
\mc{A}\cup\mc{W}\cup\mc{F}\text{ is noncrossing}, \;
D\text{ is non-ramified}\} $$
has probability one.
Without loss of generality,
for the remainder of Section \ref{sec:proof_t_dual_webs}
we condition on this event occurring.
During the course of this proof we will apply several of our previous results in reverse time,
to $\Pi^\downarrow$ values objects rather than $\Pi^\uparrow$ valued objects.
To assist with this we will use the $\cdot^\rot$ operator, defined above the statement of Theorem \ref{t:dual_webs},
which represents rotation of space-time $\Rc$ about the origin by 180 degrees.
In previous sections, for $f\in\Pi^\uparrow$ with $(x,t)\in H(f)$ we have written $f|_{(x,t)}$ for the restriction of $f$ to $[t-,\infty+]$.
Here, we need to extend this notation to allow for restriction both forwards and backwards in time.
For $f\in\Pi$ we will write 
$f\rfloor_{(x,t)}$ for the restriction of $f$ to $[\s_f-,t+]$
and
$f\lceil_{(x,t)}$ for the restriction of $f$ to $[t-,\t_f+]$.
For $f\in\Pi^\uparrow$ we will avoid the notation $f|_z$ within this section, writing $f\lceil_z$ instead.

We begin the proof of Theorem \ref{t:dual_webs} by showing that
\begin{equation}
\label{eq:WW_pre}
\{g\in\Pi^\downarrow\-g\text{ does not cross }\mc{A}\text{ and }g\text{ begins in }D\}
=
\{f\rfloor_z\in\Pi^\downarrow\-z\in D\text{ and }f\in\mc{F}(z)\}
\end{equation}
To see \eqref{eq:WW_pre},
first consider if $g$ does not cross $\mc{A}$ and begins in $D$.
By applying Theorem \ref{t:path_extension} (in reverse time) we obtain 
$f\in\Pi^\updownarrow$ such that $g\sw f$ and $f$ does not cross $\mc{A}$,
which implies that $f\in\flow(\mc{A})=\mc{F}$. 
We have $g=f\rfloor_z$, 
so the left hand side of \eqref{eq:WW_pre} is contained within the right hand side.
For the reverse inclusion, consider if $z\in D$ and $f\in\mc{F}(z)$.
Clearly $g=f\rfloor_z$ begins at $z\in D$ and does not cross $\mc{F}$,
which implies $g$ does not cross $\mc{A}$.
We have thus established \eqref{eq:WW_pre}.

Let
\begin{align}
\widehat{\mc{W}}_D
&=\ov{\{g\in\Pi^\downarrow\-g\text{ does not cross }\mc{A}\text{ and }g\text{ begins in }D\}_\downarrow}
\label{eq:What_def_repeat} \\
\widehat{\mc{W}}'_D
&=\ov{\{f\rfloor_z\in\Pi^\downarrow\-z\in D\text{ and }f\in\mc{F}(z)\}_\downarrow}
\notag
\end{align}
Note that \eqref{eq:What_def_repeat} is \eqref{eq:What_def}, repeated here for convenience.
By Proposition \ref{p:relcom_tightness},
compactness of $\mc{F}$ implies relative compactness of (the right hand side of) equation \eqref{eq:WW_pre}.
With this in hand Lemma \ref{l:relcomp_uparrow}, applied in reverse time, gives that
$\widehat{\mc{W}}_D=\widehat{\mc{W}}'_D$.
From \eqref{eq:web_op} we have that
\begin{equation}
\label{eq:What_is_dual_web}
\wh{\mc{W}}'_D=(\web_{D^\rot}(\mc{F}^\rot))^\rot
\end{equation}
from which Theorem \ref{t:weave_structure} gives that $\wh{\mc{W}}'_D$ is both a dual web
and (from Remark \ref{r:web_op_det}) does not depend on the choice of 
dense and almost surely non-ramified subset $D\sw\Rc$.
From this point on let us write $\wh{\mc{W}}=\wh{\mc{W}}_D=\wh{\mc{W}}'_D$.
From \eqref{eq:What_is_dual_web} we have that
$(\wh{\mc{W}}'_D)^\rot=\web_{D^\rot}(\mc{F}^\rot)$,
which in words says that $(\wh{\mc{W}}'_D)^\rot$ is the web associated to $\mc{F}^\rot$.
Theorem \ref{t:weave_structure} implies that $(\wh{\mc{W}}'_D)^\rot$ does not cross $\mc{F}^\rot$,
thus $\wh{\mc{W}}'_D$ does not cross $\mc{F}$,
which by Lemma \ref{l:noncr_transitive_weave} implies it also does not cross $\mc{A}$.
Therefore $(\mc{W},\wh{\mc{W}})$ is a double web.

The same argument that led to \eqref{eq:WW_pre}, but now used forwards in time, gives that
$
\{f\in\Pi^\uparrow\-f\text{ does not cross }\mc{A}\text{ and }f\text{ begins in }D\}
=
\{f\lceil_{z}\in D\text{ and }f\in\mc{F}(z)\}
$.
It follows from \eqref{eq:web_op} and Theorem \ref{t:weave_structure} that
\begin{equation}
\label{eq:web_op_other}
\mc{W}=\ov{\{f\in\Pi^\uparrow\-f\text{ does not cross }\mc{A}\text{ and }f\text{ begins in }D\}_\uparrow}.
\end{equation}

We now turn our attention to \eqref{eq:F_WWHat}.
Let
\begin{equation}
\label{eq:F_WWHat_repeat}
\mc{F}'
=\ov{\{g_{\hookrightarrow}f\in\Pi^\updownarrow\-g\in\wh{\mc{W}}\text{ ends and }f\in\mc{W}\text{ begins at the same point of }D\}}. 
\end{equation}
We must show that $\mc{F}'\stackrel{a.s}{=}\mc{F}$.
Let $h\in\mc{F}'$.
Then there exists $z_n\in D$ such that $f_n\in\mc{W}$ begins and $g_n\in\wh{\mc{W}}$ ends at $z_n$,
and $h_n\to h$ where $h_n=(g_n)_{\hookrightarrow}(f_n)$.
Using that $z_n$ is non-ramified, let $h'_n$ be the unique element of $\mc{F}(z_n)$.
Note that $f_n$ does not cross $\mc{F}$.
Hence $((h'_n)\rfloor_{z_n})_{\hookrightarrow}(f_n)$ is a path passing through $z_n$ that does not cross $\mc{F}$,
which is therefore an element of $\mc{F}$,
and therefore equal to $h_n$.
Thus $f_n=(h'_n)\lceil_{z_n}$.
By a symmetrical argument, $g_n=(h'_n)\rfloor_{z_n}$ which implies that $h'_n=(g_n)_{\hookrightarrow}(f_n)=h_n$.
Hence $h_n\in\mc{F}$, which implies that $h\in\mc{F}$.

To see the reverse inclusion, let $f\in\mc{F}$.
By Lemma \ref{l:flow_regeneration} we have $\mc{F}=\ov{\mc{F}(D)}$,
hence there exists $z_n\in D$ and $h_n\in\mc{F}(z_n)$ such that $h_n\to f$.
Let $f_n=(h_n)\lceil_{z_n}$ and $g_n=(h_n)\rfloor_{z_n}$.
We have that $h_n$ does not cross $\mc{F}$, hence $f_n$ and $g_n$ do not cross $\mc{A}$.
From \eqref{eq:What_def_repeat} and \eqref{eq:web_op_other} we have that 
$f_n\in\mc{W}$ and $g_n\in\wh{\mc{W}}$.
Hence $h\in\mc{F}'$. We thus have $\mc{F}=\mc{F}'$ which establishes \eqref{eq:F_WWHat}.

It remains only to show the uniqueness claim.
Let $\wh{\mc{U}}$ be a dual web and
suppose that $(\mc{W},\wh{\mc{U}})$ is a double web.
Then $\wh{\mc{U}}^\rot$ is a web that almost surely does not cross $\mc{F}^\rot$, 
and the same is true of $\wh{\mc{W}}^\rot$.
It is straightforward to check that $\mc{F}^\rot$ is a $\Pi^\updownarrow$ valued random variable that inherits closedness, pervasiveness and the non-crossing property from $\mc{F}$.
Proposition \ref{p:relcom_tightness} implies that relative compactness is also inherited through $\cdot^\rot$, 
so $\mc{F}^\rot$ is a weave.
Lemma \ref{l:noncr_two_weaves} now implies that $\wh{\mc{U}}^\rot\cup \wh{\mc{W}}^\rot$ is almost surely non-crossing,
from which Lemma \ref{l:noncr_same_web} implies that 
$\web(\wh{\mc{U}}^\rot)\eqas \web(\wh{\mc{W}}^\rot)$.
By Theorem \ref{t:weave_structure} we thus have $\wh{\mc{U}}^\rot\eqas \wh{\mc{W}}^\rot$,
which implies $\wh{\mc{U}}\eqas \wh{\mc{W}}$ as required.

\subsection{Proof of Theorem \ref{t:weaves_continuity}}
\label{sec:proof_t_cont_weaves}

Recall that $\Pi^\uparrow_c, \Pi^\downarrow_c$ and $\Pi^\updownarrow_c$ respectively denote
the sets of continuous forwards half-infinite, backwards half-infinite, and bi-infinite {\cadlag} paths. 
The proof of Theorem \ref{t:weaves_continuity} is based on the following lemma.

\begin{lemma}
\label{l:path_extn_continuous}
The following hold.
\begin{enumerate}
\item
Let $\mc{A}\sw\Pi^\uparrow_c$ be a deterministic weave. 
If $h\in\Pi^\updownarrow$ does not cross $\mc{A}$ then $h\in\Pi^\updownarrow_c$.
\item
Let $\mc{F}\sw\Pi^\updownarrow_c$ be a deterministic flow.
If $f\in\Pi^\uparrow$ does not cross $\mc{F}$ then $f\in\Pi^\uparrow_c$.
\end{enumerate}
\end{lemma}
\begin{proof}
We will prove each claim in turn, starting with the first.
We argue by contradiction.
Let $\mc{A}\sw\Pi_c$.
Suppose that
that $h\in\Pi^\updownarrow$ does not cross $f$ and
that $h$ is discontinuous at $t\in\R$.
Without loss of generality (or consider space reflected about the origin) 
we may assume that $h(t-)<h(t+)$.
By Lemma \ref{l:approx_before_jump} there exists $f\in\mc{A}$
such that $f(t-)\leq h(t-)$ and $h\lhd f$.
Hence $h(t+)\leq f(t+)$.
We thus have $f(t-)\leq h(t-)<h(t+)\leq f(t+)$, which means that $f\in\mc{A}$ is discontinuous at $t$.
This is a contradiction, which completes the proof.

It remains to establish the second claim.
If $f\in\Pi^\uparrow$ does not cross $\mc{F}$ then by Theorem \ref{t:path_extension} 
there exists $f'\in\Pi^\updownarrow$ such that $f\sw f'$ and $f'$ does not cross $\mc{F}$.
From what we have already proved we have $f'\in\mc{F}$, 
thus $f'\in\Pi_c$ which implies $f\in\Pi_c$.
\end{proof}

We now give the proof of Theorem \ref{t:weaves_continuity}.
Note that it suffices to prove the results for deterministic weaves.
We will prove the two claims in turn, starting with the first.
Let $\mc{A},\mc{B}$ be deterministic weaves such that $\mc{A}\sim\mc{B}$.
We seek to show that if $\mc{A}\sw\Pi^\uparrow_c$ then $\mc{B}\sw\Pi^\uparrow_c$.
It then follows by symmetry that $\mc{A}\sw\Pi^\uparrow_c$ if and only if $\mc{B}\sw\Pi^\uparrow_c$.

By Theorem \ref{t:weaves_characterization}, 
without loss of generality we may assume that
$\mc{A}$ are $\mc{B}$ coupled such that $\mc{A}\cup\mc{B}$ is almost surely non-crossing.
On that event, by Lemma \ref{l:noncr_two_weaves}, a path $f\in\Pi$ crosses $\mc{A}$ if and only if it crosses $\mc{B}$.
From \eqref{eq:flow_op} we thus obtain $\flow(\mc{A})\eqas \flow(\mc{B})$,
which we henceforth refer to as $\mc{F}$.
By part 1 of Lemma \ref{l:path_extn_continuous}, 
as $\mc{A}\sw\Pi^\uparrow_c$ we have also that $\mc{F}\sw\Pi^\updownarrow_c$.
Since $\mc{B}\preceq\mc{F}$ we have $\mc{B}\sw\mc{F}_{\uparrow}$, 
thus $\mc{B}\sw\Pi_c$.
The same applies by symmetry with the roles of $\mc{A}$ and $\mc{B}$ swapped.
This proves the first claim of Theorem \ref{t:weaves_continuity}.

It remains to prove the second claim.
Let $\mc{A}$ be a deterministic weave.
Let $\mc{W},\mc{F}$ denote the corresponding web and flow.
It remains to show that $\hat{\mc{W}}\sw\Pi^\uparrow_c$ if and only if $\mc{W}\sw\Pi^\downarrow_c$,
where $\hat{\mc{W}}$ is given by \eqref{eq:What_def} in Theorem \ref{t:dual_webs}.
By symmetry (or consider reversing the direction of time) it suffices to prove that 
if $\mc{W}\sw\Pi^\uparrow_c$ then $\hat{\mc{W}}\sw\Pi^\downarrow_c$.
To this end, suppose that $\mc{W}\sw\Pi_c$.

From what we have already proved, $\mc{F}\sw\Pi_c$.
From \eqref{eq:What_def} we have
\begin{equation}
\label{eq:What_from_F}
\hat{\mc{W}}
=\ov{\{g\in\Pi^\downarrow\-g\text{ does not cross }\mc{F}\text{ and }(g(\tau_g),\tau_g)\in D\}_\downarrow}.
\end{equation}
Let $\hat{f}\in\hat{\mc{W}}$.
Then there exists $\hat{f}_n\in\Pi^\downarrow$
such that $\hat{f}_n$ does not cross $\mc{F}$ and $\hat{f}_n\to\hat{f}$.
By Theorem \ref{t:path_extension} (applied in reverse time)
there exists $h_n\in\mc{F}$ such that $\hat{f}_n\sw h_n$.
By compactness of $\mc{F}$ we may pass to a convergent subsequence $h_n\to h\in\mc{F}$.
By Lemma \ref{l:appdx_1_sw_limits} we have $\hat{f}\sw h$.
Since $h$ does not cross $\mc{F}$, also $\hat{f}$ does not cross $\mc{F}$.
By part 2 of Lemma \ref{l:path_extn_continuous} (applied in reverse time)
we thus have $\hat{f}\in\Pi^\downarrow_c$.
Thus $\mc{W}\sw\Pi^\downarrow_c$.

\appendix


\section{Appendices}

\subsection{On the Hausdorff metric}
\label{a:hausdroff_metric}

Let $(M,d_M)$ be a metric space
and let $\mc{K}(M)$ denote the set of compact subsets of $M$, 
including the empty set.
We write $\dist_M(x,A)=\inf_{a\in A}d_M(x,a)$
for the infimum distance from the point $x\in M$ to $A\sw M$.
We now state some well known facts relating to $\mc{K}(M)$.
The function 
\begin{equation}
\label{eq:d_hausdorff}
d_{\mc{K}(M)}(A_1,A_2):=\sup_{x_1\in A_1}\dist_M(x_1,A_2)\vee\sup_{x_2\in A_2}\dist_M(x_2,A_1),
\end{equation}
defines a metric on $\mc{K}(M)$ known as the Hausdorff metric,
or more precisely the Hausdorff metric with respect to $d_M$.
If $d$ and $d'$ are two metrics generating the same topology on $M$, 
then their corresponding Hausdorff metrics generate the same topology on $M$.
This topology is known as the Hausdorff topology. 
Note the subtle difference
between `the Hausdorff topology', which refers to this particular topology,
and `a Hausdorff topology', which refers to any topology having the
Hausdorff property i.e.~that distinct points possess disjoint neighbourhoods. 
(All topological spaces mentioned within the present article have the Hausdorff property.)

Completeness of $M$ implies completeness of $\mc{K}(M)$.
The same extension from $M$ to $\mc{K}(M)$ also holds for separability, and for compactness.
We now establish some more detailed connections.

\begin{lemma}
\label{l:relcom_Pi_KPi}
Let $(M,d_M)$ be a complete metric space.
\begin{enumerate}
\item
Let $\mc{X}\sw\mc{K}(M)$.
Then $\mc{X}$ is relatively compact
if and only if $\cup_{X\in \mc{X}} X$ is a relatively compact subset of $M$.
\item
Let $X\sw M$.
Then $X$ is relatively compact 
if and only if $\{A\in\mc{K}(M)\-A\sw X\}$ is a relatively compact subset of $\mc{K}(M)$.
\end{enumerate}
\end{lemma}
\begin{proof}
The two claims are readily seen to be equivalent: take
$\mc{X}=\{A\in\mc{K}(M)\-A\sw X\}$
to see that the (1)$\ra$(2) and take
$X=\cup_{A\in\mc{X}} A$ 
to see that (2)$\ra$(1).
We will give proof of (1).
As completeness of $M$ implies completeness of $\mc{K}(M)$, 
in both $M$ and $\mc{K}(M)$ we have that relative compactness is equivalent to total boundedness.

Suppose that $\mc{X}\sw\mc{K}(M)$ is totally bounded.
Then, for each $\epsilon>0$ there is a finite set $X_1,\ldots,X_n$ of elements of $\mc{K}(M)$ such that, 
for any $X\in \mc{X}$ there is some $X_i$ such that $d_{\mc{K}(M)}(X,X_i)<\epsilon$. 
Let $Y=\bigcup_{i=1}^n X_i$ and note that $\bigcup_{X\in \mc{X}}X\sw Y^{(\epsilon)}$.
Since each $X_i$ is compact in $M$, $Y$ is also compact in $M$, and in particular $Y$ is totally bounded. Hence also $\bigcup_{X\in \mc{X}}X$ is totally bounded.

Conversely, suppose that $\mc{X}\sw\mc{K}(M)$ is such that $\bigcup_{X\in \mc{X}}X\sw M$ is totally bounded.
Let $\epsilon>0$.
There exists a finite set $\{x_1,\ldots,x_n\}\sw M$
and a map $f:\bigcup_{X\in \mc{X}}X\to \{x_1,\ldots,x_n\}$,  
such that for any $x\in \bigcup_{X\in \mc{X}}X$ we have $d_{M}(x,f(x))<\epsilon$. 
For any $X\in\mc{X}$, the set $f(X)=\{f(x)\-x\in X\}$ is finite and therefore compact,
meaning that $f(X)\in\mc{K}(M)$.
By construction we have $d_{\mc{K}(M)}(X,f(X))<\epsilon$.
Let $\mathscr{X}$ be the set of subsets of $\{x_1,\ldots,x_n\}$,
hence $\mathscr{X}$ is a finite subset of $\mc{K}(M)$.
We have shown that for any $X\in\mc{X}$ there is some $X'\in\mathscr{X}$ such that $d_{\mc{K}(M)}(X,X')<\epsilon$.
Thus $\mc{X}$ is totally bounded.
\end{proof}

The next Lemma is a key ingredient of the proof of Lemma \ref{l:meas_Az}.
It provides a supply of closed (and consequently measurable) subsets of $\mc{K}(M)$.
We define
\begin{align*}
A^{(\eps)}&=\{x\in M\- \dist_{M}(x,A) < \eps\} \qquad \text{ for }\epsilon>0, \\
A^{[\eps]}&=\{x\in M\- \dist_{M}(x,A)\leq \eps\} \qquad \text{ for }\epsilon\geq0, 
\end{align*}
as the (respectively) open and closed $\epsilon$-expansions of $A\sw M$.
Note that $A^{[0]}=A$.

\begin{lemma}
\label{l:meas_various_KPi}
Let $X,Y\sw M$, $\epsilon\geq 0$ and $\mathscr{C}\sw\mc{K}(M)$. Then:
\begin{enumerate}
\item the set $\{A\in\mc{K}(M)\- X\sw A^{[\epsilon]}\}$ is closed;
\item the set $\{A\in\mc{K}(M)\- A\cap (Y\sc X)=\emptyset\}$ is closed if $X$ is closed and $Y$ is open;
\item the set $\{A\in\mc{K}(M)\- A\cap (X\sc Y)\neq\emptyset\}$ is closed if $X$ is closed and $Y$ is open;
\item the set $\{A\in\mc{K}(M)\- \exists C\in \mathscr{C}, C\sw A\}$ is closed if $\mathscr{C}$ is closed.
\end{enumerate}
\end{lemma}
\begin{proof}
We prove the claims independently. 
For the first claim, 
assume that $A_n\to A$ in $\mc{K}(M)$ and $X\sw A_n^{[\epsilon]}$ for all $n$, where $\epsilon\geq 0$.
Let $f\in X$, so $d_{\mc{K}(M)}(\{f\},A_n)\leq\epsilon$.
Letting $n\to\infty$ we obtain $d_{\mc{K}(M)}(\{f\},A)\leq\epsilon$.
Since $f\in X$ was arbitrary $X\sw A^{[\epsilon]}$, as required.

For the second claim, 
assume that $A_n\to A$ in $\mc{K}(M)$ and $A_n\cap (Y\sc X)=\emptyset$ for all $n$, where $Y$ is open and $X$ is closed.
We must show that $A\cap (Y\sc X)$ is empty.
We will argue by contradiction.
Suppose there exists $f\in A\cap(Y\sc X)$,
which means that $f\in A\cap Y$ and $f\notin X$.
Since $A_n\to A$ there exists $f_n\in A_n$ such that $f_n\to f$.
For each $n$, noting that $A_n\cap (Y\sc X)$ is empty we have (a) $f_n\notin Y$ or (b) $f_n\in X$;
thus one of these two alternatives must hold for an infinite subsequence of $n\in\N$.
If (a) holds for infinitely many $n$ then along that subsequence we have $f_n\to f$ with $f_n\notin Y$ and $f\in Y$,
which contradicts the fact that $Y$ is open. 
If (b) holds for infinitely many $n$ then along that subsequence we have $f_n\to f$ with $f_n\in X$ and $f\notin X$,
which contradicts the fact that $X$ is closed.
We thus reach the desired contradiction.

For the third claim, 
suppose that $A_n\to A$ in $\mc{K}(M)$ and $A_n\cap(X\sc Y)\neq\emptyset$ for all $n$,
where $X\sw M$ is closed and $Y\sw M$ is open.
Take $f_n\in A_n\cap (X\sc Y)$,
so $f_n\in A_n\cap X$ and $f_n\in M\sc Y$.
By Lemma \ref{l:relcom_Pi_KPi} the set $\cup_n A_n$ is a relatively compact $M$
and contains the sequence $(f_n)$,
so we may pass to a convergence subsequence $f_n\to f$.
Since $A_n\to A$ we have $f\in A$.
We have $f_n\in X$ and $f_n\in M\sc Y$, and since both $X$ and $M\sc Y$ are closed
we thus have $f\in X$ and $f\in M\sc Y$.
Thus $f\in A\cap(X\sc Y)$, which completes the proof.

For the final claim, 
suppose that $A_n\to A$ in $\mc{K}(M)$ and that $C_n\in\mathscr{C}$ with $C_n\sw A_n$.
Using Lemma \ref{l:relcom_Pi_KPi} the set $A\cup(\cup_n A_n)$ is compact, which implies that $\cup_n C_n$ is relatively compact,
so we may pass to a subsequence and assume $C_n\to C\in\mc{K}(M)$. 
As $C_n\sw A_n$ we have $C\sw A$ and
since $\mathscr{C}$ is closed we have $C\in\mathscr{C}$.
\end{proof}

\begin{lemma}
\label{l:meas_cts_upgrade}
If $F:M\to M$ is continuous 
then the map from $\mc{K}(M)$ to itself given by
$A\mapsto\{F(f)\-f\in A\}$ is continuous.
\end{lemma}
\begin{proof}
Note that if $F$ is uniformly continuous then 
the conclusion is clear from the definition of the Hausdorff metric, see \eqref{eq:d_hausdorff}.
For general continuous $F$, note by Lemma \ref{l:relcom_Pi_KPi}
that if $A_n\to A$ in $\mc{K}(M)$ then the set $(\cup_n A_n)\cup A$ is compact and 
thus the restriction of $F$ to this set is uniformly continuous.
The result follows.
\end{proof}

\begin{lemma}
\label{l:sw_KM}
Let $A_n,B_n,A,B\in\mc{K}(M)$ with $A_n\sw B_n$ for all $n$.
If $A_n\to A$ and $B_n\to B$ then $A\sw B$.
\end{lemma}
\begin{proof}
Let $a\in A$.
Since $A_n\to A$ there exists $a_n\in A_n$ such that $a_n\to a$.
Thus $a_n\in B_n$.
Since $B_n\to B$ we thus have $a\in B$.
\end{proof}

\subsection{On the M1 topology of $\Pi$ and $\mc{K}(\Pi)$}
\label{a:M1}

The book of \cite{Whitt2002}
details relative compactness and weak convergence for real valued stochastic processes (i.e.~single {\cadlag} paths) 
in all four Skorohod topologies.
In \cite{FreemanSwart2023} we introduce a unified framework for these four topologies,
suitable for random sets of {\cadlag} paths.
We recall some properties of the M1 version of this framework here.

Compact subsets of $\Pi$ play a key role in our main results and, for this reason,
we require corresponding criteria for relative compactness (of subsets of $\Pi$)
and tightness (relating to convergence in law of $\mc{K}(\Pi)$ valued random variables).
Fix a metric $d_{\ov{\R}}$ generating the topology on $\ov{\R}$ and
for $A\sw\ov{\R}$ let us write $\dist_{\ov{\R}}(x,A)=\inf\{d_{\ov{\R}}(x,y)\- y\in A\}$.
Relative compactness for sets of continuous paths is often characterised using the modulus of continuity,
see for example Theorem 7.2 of \cite{Billingsley1995}.
The analogous object for the M1 topology is
\begin{align}
\label{eq:moduli_cadlag}
w_{T,\de}(f)=
\sup\Big\{\dist_{\ov{\R}}\big(f(t_2\star_2),[f(t_1\star_1),f(t_3\star_3)]\big) 
\-\;&t_1\star_1,t_2\star_2,t_3\star_3\in I(f)_\mfs, \notag \\
&-T\leq t_1<t_2<t_3\leq T,\ t_3-t_1<\de\Big\},
\end{align}
with the conventions that the supremum over the empty set is zero,
and $[a,b]=[a\wedge b,a\vee b]$.

\begin{prop}
\label{p:relcom_tightness}
The following hold:
\begin{enumerate}
\item
A subset $\Ai\sw\Pi$ is relatively compact 
if and only if 
for all
$0<T<\infty$
$$\lim_{\de\to 0}\sup_{f\in\Ai}w_{T,\de}(f)=0.$$

\item
A subset $\mathscr{A}\sw\mc{K}(\Pi)$ is relatively compact 
if and only if 
for all
$0<T<\infty$
$$\lim_{\de\to 0}\sup_{\Ai\in\mathscr{A}}\sup_{f\in\Ai}w_{T,\de}(f)=0.$$

\item
A sequence of $\Pi$ valued random variables $(f_n)$ 
is tight, in the sense that their laws comprise a relatively compact sequence of probability measures on $\Pi$, if and only if
for all
$0<T<\infty$ and 
$\eps>0$ we have
$$\lim_{\delta\to 0}\limsup_{n\to\infty}\P\l[w_{T,\de}(f_n)\geq\eps\r]=0.$$

\item
A sequence of $\mc{K}(\Pi)$ valued random variables $(\mc{A}_n)$ 
is tight, in the sense that their laws comprise a relatively compact sequence of probability measures on $\mc{K}(\Pi)$, if and only if
for all
$0<T<\infty$ and 
$\eps>0$ we have
\begin{equation}
\label{eq:KPi_tightness}
\lim_{\delta\to 0}\limsup_{n\to\infty}\P\l[\sup_{f\in\mc{A}_n}w_{T,\de}(f)\geq\eps\r]=0.
\end{equation}
\end{enumerate}
\end{prop}
\begin{proof}
Part 1 follows from the relative compactness criteria given in Theorem 3.7 of \cite{FreemanSwart2023}.
In the language of that theorem,
compact containment is automatic as $\ov{\R}$ is compact,
and the Skorohod-equiconinuity requirement thus becomes part 1 above.
Note that, by Lemma \ref{l:relcom_Pi_KPi}, parts 1 and 2 of Proposition \ref{p:relcom_tightness} are in fact equivalent to each other.
Note also that part 3 follows from part 4 by taking $\mc{A}_n=\{f_n\}$.
Therefore, to complete the proof of Proposition \ref{p:relcom_tightness} it suffices to deduce part 4 as a consequence of part 2.
(In fact, a similar argument deduces part 3 from part 1.)

Suppose first that $(\mc{A}_n)$ is tight.
That is, for each $\kappa>0$ there exists a compact set $\mathscr{B}\sw\mc{K}(\Pi)$ such that
$\liminf_{n\to\infty}\P[\mc{A}_n\in\mathscr{B}]\geq 1-\kappa$.
By part 2, for any $\eps>0$ and $T\in(0,\infty)$ there exists $\de_0>0$
such that for all $\de\in(0,\de_0)$ we have $\sup_{B\in\mathscr{B}}\sup_{f\in B} w_{T,\de}(f)\leq\eps$.
Thus $\liminf_{n\to\infty} \P[\sup_{f\in\mc{A}_n} w_{T,\de}(f)\leq\eps]\geq 1-\kappa$.

It remains to show the reverse implication.
Let $(\mc{A}_n)$ satisfy \eqref{eq:KPi_tightness} and let $\kappa>0$.
Note that $w_{T,\de}(f)$ is an increasing function of both $\de$ and $T$.
By \eqref{eq:KPi_tightness}, 
for each $k\in\N$
set $\eps=1/k$ and 
choose $\delta_k>0$ and $T_k\in(0,\infty)$ such that
\begin{equation}
\label{eq:KPi_tightness_appl}
\limsup_{n\to\infty}\P[\mc{A}_n\sc B_k\neq\emptyset]\leq \eps 2^{-k}
\end{equation}
where
$B_k=\{f\in \Pi\- w_{T,\de}(f)\leq k^{-1}\}$.
Note that $B_{k+1}\sw B_k$ and let $B=\cap_{k} B_k$.
Then $\lim_{\de\downarrow 0} \sup_{f\in B}w_{T,\de}(f)=0$,
so part 2 gives that $\mathscr{B}=\{X\in\mc{K}(\Pi)\-X\sw B\}$ is a compact subset of $\mc{K}(\Pi)$.
We have that
\begin{align*}
\textstyle 
\limsup_{n}\P[\mc{A}_n\sc B\neq\emptyset] &= 
\textstyle
\limsup_{n}\P[\cup_{k} (\mc{A}_n\sc B_k)\neq\emptyset] \\
&\leq 
\textstyle
\limsup_{n}\sum_k\P[\mc{A}_n\sc B_k\neq\emptyset] \\
&\leq 
\textstyle
\sum_k \limsup_{n}\P[\mc{A}_n\sc B_k\neq\emptyset] \\
&\leq \eps.
\end{align*}
In the above, the third line uses the reverse Fatou lemma and the final line uses \eqref{eq:KPi_tightness_appl}.
Noting that $\{\mc{A}_n\sc B=\emptyset\}=\{\mc{A}_n\sw B\}=\{\mc{A}_n\in\mathscr{B}\}$
we thus obtain $\limsup_{n} \P[\mc{A}_n\in\mathscr{B}]\geq 1-\eps$.
Thus $(\mc{A}_n)$ is tight,
which completes the proof.
\end{proof}

The following three lemmas are consequences of Propositions \ref{p:J1M1} and (part 1 of) \ref{p:relcom_tightness}.
They cover most of our interaction with the M1 topology within the present article.
Recall from Section \ref{sec:pi} that for $f\in\Pi$ there is 
a natural total order $\sqsubseteq$ on the interpolated graph $H(f)$.
We write the associated strict order relation as $\sqsubset$.
We slightly extend the terminology introduced in \eqref{eq:f|z}:
if $w,z\in H(f)$ with $w\sqsubseteq z$ in the total order on $H(f)$,
we write $f|_{[w,z]}$ for the unique $g\sw f$ such that $g\in\Pi$ begins at $w$ and ends at $z$.

\begin{lemma}
\label{l:appdx_1_sw_limits}
Let $f_n,f\in\Pi$.
Suppose that $f_n\to f$, and that $w_n,z_n\in H(f_n)$ with $z_n\to z$ and $w_n\to w$,
and suppose that $w_n\sqsubseteq z_n$ in the induced order from $H(f_n)$.
Then $w,z\in H(f)$, with $w\sqsubseteq z$ in the induced order from $H(f)$, and $f_n|_{[w_n,z_n]}\to f|_{[w,z]}$.

In particular, if $z_n\in H(f_n)$ with $f_n\to f$ and $z_n\to z$ then $z\in H(f)$.
\end{lemma}
\begin{proof}
The second claim follows immediately from the first, so we will prove the first.
Let $g_n=f_n|_{[w_n,z_n]}$ and $g=f_{[w,z]}$.
From the hypothesis of the lemma $(f_n)$ is a relatively compact sequence in $\mc{K}(\Pi)$.
Since $g_n\sw f_n$ it follows from part 1 of Proposition \ref{p:relcom_tightness} that $(g_n)$ is also relatively compact.
To establish the present lemma it therefore suffices to show that any limit point of $(g_n)$ is equal to $g$.

Let $g'$ be a limit point of $(g_n)$ and,
with slight abuse of notation, 
let us pass to a subsequence and assume that $g_n\to g'$,
in the M1 topology.
From Proposition \ref{p:J1M1} we thus have 
$d_{\mc{K}((\Rc)^2)}(H^{(2)}(g_n),H^{(2)}(g'))\to 0$.
As we noted in comments above Proposition \ref{p:J1M1},
it is trivial to see that this implies
$d_{\mc{K}(\Rc)}(H(g_n),H(g'))\to 0$.
Moreover $(h,h')\mapsto d_{\mc{K}(\Rc)}(H(h),H(h'))$ is the Hausdorff metric on $\Pi$,
generating a coarser topology than the M1 topology;
as discussed in Section 3.4 of \cite{FreemanSwart2023} the topology generated is Skorohod's M2 topology.
It follows immediately by uniqueness of limits that 
if 
$d_{\mc{K}(\Rc)}(H(g_n),H(g))\to 0$
then
$d_{\mc{K}((\Rc)^2)}(H^{(2)}(g_n),H^{(2)}(g))\to 0$.
Therefore, to establish the present lemma we need only show that $H(g')=H(g)$.

Consider $v\in H(g')$.
As $H(g_n)\to H(g')$ there exists $v_n\in H(g_n)$ such that $v_n\to v$.
Since $g_n\sw f_n$ we thus have $v_n\in H(f_n)$
with $w_n\sqsubseteq v_n\sqsubseteq z_n$.
Thus $(w_n,v_n),(v_n,z_n)\in H^{(2)}(f_n)$.
Since $H^{(2)}(f_n)\to H^{(2)}(f)$ we thus obtain $(w,v),(v,z)\in H^{(2)}(f)$,
which implies that $v\in H(f)$ with $w\sqsubseteq v\sqsubseteq z$.
Hence $v\in H(g)$.
We thus obtain $H(g')\sw H(g)$.

It remains to show the reverse inclusion.
Consider $v\in H(g)$.
Thus $v\in H(f)$ with $w\sqsubseteq v\sqsubseteq z$,
which means that $(w,v),(v,z)\in H^{(2)}(f)$.
Hence there exists $(w'_n,v_n),(v_n,z'_n)\in H^{(2)}(f_n)$
such that $(w'_n,v_n)\to (w,v)$ and $(v_n,z'_n)\to (v,z)$.
If $v=w$ then $w_n\to v$ and as $w_n\in H(g_n)$ we obtain $v\in H(g')$.
Similarly, if $v=z$ then $z\in H'(g)$.
Without loss of generality we may therefore assume that 
$w\sqsubset v\sqsubset z$.

Note that $w_n\vee w'_n$ and $z_n\wedge z'_n$, 
where $\wedge$ and $\vee$ respectively denote $\min$ and $\max$ under $\sqsubseteq$,
are elements of $H(f_n)$.
Moreover $w_n\vee w'_n\to w$ and $z_n\wedge z'_n\to z$.
If $v_n\sqsubseteq w_n\vee w'_n$ for infinitely many $n\in\N$ then
$(v_n,w_n\vee w'_n)\in H^{(2)}(f_n)$ for such $n$ and
we would have $(v,w)\in H^{(2)}(f_n)$,
meaning that $v\sqsubseteq w$, which is a contradiction.
Hence $w_n\vee w'_n\sqsubseteq v_n$ for all but finitely many $n$.
Similarly $v_n \sqsubseteq z_n\wedge z'_n$.
Thus $w_n\sqsubseteq v_n\sqsubseteq z_n$ for all but finitely many $n$.
For such $n$ we have $v_n\in H(g_n)$,
which implies that $v\in H(g)$.
We therefore obtain $H(g')\sw H(g)$,
so in fact $H(g')=H(g)$, 
which completes the proof.
\end{proof}

\begin{lemma}
\label{l:appdx_2_fntn}
Let $f_n,f\in\Pi$.
Suppose that $f_n\to f$ and $t_n\in I(f_n)$ with $t_n\to t\in\R$.
For each $n$ let $\star_n\in\{-,+\}$.
Then $f\in I(f)$ and
$f(t-)\wedge f(t+)\leq \liminf_{n} f_n(t_n\star_n)\leq \limsup_n f_n(t_n\star_n)\leq f(t-)\vee f(t+).$
\end{lemma}
\begin{proof}
Let $x$ be any subsequential limit point of $(f_n(t_n\star_n))_{n\in\N}$,
which means that $(x,t)$ is a subsequential limit point of $(f_n(t_n\star_n),t_n)$.
Note that $(f_n(t_n\star_n),t_n)\in H(f_n)$.
By Lemma \ref{l:appdx_1_sw_limits} we thus have $(x,t)\in H(f)$,
which implies that $x\in[f(t-)\wedge f(t+),f(t-)\vee f(t+)]$.
The result follows.
\end{proof}

\begin{lemma}
\label{l:appdx_3_time_approx}
Let $f_n,f\in\Pi$, with $f_n\to f$ and $t\star\in I(f)_\mfs$.
There exists $t_n\in\R$ such that $t_n \to t$ and $f_n(t_n\star)\to f(t\star)$.
\end{lemma}
\begin{proof}
Noting that $H(f_n)\to H(f)$, take $(x_n,t_n)\in H(f_n)$ such that $(x_n,t_n)\to (f(t\star),t)$.
Then 
$(f_n(t_n-),t_n)\sqsubseteq (x_n,t_n) \sqsubseteq (f_n(t_n+),t_n)$.
By compactness of $\ov{\R}$ we may, without loss of generality,
pass to a subsequence along which both $f_n(t_n-)$ and $f_n(t_n+)$ converge,
say $f_n(t_n-)\to a$ and $f_n(t_n+)\to b$.
Hence, by Lemma \ref{l:appdx_1_sw_limits} we have
$(a,t)\in H(f)$ and $(b,t)\in H(f)$ with 
$(f(t-),t)\sqsubseteq (a,t)\sqsubseteq (f(t\star),t) \sqsubseteq (b,t) \sqsubseteq (f(t+),t)$.
If $\star=-$ then $(a,t)=(f(t\star),t)$ and $f_n(t_n-)\to f(t-)$.
If $\star=+$ then $(b,t)=(f(t\star),t)$ and $f_n(t_n+)\to f(t+)$.
\end{proof}

\subsection{On measurability in $\mc{K}(\Pi)$}
\label{sec:meas_1}

In this section we establish that various basic maps involving $\mc{K}(\Pi)$ are measurable.
Such things are required to work with $\mc{K}(\Pi)$ valued random variables in Section \ref{sec:weaves_random}.
The vast majority of the work involved in this section is Lemma \ref{l:meas_Az},
which is also used in the proof of Lemma \ref{l:biinf_ramification_meas_zero}.


Recall that we use the Borel $\sigma$-fields on $\Rc$, $\Pi$ and $\mc{K}(\Pi)$.
These $\sigma$-fields are generated in each case by the closed (or equivalently, open) subsets. 
Due to our focus on compactness it is helpful to work with closed sets whenever possible.
Recall that $d_{\Pi}$, generating the M1 topology on $\Pi$, is defined via Proposition \ref{p:J1M1}
and the corresponding Hausdorff metric on $\mc{K}(\Pi)$ is defined via \eqref{eq:d_hausdorff}.

\begin{lemma}
\label{l:meas_Az}
The map $(A,z)\mapsto A(z)$ from $\mc{K}(\Pi)\times\Rc\to \mc{K}(\Pi)$ is measurable.
\end{lemma}
\begin{proof}
The proof is rather technical.
Recall that $A(z)=\{f\in A\-z\in H(f)\}$.
Note that Proposition \ref{p:J1M1} gives that $A(z)$ is a closed subset of $\Pi$,
which implies compactness since $A(z)$ is a subset of the compact set $A$.
As both $\mc{K}(\Pi)$ and $\Rc$ are separable,
in order to establish measurability of $(A,z)\mapsto A(z)$
it suffices to show that the marginal maps $A\mapsto A(z)$ and $z\mapsto A(z)$ are both measurable.
We split the proof into these two parts, which will be proved independently.
In each part we will show that the pre-image of a closed subset of $\mc{K}(\Pi)$ is measurable;
we will represent this pre-image explicitly using countably many set operations on measurable subsets.
Lemma \ref{l:meas_various_KPi} provides a supply of measurable (in fact, closed) subsets of $\mc{K}(\Pi)$.
The fact that a closed graph $H(f)$ is measurable (in fact, compact) provides a supply of measurable subsets of $\Rc$.

\textbf{Measurability of $A\mapsto A(z)$:}
Fix $z\in\Rc$ and denote this map by $M_z:\mc{K}(\Pi)\to\mc{K}(\Pi)$,
so $M_z(A)=A(z)$.
For $z\in\Rc$ let us write 
$\Pi_z=\{f\in\Pi\-z\in H(f)\}$
and 
$\mc{K}(\Pi_z)$ for the set of compact subsets of $\Pi_z$.
Proposition \ref{p:J1M1} implies that $\Pi_z$ is a closed subset of $\Pi$,
from which part 1 of Lemma \ref{l:meas_various_KPi} gives that $\mc{K}(\Pi_z)$ a closed subset of $\mc{K}(\Pi)$.
Note that $M_z$ maps into $\mc{K}(\Pi_z)$.
It therefore suffices to show that the pre-image of a closed subset of $\mc{K}(\Pi_z)$ is measurable.
To this end, 
let $\mathscr{C}\sw\mc{K}(\Pi_z)$ be closed and let $\mathscr{C}'$ be a dense countable subset of $\mathscr{C}$.
We will show that the following are equivalent:
\begin{enumerate}
\item $A(z)\in \mathscr{C}$;
\item for all $\epsilon>0$ there exists $\delta>0$ and $C\in \mathscr{C}'$ such that 
$\big(A\cap\Pi_z^{(\delta)}\big)\sc C^{[\epsilon]}=\emptyset$ and $C\sw A^{[\epsilon]}$.
\end{enumerate}
We first give the forwards implication (1)$\ra$(2).
Let us assume $A(z)\in \mathscr{C}$ 
and suppose that (2) fails,
in preparation for an argument by contradiction.
Then there exists $\epsilon>0$ such that for all $C\in\mathscr{C}'$ and all $\delta>0$ it holds that
$\big(A\cap\Pi_z^{(\delta)}\big)\sc C^{[\epsilon]}\neq\emptyset$ or $C\nsubseteq A^{[\epsilon]}$.
We have $A(z)\in\mathscr{C}$ so we may choose $C\in\mathscr{C}'$ such that
\begin{equation}
\label{eq:CAz_close}
d_{\mc{K}(\Pi)}(C,A(z))\leq\tfrac{\epsilon}{2}.
\end{equation}
Hence $C\sw A(z)^{[\epsilon]}$, which implies $C\sw A^{[\epsilon]}$.
Taking $\delta=\frac{1}{n}$ 
we thus have that 
there exists an infinite subsequence of $n\in\N$ for which
$\big(A\cap\Pi_z^{(1/n)}\big)\sc C^{[\epsilon]}\neq\emptyset$.
For such $n$, take $f_n\in \big(A\cap\Pi_z^{(1/n)}\big)\sc C^{[\epsilon]}$.
We thus have $f_n\notin C^{[\epsilon]}$,
which by \eqref{eq:CAz_close} implies that $f_n\notin A(z)^{[\epsilon/2]}$.
We also have that $f_n\in A\cap \Pi_z^{(1/n)}$
which, by compactness of $A$ implies that $(f_n)$ has a subsequential limit $f\in A(z)$.
However, this contradicts the conclusion reached in the previous sentence.
Thus (1)$\ra$(2).

Let us now establish the reverse implication (2)$\ra$(1).
Take $\epsilon=\frac1n$ and take $\delta=\delta_n>0$ and $C=C_n\in\mathscr{C}'$ as given from (2).
That is, we have 
\begin{equation}
\label{eq:CnAcomp}
C_n\sw A^{[1/n]}
\qquad\text{ and }\qquad
A\cap \Pi_z^{(\delta_n)}\sw C_n^{[1/n]}.
\end{equation}
From the first statement in \eqref{eq:CnAcomp}, 
since $C_n\in\mc{K}(\Pi_z)$ we have $C_n\sw A^{[1/n]}\cap \Pi_z\sw A(z)^{[1/n]}$.
It is automatic that
$A(z)\sw A\cap \Pi_z^{(\delta_n)}$
so from the second statement in \eqref{eq:CnAcomp} we obtain $A(z)\sw C_n^{[1/n]}$.
Putting these together, we obtain $d_{K(\Pi)}(A(z),C_n)\leq \frac1n$.
Thus $C_n\to A(z)$ as $n\to\infty$. Since $\mathscr{C}$ is closed we thus obtain $A(z)\in\mathscr{C}$.
Thus (2)$\ra$(1).

We now have that (1)$\iff$(2).
It follows that
\begin{align*}
M_z^{-1}(\mathscr{C})
&=
\bigcap_{\epsilon>0}\bigcup_{\delta>0}\bigcup_{C\in\mathscr{C}'}
\l\{B\in\mc{K}(\Pi)\-\l(B\cap \Pi_z^{(\delta)}\r)\sc C^{[\epsilon]}=\emptyset\r\} 
\cap
\l\{B\in\mc{K}(\Pi)\-C\sw B^{[\epsilon]} \r\}
\\
&=
\bigcap_{n\in\N}\bigcup_{m\in\N}\bigcup_{C\in\mathscr{C}'}
\l\{B\in\mc{K}(\Pi)\-B\cap \l(\Pi_z^{(1/m)}\sc C^{[1/n]}\r)=\emptyset\r\} 
\cap
\l\{B\in\mc{K}(\Pi)\-C\sw B^{[1/n]} \r\}
\end{align*}
Note that we have used the set algebraic identity $X\cap (Y\sc Z)=(X \cap Y)\sc Z$.
By parts 1 and 2 of Lemma \ref{l:meas_various_KPi} 
the last line of the above consists of countable unions and intersections of closed subsets of $\mc{K}(\Pi)$.
Thus $M_z$ is measurable.

\textbf{Measurability of $z\mapsto A(z)$:}
We will recycle some parts of our notation, to preserve the symmetry between this argument the above.
Fix $A\in\mc{K}(\Pi)$ and let us denote the map in question by $M_A:\Rc\to \mc{K}(\Pi)$,
so $M_A(z)=A(z)$.
For any $z\in\Rc$ we have that $M_A(z)\in\mc{K}(A)$, 
where $\mc{K}(A)$ denotes the set of all compact subsets of $A$.
Part 1 of Lemma \ref{l:meas_various_KPi} gives that
$\mc{K}(A)$ is a closed subset of $\mc{K}(\Pi)$.
It is straightforward to check that if $B\in\mc{K}(A)$ then
\begin{align}
\label{eq:MAB_preimage}
M_A^{-1}(B)
&=\l(\bigcap_{f\in B}H(f)\r)\sc\l(\bigcup_{f\in A\sc B}H(f)\r).
\end{align}
In words, equation \eqref{eq:MAB_preimage} says that
$M_A(z)=B$ if and only if $B$ is precisely the set of $f\in A$ such that $z\in H(f)$.
In order to establish measurability of $M_A$ it suffices to 
fix a closed subset $\mathscr{C}$ of $\mc{K}(A)$ and show that $M_A^{-1}(\mathscr{C})$ is a measurable subset of $\Rc$.
Let $\mathscr{C}'$ be a dense countable subset of $\mathscr{C}$.
We will show that the following are equivalent:
\begin{enumerate}[resume]
\item $A(z)\in\mathscr{C}$;
\item for all $\epsilon>0$ there exists $C\in\mathscr{C}'$ such that 
$z\in \l(\bigcap_{f\in C}H(f)^{[\epsilon]}\r)\sc\l(\bigcup_{f\in A\sc C^{[\epsilon]}}H(f)\r)$.
\end{enumerate}
We first give the forwards implication (3)$\ra$(4).
Suppose that $A(z)\in\mathscr{C}$, that is $M_A(z)\in\mathscr{C}$.
Then there exists $B\in\mathscr{C}$ such that $z\in M_A^{-1}(B)$.
Choose $C\in\mathscr{C}'$ such that $d_{\mc{K}(\Pi)}(B,C)\leq\epsilon$.
From \eqref{eq:MAB_preimage} we have that $z\in\bigcap_{f\in B}H(f)$.
It is straightforward to check that
$\bigcap_{f\in B}H(f)\sw\bigcap_{f\in B^{[\epsilon]}}H(f)^{[\epsilon]}$
and as $C\sw B^{[\epsilon]}$ we obtain that
$z\in \bigcap_{f\in C} H(f)^{[\epsilon]}$.
Similarly, we have $B\sw C^{[\epsilon]}$, so $A\sc C^{[\epsilon]}\sw A\sc B$ 
and from \eqref{eq:MAB_preimage} we have
$z\notin \bigcup_{f\in A\sc B}H(f) \supseteq \bigcup_{f\in A\sc C^{[\epsilon]}}H(f)$.
We have thus obtained that (3)$\ra$(4).

Let us now establish the reverse implication (4)$\ra$(3).
Take $\epsilon=\frac{1}{n}$ and let $C=C_n$ be as given from (4).
Noting that $C_n\sw A$ for all $n\in\N$, 
Proposition \ref{p:relcom_tightness} gives that the sequence 
$(C_n)_{n\in\N}$ is relatively compact (as sequence of elements of $\mc{K}(\Pi)$),
and thus has a convergence subsequence. 
With slight abuse of notation let us pass to this convergent subsequence and set
$B=\lim_n C_n$.
It follows immediately that $B\in \mathscr{C}$.

From (4) we have $z\in\bigcap_{f\in C_n}H(f)^{[\epsilon_n]}$.
For each $g\in B$ there exists $g_n\in C_n$ such that $g_n\to g$.
As $g_n\in C_n$ we have $z\in H(g_n)^{[\epsilon_n]}$.
Proposition \ref{p:J1M1} gives that $H(g_n)\to H(g)$ in $\mc{K}(\Rc)$,
which implies that $H(g_n)^{[\epsilon_n]}\to H(g)$.
It follows immediately that $z\in H(g)$,
and as $g\in B$ was arbitrary we have $z\in \bigcap_{g\in B} H(g)$.

Similarly, from (4) we have $z\notin \bigcup_{f\in A\sc C_n^{[\epsilon_n]}}H(f)$ for all $n$.
Take $g\in A\sc B$ and note that because $B\sw\Pi$ is closed we have $d_{\mc{K}(\Pi)}(\{g\},B)>0$.
As $C_n\to B$ we have also that $C_n^{[\epsilon_n]}\to B$, 
so for sufficiently large $n$ we have $g\notin C_n^{[\epsilon_n]}$.
Hence $z\notin H(g)$.
As $g\in B$ was arbitrary we thus have $g\notin\bigcup_{g\in A\sc B}H(g)$.
Putting this together with the conclusion of the previous paragraph and \eqref{eq:MAB_preimage}
we obtain that $z\in M_A^{-1}(B)$.
Thus (4)$\ra$(3).

We now have that (3)$\iff$(4).
It follows that
\begin{align}
M_A^{-1}(\mathscr{C})
&=\bigcap_{\epsilon>0}\bigcup_{C\in\mathscr{C}'}
\l(\l(\bigcap_{f\in C}H(f)^{[\epsilon]}\r)\sc\l(\bigcup_{f\in A\sc C^{[\epsilon]}}H(f)\r)\r) \notag\\
&=\bigcap_{n\in\N}\bigcup_{C\in\mathscr{C}'}
\l(\l(\bigcap_{f\in C}H(f)^{[1/n]}\r)\sc\l(\bigcup_{f\in A\sc C^{[1/n]}}H(f)\r)\r)
\label{eq:MA_preimage}
\end{align}
We now examine the two terms in brackets on the right hand side of \eqref{eq:MA_preimage}.
The set $\bigcap_{f\in C}H(f)^{[1/n]}$ is an intersection of closed sets, and is therefore closed.
Note also that
\begin{align}
\bigcup_{f\in A\sc C^{[1/n]}}H(f)
=\bigcup_{m\in\N}\l(\bigcup_{f\in A\sc C^{(1/n+1/m)}}H(f)\r).
\label{eq:MA_preimage_part}
\end{align}
We claim that $\bigcup_{f\in E} H(f)$ is closed whenever $E\sw \Pi$ is compact;
to see this take $x_n\in H(f_n)$ where $f_n\in E$ and $x_n\to x\in\Rc$, 
pass to a convergence subsequence $f_n\to f\in E$ and then by Proposition \ref{p:J1M1} we have $H(f_n)\to H(f)$ so $x\in H(f)$.
In particular, as $C^{(\delta)}$ is open for all $\delta>0$ we have that $A\sc C^{(\delta)}$ is compact, 
so the right hand side of \eqref{eq:MA_preimage_part} is a countable union of closed sets.
We have now shown that \eqref{eq:MA_preimage} represents $M_A^{-1}(\mathscr{C})$ using countably many set operations of measurable subsets of $\Rc$.
Thus $M_A$ is measurable.
\end{proof}

\begin{lemma}
\label{l:meas_A|z}
Let $z\in\Rc$.
The map $(A,z)\mapsto A|_{z}$ is a measurable map from $\mc{K}(\Pi)\times\Rc\to\mc{K}(\Pi)$.
\end{lemma}
\begin{proof}
Recall that $A|_z=\{f|_z\-f\in A(z)\}$.
As both $\mc{K}(\Pi)$ and $\Rc$ are separable,
in order to establish measurability of $(A,z)\mapsto A|_z$
it suffices to show that the marginal maps $A\mapsto A|_z$ and $z\mapsto A|_z$ are both measurable.
Recall $\Pi_z=\{f\in\Pi\-z\in H(f)\}$ from the proof of Lemma \ref{l:meas_Az},
in which we showed that $\Pi_z\sw\Pi$ and $\mc{K}(\Pi_z)\sw\mc{K}(\Pi)$ were both closed.
By Lemma \ref{l:meas_Az} the map $(A,z)\mapsto A(z)$ is measurable,
and it clear that it maps into $\mc{K}(\Pi_z)$.
Recall that for $f\in\Pi_z$, $f|_z$ is the unique $g\sw f$ such that $(g(\sigma_g-),\sigma_g)=z$,
and $A|_z=\{f|_z\-f\in A\cap\Pi_z\}$.
By Lemma \ref{l:appdx_1_sw_limits} the map $f\mapsto f|_z$ defined from $\Pi_z$ to itself is continuous.
It follows from Lemma \ref{l:meas_cts_upgrade} that $A\mapsto A|_z$ is continuous on $\mc{K}(\Pi_z)$.
Since $A|_z=(A(z))|_z$, this completes the proof.
\end{proof}

\begin{lemma}
\label{l:meas_Auparrow}
The map $A\mapsto A_{\uparrow}$ is a continuous map from $\mc{K}(\Pi)$ to itself.
\end{lemma}
\begin{proof}
Let $A_n,A\in\mc{K}(\Pi)$ with $A_n\to A$.
Then $(A_n)$ is relatively compact and, noting that $g\sw f\ra w_{T,\de}(f)\leq w{T,\de}(f)$,
part 2 of Proposition \ref{p:relcom_tightness} gives that $((A_n)_\uparrow)$ is relatively compact.
Let $B\in\mc{K}(\Pi)$ be a limit point of $((A_n)_\uparrow)$.
We must show that $B=A_\uparrow$.
Without loss of generality let us pass to a subsequence and assume that $(A_n)_\uparrow\to B$.

Let $g\in B$.
Then there exists $f_n\in A_n$ with $g_n\sw f_n$ and $g_n\to g$.
By Lemma \ref{l:relcom_Pi_KPi} the set $\cup_n A_n$ is relatively compact, 
so we may pass to a further subsequence and assume that $f_n\to f$.
Thus $f\in A$.
Lemma \ref{l:appdx_1_sw_limits} gives that $g\sw f$, 
so $g\in A_\uparrow$.
Hence $B\sw A_\uparrow$.

It remains to show the reverse inclusion.
Let $g\in A_\uparrow$.
Then there exists $f\in A$ with $g\sw f$.
As $A_n\to A$ there exists $f_n\in A_n$ such that $f_n\to f$.
From Proposition \ref{p:J1M1} (and the remarks just above it) we have $H(f_n)\to H(f)$.
In particular, there exists $z_n\in H(f_n)$ such that $z_n\to (g(\s_g-),\s_g)\in H(f)$.
Let $g_n=f_n|_{z_n}$.
Lemma \ref{l:appdx_1_sw_limits} gives that $g_n\to g$,
so $g\in B$.
Hence $A_\uparrow \sw B$.
Thus $A_\uparrow=B$ and the proof is complete.
\end{proof}

\subsection{Proof of Lemma \ref{l:bw_web}}
\label{a:bw_web}

We must show that the Brownian web $\Wb$ satisfies our definition of a web, in Definition \ref{d:weave}.
The argument rests on well known properties of the Brownian web.
We noted in Section \ref{sec:pi} that $(\Pi^\uparrow_{\rm c},d_\Pi)$
is the state space that is in common usage for the Brownian web.
As $\Wb$ is a $\mc{K}(\Pi^\uparrow_{\rm c})$ valued random variable,
it is also a $\mc{K}(\Pi^\uparrow)$ valued random variable.
We will now refer to points (a)-(c) of Theorem 2.3 of \cite{SchertzerSunEtAl2017},
which defines $\Wb$.

We first show that $\Wb$ is a weave.
Let $D\sw\Rc$ be dense and countable.
Point (a) of this definition gives that 
almost surely,
$\Wb(z)$ is non-empty at all $z\in D$,
which by Lemma \ref{l:appdx_1_sw_limits} implies that $\Wb$ is almost surely pervasive.
Moreover since $\Wb(z)$ is almost surely a singleton, $D$ is almost surely non-ramified.
It is well known in the literature that $\Wb$ is almost surely non-crossing;
strictly this follows because point (b) of the definition gives that $\Wb(D)$ is non-crossing,
point (c) gives $\Wb\eqas\ov{\Wb(D)}$,
and the non-crossing property is preserved by taking limits of continuous paths 
(this last implication uses Lemma \ref{l:noncr_cont_limits}).
We have now shown that $\Wb$ is a weave.

Point (c) of the definition gives that $\Wb\eqas\ov{\Wb(D)}$,
for any deterministic dense countable $D\sw\R^2$.
Noting our remarks immediately above the present lemma 
regarding $\Wb(D)$ and $\Wb|D$, we have $\Wb\eqas\ov{(\Wb|_D)}$.
Lemma \ref{l:appdx_1_sw_limits} implies that if $A_n\to A$ in $\mc{K}(\Pi)$
then $(\mc{A}_n)_\uparrow\to\mc{A}_\uparrow$.
It is straightforward to combine this fact with the usual system of 
random walk approximations to the Brownian web (e.g.~Figure 9 of \cite{SchertzerSunEtAl2017}) 
to show that 
$\Wb\eqas(\Wb)_\uparrow$.
Hence in particular $\Wb|_D\eqas(\Wb|_D)_\uparrow$, 
from which \eqref{eq:web_op} gives $\Wb\eqas\web(\Wb)$.
Theorem \ref{t:weave_structure} thus gives that $\Wb$ is a web.

\bibliographystyle{abbrvnat-nodoi}
\bibliography{biblio}

\end{document}